\documentclass[11pt]{article}
\usepackage[T1]{fontenc}
\usepackage{lmodern}
\usepackage[margin=1in]{geometry}
\usepackage{amsmath,amssymb,amsthm,mathrsfs,mathtools}
\usepackage{graphicx,xcolor,booktabs,array,longtable}
\usepackage[shortlabels]{enumitem}
\usepackage{tikz-cd}
\usepackage[all]{xy}
\usepackage{comment}
\usepackage{microtype}
\usepackage[hidelinks]{hyperref}
\usepackage[nameinlink,noabbrev]{cleveref}

\newtheorem{theorem}{Theorem}[section]
\newtheorem{lemma}[theorem]{Lemma}

\theoremstyle{definition}

\theoremstyle{remark}

\crefname{prop}{Proposition}{Propositions}
\Crefname{prop}{Proposition}{Propositions}
\crefname{claim}{Claim}{Claims}
\Crefname{claim}{Claim}{Claims}
\crefname{cor}{Corollary}{Corollaries}
\Crefname{cor}{Corollary}{Corollaries}

\newcommand{\OO}{\mathscr{O}}

\newcommand{\PP}{\mathbb{P}}

\DeclareMathOperator{\rk}{rk}

\DeclareMathOperator{\Pic}{Pic}
\DeclareMathOperator{\Gr}{Gr}

\newcommand{\casehead}[1]{\smallskip\noindent\textit{#1}\enspace}

\makeatletter
\newcommand{\oset}[3][0ex]{\mathrel{\mathop{#3}\limits^{
  \vbox to#1{\kern-2\ex@\hbox{\(\scriptstyle#2\)}\vss}}}}
\makeatother

\allowdisplaybreaks[2]
\numberwithin{equation}{section}

\theoremstyle{plain}
\newtheorem{proposition}[theorem]{Proposition}
\newtheorem{corollary}[theorem]{Corollary}
\theoremstyle{remark}
\newtheorem{remark}[theorem]{Remark}

\newcommand{\nf}[1]{\nolinkurl{#1}}

\hypersetup{pdftitle={Normal bundles of hyperelliptic curves and uniform propagation of semistability},pdfauthor={Xuyan Liu}}
\title{Normal bundles of hyperelliptic curves\\and uniform propagation of semistability}
\author{Xuyan Liu}
\date{}
\begin{document}
\maketitle
\begin{abstract}
Let $C$ be a smooth hyperelliptic curve of genus $g\ge2$ over an
algebraically closed field of characteristic zero. We prove that the
normal bundle of the embedding of $C$ in $\PP^n$, $n\ge3$, by the
complete linear series of a general line bundle of degree $n+g$ is
semistable if and only if $n\ge3g+1$. By openness, general complete
embeddings of general curves of genus $g$ also have semistable normal
bundle when $n\ge3g+1$. We then show that a general nonspecial curve
of genus two and degree $d$ in $\PP^n$ has semistable normal bundle
except when $(n,d)$ is $(3,5)$, $(4,6)$, $(5,7)$ or $(6,8)$. The
degenerations used in genus two extend to every genus $g\ge3$:
semistability in all nonspecial degrees and dimensions follows from
the cases $0\le d-n-g\le4g-4$. Together with a separate estimate for
large $n$, this reduces semistability in each fixed genus to finitely
many cases. The proofs combine degenerations to reducible nodal
curves, explicit bounds on the degrees of subsheaves, and, in genus
two, finitely many exact computations, which are provided as
ancillary files.
\end{abstract}
\setcounter{tocdepth}{1}
\makeatletter
\renewcommand{\l@section}{\@dottedtocline{1}{0em}{3.1em}}
\makeatother
\tableofcontents

\section{Introduction}
\label{sec:introduction}
Let $C\subset\PP^n$ be a smooth nondegenerate curve of genus $g$ and
degree $d$, with $n\ge3$. Its normal bundle $N_C=N_{C/\PP^n}$ has
rank $n-1$ and degree $(n+1)d+2g-2$. Recall that the \emph{slope} of
a vector bundle $E$ on $C$ is $\mu(E)=\deg E/\rk E$. The bundle $E$
is \emph{semistable} if $\mu(S)\le\mu(E)$ for every nonzero subsheaf
$S\subset E$, and \emph{stable} if $\mu(S)<\mu(E)$ whenever
$0<\rk S<\rk E$. Informally, semistability of $N_C$ says that no
family of normal directions along $C$ is more positive than the
average. This paper asks when $N_C$ is semistable, in two settings.

The first is a \emph{complete} embedding, given by all sections of a
very ample line bundle $L$ on $C$. The second is a \emph{general
nonspecial} embedding: a general curve of genus $g$, a general line
bundle $L$ of degree $d$ with $H^1(L)=0$, and a general
$(n+1)$-dimensional space of sections of $L$
(Section~\ref{sec:prelim}). Such an $L$ has $d-g+1$ sections, so
this requires $d\ge n+g$, with equality exactly when the embedding is
complete. We work over an algebraically closed field of
characteristic zero.

\subsection{Main results}

\paragraph{Hyperelliptic curves.}
Let $C$ be any smooth hyperelliptic curve of genus $g\ge2$, and let
$n\ge3$. A general line bundle $L$ of degree $n+g$ on $C$ embeds $C$
in $\PP^n$ by its complete linear series. The normal bundle of this
embedding is semistable if and only if $n\ge3g+1$
(Theorem~\ref{thm:hyperelliptic-iff}).

The instability for $n\le3g$ has a simple geometric source. The lines
spanned by the fibers of the hyperelliptic map $C\to\PP^1$ sweep out
a rational normal surface scroll $X\subset\PP^n$ containing $C$. The
normal bundle $N_{C/X}$ is a line subbundle of $N_C$ of degree
$4g+4$, and this degree exceeds $\mu(N_C)$ exactly when $n\le3g$.
The content of the theorem is that for $n\ge3g+1$ nothing else
destabilizes $N_C$. Since semistability is an open condition, it
follows that a general curve of genus $g$, embedded by the complete
linear series of a general line bundle of degree $n+g$, has semistable
normal bundle when $n\ge3g+1$ (Corollary~\ref{cor:general-complete}).
For a general curve we do not claim a converse.

\paragraph{Genus two.}
Every curve of genus two is hyperelliptic, and in genus two we obtain
a complete answer for nonspecial embeddings. For $n\ge3$ and
$d\ge n+2$, the normal bundle of a general nonspecial curve of genus
two and degree $d$ in $\PP^n$ is semistable if and only if
\[
 (n,d)\notin\{(3,5),(4,6),(5,7),(6,8)\},
\]
and it is stable if in addition $\gcd(n-1,2d+2)=1$
(Theorem~\ref{g2:thm:main}). The four exceptions are the complete
embeddings with $n\le6=3g$. They are unstable by the scroll argument
above, which is also how Coskun, Larson and Vogt explain their
instability \cite[Example 5.6]{CLVSurvey}. In genus two,
Theorem~\ref{g2:thm:main} therefore settles the semistability part of
the problem of classifying the Brill--Noether curves whose normal
bundles are not (semi)stable \cite[Problem 5.28]{CLVSurvey}.

\paragraph{Higher genus.}
For $g\ge3$ we reduce the problem to finitely many cases. Measure the
degree by the \emph{excess} $c=d-n-g\ge0$, which vanishes exactly for
complete embeddings. We prove two results about general nonspecial
curves of genus $g\ge3$.
\begin{enumerate}[(i)]
\item If the normal bundle is semistable whenever $0\le c\le4g-4$, in
every dimension $n\ge3$, then it is semistable for every $c\ge0$
(Theorem~\ref{uni:finite-width}).
\item The normal bundle is semistable whenever
$n\ge10g^2-10g+4+(5g-3)c$ (Theorem~\ref{high:affine}).
\end{enumerate}
Together these reduce semistability, in all degrees and dimensions, to
an explicit finite list of pairs $(n,d)$
(Theorem~\ref{thm:finite-verification}). For example, when $g=3$ it
suffices to treat $0\le c\le8$ and $n\le159$. This is a sufficient
criterion: it does not decide which of the finitely many remaining
cases are semistable (see Section~\ref{sec:finite-verification}).

\subsection{Obstacles}
It is convenient to twist $N_C$ by a line bundle, which does not
affect (semi)stability, so as to obtain a bundle $F$ of rank $r=n-1$
and degree $2k$, where $k=d-n+g$ (see \eqref{eq:notation}). Thus $F$
is semistable when every subsheaf of rank $t$ has degree at most
$2kt/r$. Two tools from the literature apply to such bundles. Each
stops short of what we need, and a third difficulty arises for
complete embeddings.

\paragraph{Interpolation.}
Interpolation for normal bundles \cite{ALY,LV} gives semistability
when the slope $2k/r$ is an integer, apart from a few exceptions
(Lemma~\ref{lem:integral}). For other slopes it bounds the degree of
a rank-$t$ subsheaf only by $\lceil 2k/r\rceil t$
\eqref{eq:interpolation-profile}, whereas semistability requires the
bound $\lfloor 2kt/r\rfloor$. To close this gap we keep track of
exact integer bounds for subsheaf degrees, rank by rank, rather than
of slopes alone.

\paragraph{Degeneration to nodal curves.}
Coskun, Larson and Vogt \cite{CLV} introduced a notion of stability
for vector bundles on nodal curves, and developed the strategy of
degenerating $C$ to a nodal curve and proving that the limit of its
normal bundle is semistable. In many of the cases we need, however,
the limits produced by such degenerations are unstable, so
semistability of the general normal bundle cannot be read off from
the limit alone. Our degenerations also have unstable limits, but in
a controlled way. The degree bounds force every destabilizing subsheaf
of the limit to be one explicit subbundle, whose determinant changes
as the degeneration varies. A destabilizing subsheaf of the general
member would have to specialize to that subbundle while keeping a
constant determinant. Hence the general member is semistable; this is
the moving determinant argument described below.

\paragraph{Complete embeddings.}
Complete embeddings occur in Theorem~\ref{thm:hyperelliptic-iff}, and
also as the quotients at the bottom of our inductions. For them there
is little room to degenerate: the linear series is complete, and in
Theorem~\ref{thm:hyperelliptic-iff} the curve itself is fixed.
Instead we work on the curve itself, using the hyperelliptic scroll
and the following idea. A candidate destabilizing subbundle determines
a map from $C$ to a Grassmannian. Degree computations show that, in
the extreme case, this map is governed by a fixed linear space in
$\PP^n$: all the normal directions of the candidate point toward it.
The geometry of the projection from that linear space then rules the
candidate out (see \emph{Fixed linear spaces} below).

\subsection{The main ideas}

\paragraph{Degree profiles.}
Most of our arguments bound, for each rank $t$, the degree of every
subsheaf of rank $t$. We call such a bound a \emph{degree profile}
(Section~\ref{sec:profiles}). Profiles behave well in two ways. First,
a profile for one member of an irreducible family of bundles is also
a profile for the general member. Second, if $0\to P\to M\to Q\to0$
is exact and $S\subset M$ has rank $t$, then $S$ meets $P$ in some
rank $u$ and maps to $Q$ with rank $t-u$. So the degree of $S$ is at
most the bound for $P$ in rank $u$ plus the bound for $Q$ in rank
$t-u$ \eqref{eq:convolution}. To prove that a general normal bundle is
semistable, it therefore suffices to specialize it to an extension of
two bundles whose profiles are good enough.

\paragraph{Pointing bundles and projection.}
Let $q$ be a point of $C$. For $x\in C$ different from $q$, the line
through $x$ and $q$ determines a normal direction to $C$ at $x$. These
directions span a line subbundle $N_{C\to q}\subset N_C$, the
\emph{pointing bundle} toward $q$; we use the pointing bundles and
elementary modifications of \cite{ALY}. Pointing bundles are the
kernels of projection. Projecting $C$ from the linear span of $b$
general points $q_1,\ldots,q_b\in C$ gives a curve of degree $d-b$ in
$\PP^{n-b}$, and $N_C$ is an extension of the normal bundle of this
curve, suitably twisted, by
$N_{C\to q_1}\oplus\cdots\oplus N_{C\to q_b}$. After the twist that
produces $F$, each pointing bundle becomes a line bundle of degree
zero.

On a hyperelliptic curve, the twisted pointing bundles toward the
$2g+2$ Weierstrass points are trivial. Projecting from all of them
exhibits $F$ as an extension of the corresponding bundle for a complete
embedding in $\PP^{n-2g-2}$ by a trivial bundle of rank $2g+2$
\eqref{high:projection}. This gives an induction on $n$ in steps of
$2g+2$ (Section~\ref{high:section}). The first ranks of the induction,
$3g\le r\le5g+1$, are treated directly. There, the scroll $X$ supplies
degree bounds for subsheaves, sections of $F$ supply others, and in
the last of these ranks we compute the connecting maps of the
projection sequence.

\paragraph{Fixed linear spaces.}
The bounds coming from sections rest on a geometric argument
(Lemmas~\ref{high:original-bounds} and~\ref{high:maximal-degree}).
Pointing bundles make sense for any linear subspace
$\Lambda\subset\PP^n$: the normal directions toward $\Lambda$ span a
subsheaf of $N_C$, the kernel of projection from $\Lambda$. Its degree
is determined by the base points and the ramification of that
projection \eqref{high:pointing-det}. Now let $S\subset F$ be a
candidate destabilizing subbundle of rank $s$. At each point $x\in C$,
the tangent line to $C$ and the normal directions in $S$ span a linear
subspace of $\PP^n$ of dimension $s+1$, so $S$ defines a map from $C$
to a Grassmannian. Counting the sections of $F/S$ by Riemann--Roch
bounds the degree of $S$. In the extreme case allowed by this count,
the computation shows that all these subspaces contain a fixed linear
subspace $\Lambda$ of dimension $s-1$. Then $S$ is exactly the bundle
of normal directions pointing toward $\Lambda$, and its degree is
governed by the projection from $\Lambda$. Adding to $\Lambda$ a
ramification point of this projection produces a subbundle of larger
rank whose degree is too large, which is a contradiction.

\paragraph{Pointing models.}
Projection lowers $n$ and $d$ by the same amount, so it does not
change the excess $c$. To lower the excess as well, we first
degenerate $C$, keeping the abstract curve and the line bundle $L$
fixed, to a reducible curve: a curve of degree $d-h$ in $\PP^n$
together with $h$ lines, each meeting it at one point $p_i$. We then
project from $b$ points $q_1,\ldots,q_b$ of $C$. The limit of the
normal bundles is an extension
(Proposition~\ref{g2:prop:models}, Lemma~\ref{g2:prop:rank-one} and
Theorem~\ref{thm:models})
\[
 0\longrightarrow P_h(b)\longrightarrow M\longrightarrow Q
 \longrightarrow0,
\]
where $M$ is a specialization of $F$, and:
\begin{itemize}
\item $Q$ is the twisted normal bundle, of rank $r-b$ and degree
$2(k-h)$, of $C$ embedded in $\PP^{n-b}$ in degree $d-b-h$ by a
general series;
\item the \emph{pointing kernel} $P_h(b)$, of rank $b$ and degree
$2h$, is obtained from the sum of the $b$ twisted pointing bundles
toward $q_1,\ldots,q_b$ by a modification at each $p_i$ in the
direction of the attached line; each modification raises the degree
by two.
\end{itemize}
We call such an extension a \emph{pointing model}. Its quotient is a
smaller instance of the same problem, so it can be controlled by
induction; the kernel is controlled by the profile described below.

\paragraph{Choosing the model: balanced and Farey splittings.}
The kernel and quotient of a pointing model have slopes $2h/b$ and
$2(k-h)/(r-b)$, and their average, weighted by rank, is the target
slope $2k/r$. If $\gcd(k,r)>1$, we choose $h/b=k/r$, so that both
pieces have the target slope. The quotient is semistable by induction
and the kernel is semistable, hence so is $M$, and with it the general
$F$. This is the \emph{balanced} case.

If $\gcd(k,r)=1$, equal slopes are impossible, and we take $h/b$ as
close to $k/r$ as possible. Concretely, we require
\[
 hr-kb=\pm1.
\]
The fractions $h/b$ with $0<b<r$ satisfying this condition are the
\emph{Farey neighbors} of $k/r$. There are exactly two of them, one on
each side of $k/r$. They are the fractions closest to $k/r$ among all
fractions with denominator less than $r$, and their numerators and
denominators add up to $k$ and $r$. So if $h/b$ is one neighbor, then
$(k-h)/(r-b)$ is the other. The slopes of kernel and quotient then lie
on opposite sides of $2k/r$, as close to it as integrality allows. The extension bound excludes every
destabilizing subsheaf except at most two explicit candidates
(Lemmas~\ref{g2:lem:normal-farey} and~\ref{lem:farey}). The first is
the piece of the extension whose slope exceeds $2k/r$: the kernel
itself, or a lift of the quotient. The second arises only when $r$
and $b$ are both odd, and has rank $(r\pm b)/2$. When $r$ is odd,
exactly one of the two neighbors has $b$ even, and using it avoids
the second candidate. When that neighbor cannot be used, we exclude
the candidate with a second, auxiliary pointing model and an integer
rounding estimate (Lemma~\ref{g2:prop:rounding}). If a neighbor has
$b=1$, that is, $k=jr\pm1$, the kernel is a line bundle. These cases
are treated first, by induction on $r$
(Lemmas~\ref{g2:prop:diagonals} and~\ref{uni:shifted-diagonals}).

\paragraph{Moving determinants.}
The first candidate is excluded by a simple principle
(Lemmas~\ref{g2:pk:determinant} and~\ref{lem:determinant}). Suppose
a family of bundles is parametrized by a rational variety. If its
general member were unstable, the determinant of the maximal
destabilizing subbundle would define a rational map from the parameter
space to a Picard variety of $C$, and such a map is constant. In our
models the determinants of the kernel and of the quotient change as
the points $q_j$ move. So a destabilizing subsheaf of the general
member cannot specialize to one of these pieces.

\paragraph{Profiles of pointing kernels.}
The argument above needs a precise bound for subsheaves of the kernel.
We prove that a general pointing kernel $P_h(b)$ is semistable, and
that each of its subsheaves of rank $t$ has degree at most an explicit
integer $\Phi_{h,b}(t)$ (Theorems~\ref{g2:pk:theorem}
and~\ref{thm:kernel-profile}; $\Phi$ is defined in
\eqref{g2:pk:profile}). This bound never exceeds the slope bound
$2ht/b$, and in many ranks it is one less than $\lfloor2ht/b\rfloor$.
These integer savings are used throughout the induction.

The proof is again a Farey argument, now applied to the kernel. If
$\gcd(h,b)=1$, write $h/b$ as the mediant $(h_1+h_2)/(b_1+b_2)$ of its
two Farey neighbors $h_1/b_1$ and $h_2/b_2$ among fractions with
smaller denominator. The kernel specializes to the direct sum
$P_{h_1}(b_1)\oplus P_{h_2}(b_2)$ of two smaller kernels. By
induction on $b$, every subsheaf of this sum obeys the bound
$\Phi_{h,b}$, with one exception: the summand of larger slope
(Lemma~\ref{g2:pk:farey-convolution}). That summand is removed by
moving determinants, since exchanging one of its points $q_j$ with a
point of the other summand changes its determinant. The induction
starts from kernels with a single modification. This base case is the
only part of the kernel profile whose proof depends on the genus
(Proposition~\ref{g2:pk:one-tail} for $g=2$, and
Section~\ref{sec:kernels} for $g\ge3$).

\paragraph{Starting cases.}
In genus two the complete case $c=0$ is
Theorem~\ref{thm:hyperelliptic-iff}. The excesses $1\le c\le4$
(Proposition~\ref{g2:ld:theorem}) are proved with pointing models
whose quotients have smaller excess. This step also uses sharper
profiles for complete normal bundles (Section~\ref{g2:cb:section}),
earlier results in low dimension and high degree \cite{CLV,CS},
nineteen types checked directly, and a finite integer computation,
which a periodicity lemma (Lemma~\ref{g2:ld:periodic}) extends to
all ranks. All larger excesses follow from the balanced and Farey induction
(Proposition~\ref{g2:prop:theorem}). The only complication is that the
quotient of a model may not be one of the four unstable complete
embeddings. Avoiding them requires three more types,
$(n,d)=(8,15),(10,18),(12,19)$, to be checked directly.

For these twenty-two types we use a cohomological criterion
(Lemma~\ref{g2:co:criterion}): if a bundle $G$ satisfies
$\chi(F\otimes G)=0$ and $H^0(F\otimes G)=0$, then $F$ is semistable.
We verify the vanishing for one explicit curve by exact linear algebra
over a finite field. Semicontinuity then gives it for the general
curve in characteristic zero (Section~\ref{g2:co:section}).

For $g\ge3$ the same induction applies, with results of \cite{CLV} in
$\PP^3$ and of \cite{CS} in $\PP^4$ in the smallest ranks. Suppose
semistability is known in the strip $0\le c\le4g-4$. This strip is
wide enough for the special cases of the genus-two argument to be
handled inside the induction, so semistability propagates to all
excesses (Theorem~\ref{uni:finite-width}). The large-dimension
estimate (Theorem~\ref{high:affine}) does not use this assumption.
For $0\le c\le2g$ it uses pointing models whose quotients are complete
embeddings. These are controlled by
Theorems~\ref{high:complete-semistability} and~\ref{high:complete-gap},
which extend from hyperelliptic curves to general curves by openness
(Corollary~\ref{high:general-curves}).
It then runs the same induction within the region
$n\ge10g^2-10g+4+(5g-3)c$, which the balanced and Farey splittings
preserve.

\subsection{Relation to earlier work}
Interpolation for normal bundles of general nonspecial curves was
proved by Atanasov--Larson--Yang \cite[Theorem 1.3]{ALY}, and for
general Brill--Noether curves by Larson--Vogt \cite[Theorem 1.4]{LV}.
We use the classification for space curves of Coskun--Larson--Vogt
\cite[Theorem 1]{CLV} and the high-degree semistability and
specialization results of Coskun--Smith \cite{CS}.
Coskun--Jovinelly--Larson have classified the stability of normal
bundles of Brill--Noether curves in $\PP^4$ \cite{CJL}. The Farey
arithmetic is related to that used
by Coskun--Larson--Vogt for normal bundles of canonical curves
\cite{CLVCanonical}.

For a survey of these questions see Coskun--Larson--Vogt
\cite{CLVSurvey}, which records the conjecture that all but finitely
many Brill--Noether triples $(g,r,d)$ with $g\ge2$ are stable
\cite[Conjecture 5.24]{CLVSurvey}. Coskun--Smith show that for each
fixed dimension only finitely many pairs $(g,d)$ fail semistability
\cite{CS}. Theorems~\ref{uni:finite-width} and~\ref{high:affine}
approach the problem from the other side: they fix the genus and let
the dimension vary.

\subsection{Organization}
Section~\ref{sec:hyperelliptic-main} proves the instability part of
Theorem~\ref{thm:hyperelliptic-iff} using the scroll, and proves
Corollary~\ref{cor:general-complete}. Section~\ref{sec:prelim} sets up
notation and recalls degree profiles and interpolation.
Section~\ref{high:section} proves the semistability part of
Theorem~\ref{thm:hyperelliptic-iff} by the Weierstrass-point
induction. Section~\ref{sec:genus-two} proves the genus-two
classification; it treats in turn sharper profiles of complete
normal bundles, profiles of pointing kernels, pointing models in genus
two, the twenty-two types checked directly, the excesses
$1\le c\le4$, and the balanced and Farey induction. Sections~\ref{sec:models} and~\ref{sec:kernels} construct pointing
models and prove the kernel profile for $g\ge3$.
Section~\ref{sec:uniform-propagation} proves the reduction to the
strip $0\le c\le4g-4$, Section~\ref{sec:affine} proves the
large-dimension estimate, and Section~\ref{sec:finite-verification}
combines the two.

Two steps in genus two use exact computer calculations: the
twenty-two types of Proposition~\ref{g2:co:seeds}, checked by linear
algebra over $\mathbb F_{101}$ and $\mathbb F_{10007}$, and the integer
computation in Proposition~\ref{g2:ld:theorem}, which covers 156
residue classes and 184 further cases. All other steps, including every infinite family of cases, are proved
in the text. The programs, their inputs and their outputs are in the
ancillary directory \texttt{anc/genus\_two}.

\subsection*{Acknowledgements}
I am deeply grateful to my advisor, Eric Larson, for countless
enlightening conversations and for his thoughtful and dedicated
mentorship. I also thank Eric Jovinelly for many useful discussions.

I would also like to thank ChatGPT (GPT-6 Astra) for substantive assistance
in preparing this manuscript. Roughly speaking, I first came up with
the package of inductive arguments and geometric constructions used in
this paper, and used them to prove the main theorem in a range of cases
that was broad enough to become convinced that this package was likely
sufficient to get strong results. I then provided this template
illustrating this package to ChatGPT and asked it to figure out the
combinatorial problem of exactly which base cases were needed to let
this same inductive strategy run more generally and to verify each of
these finitely many base cases by explicitly calculating the normal
bundle. The final manuscript was prepared jointly by the author and
ChatGPT.

I also thank Claude (Anthropic). It wrote the original programs for
the larger cohomological certificates in the ancillary files, audited
earlier versions of the arguments, and drafted the introduction and
the final expository revisions of this manuscript.

\section{The sharp hyperelliptic threshold}
\label{sec:hyperelliptic-main}

Throughout the paper the ground field is algebraically closed of
characteristic zero.  A complete embedding is the morphism defined by
the full space of sections of a very ample line bundle.  Generality in
the polarization is part of the following theorem; the underlying
hyperelliptic curve may be fixed arbitrarily.

\begin{theorem}[Sharp hyperelliptic threshold]
\label{thm:hyperelliptic-iff}
Let $C$ be any fixed smooth hyperelliptic curve of genus $g\ge2$, and
let $n\ge3$.  A general line bundle
$L\in\Pic^{n+g}(C)$ is nonspecial and defines a complete embedding
$C\hookrightarrow\PP^n$.  Its normal bundle is semistable if and only if
\[
                         n\ge3g+1.
\]
\end{theorem}

The necessity of the threshold comes from a line subbundle supplied by
the hyperelliptic scroll.  We first make this scroll available in every
dimension occurring in the theorem, including $n<g+1$.

\begin{lemma}[The scroll for a general polarization]
\label{lem:balanced-hyperelliptic-scroll}
Let $\pi:C\to\PP^1$ be the hyperelliptic map, let
$A=\pi^*\OO_{\PP^1}(1)$, and put $r=n-1\ge2$.  For a general
$L\in\Pic^{r+g+1}(C)$ one has
\[
 \pi_*L\simeq
 \OO_{\PP^1}\!\left(\left\lfloor\frac r2\right\rfloor\right)
 \oplus
 \OO_{\PP^1}\!\left(\left\lceil\frac r2\right\rceil\right).
\]
The complete embedding of $C$ lies on the smooth rational normal
scroll $X=\PP(\pi_*L)\subset\PP^{r+1}$.  If $H$ is the hyperplane
class and $f$ is the ruling class on $X$, then
\[
 H^2=r,\qquad Hf=1,\qquad f^2=0,\qquad
 C\sim2H+(g+1-r)f.
\]
In particular, the scroll normal sequence supplies a line subbundle
\[
 N_{C/X}\subset N_{C/\PP^{r+1}},\qquad
                 \deg N_{C/X}=4g+4.
\]
\end{lemma}

\begin{proof}
Write $\pi_*L=\OO(a)\oplus\OO(b)$ with $a\le b$.  The finite map
$\pi$ preserves Euler characteristics, so $a+b=r$.  Set
$m=\lceil r/2\rceil+1$.  The degree of $LA^{-m}$ is $g-1$ if $r$
is even and $g-2$ if $r$ is odd.  A general line bundle of either
degree has no sections: its effective locus is the image of a
symmetric power of $C$ of dimension strictly less than $g$.
Translation by $A^{-m}$ is an isomorphism of Picard varieties, so
\[
 H^0(C,LA^{-m})=H^0(\PP^1,\OO(a-m)\oplus\OO(b-m))=0
\]
for general $L$.  Consequently $b\le m-1=\lceil r/2\rceil$.
Together with $a+b=r$ and $a\le b$, this gives the displayed splitting.
Both summands are positive because $r\ge2$.

The evaluation quotient $\pi^*\pi_*L\twoheadrightarrow L$ embeds
$C$ in the projective bundle $\PP(\pi_*L)$: on every length-two
fiber of $\pi$ it is the morphism defined by the full two-dimensional
space of sections.  Locally it is the familiar embedding of a
quadratic algebra using the coordinates $1,z$, and therefore is a
closed immersion also at a ramification point.  Positivity of both
summands makes the tautological line bundle on $\PP(\pi_*L)$ very
ample.  Its embedding is the smooth rational normal scroll above,
and its restriction to $C$ is the complete series of $L$.
The splitting also gives $H^1(C,L)=0$ and $h^0(C,L)=r+2$.

Since $C$ meets a ruling fiber in degree two, write
$C\sim2H+\beta f$.  Intersecting with $H$ gives
$r+g+1=2r+\beta$, hence $\beta=g+1-r$.  Thus
\[
 \deg N_{C/X}=C^2=4r+4(g+1-r)=4g+4.
\]
Both embeddings $C\subset X\subset\PP^{r+1}$ are smooth regular
embeddings.  The normal sequence has locally free quotient, so the
indicated line is a subbundle of the normal bundle of $C$.
\end{proof}

\begin{proof}[Proof of Theorem~\ref{thm:hyperelliptic-iff}]
The existence and nonspeciality assertions follow from
Lemma~\ref{lem:balanced-hyperelliptic-scroll}.  Put
$N=N_{C/\PP^n}$.  At degree $d=n+g$ one has
\[
 \rk N=n-1,\qquad
 \mu(N)=n+g+2+\frac{4g}{n-1}.
\]
For $3\le n\le3g$, the scroll subbundle satisfies
\begin{align*}
 (n-1)\bigl(\deg N_{C/X}-\mu(N)\bigr)
   &=(3g+2-n)(n-1)-4g\\
   &=(n-3)(3g-n)+2g-2>0.
\end{align*}
It therefore strictly destabilizes $N$ in every dimension below the
asserted threshold.

For every $g\ge2$ and $n\ge3g+1$, set $r=n-1$.  By
Theorem~\ref{high:complete-semistability},
\[
 F_r=N_{C/\PP^{r+1}}\otimes L^{-1}\otimes A^{-1}
\]
is semistable for general $L$ whenever $r\ge3g$.  Tensoring by a
line bundle preserves semistability, giving the required sufficiency.
\end{proof}

\begin{corollary}[General complete embeddings]
\label{cor:general-complete}
Let $g\ge2$ and $n\ge3g+1$.  The normal bundle of the complete
embedding defined by a general line bundle of degree $n+g$ on a
general smooth curve of genus $g$ is semistable.
\end{corollary}

\begin{proof}
Consider the irreducible parameter space of smooth genus-$g$ curves
with a degree-$(n+g)$ line bundle, restricted to the open locus where
the line bundle is nonspecial and very ample, and with a projective
frame for its complete section space.  One can equally work on a
smooth cover of this parameter stack.  The universal embeddings
give a family of normal bundles, and semistability is open in this
family.  Theorem~\ref{thm:hyperelliptic-iff} supplies a point of its
semistable locus.  This locus is therefore a nonempty open subset
of the irreducible parameter space, proving the assertion for a
general curve and general polarization.
\end{proof}

\begin{remark}[Scope of the threshold]
\label{rmk:hyperelliptic-scope}
The converse in Theorem~\ref{thm:hyperelliptic-iff} is specific to
hyperelliptic curves.  The corollary asserts only a sufficient bound
for a general smooth curve; it makes no instability assertion below
$n=3g+1$.  Likewise, the all-genus sufficiency assertion retains
generality in $L$.  It is not an assertion for every polarization on
a fixed hyperelliptic curve.

The nonspecial degree convention loses no complete hyperelliptic
embeddings.  Indeed, if a line bundle $L$ on a hyperelliptic curve
were special, writing $\pi_*L=\OO(a)\oplus\OO(b)$ with $a\le b$
would give $a\le-2$.  All global sections then come from the other
summand (if there are any).  Away from their base locus, the resulting
map factors through $\pi$, and hence cannot be an embedding.
Thus a very ample line bundle on a hyperelliptic curve is nonspecial;
if its complete series has dimension $n$, its degree is $n+g$.
\end{remark}

\section{Conventions, profiles, and interpolation}
\label{sec:prelim}\label{sec:profiles}
We work over an algebraically closed field of characteristic zero.
A bundle $E$ on a smooth curve is semistable if every nonzero proper
subbundle $S$ satisfies $\deg S/\rk S\le\deg E/\rk E$; stability
means strict inequality. It is equivalent to test torsion-free
subsheaves, since saturation preserves rank and increases degree.
Tensoring by a line bundle preserves both properties.

A general nonspecial embedding of numerical type $(g,n,d)$ means a
general smooth genus-$g$ curve $C$, a general nonspecial line bundle
$L\in\Pic^d(C)$, and a general $(n+1)$-dimensional subspace of
$H^0(C,L)$ defining an embedding. Throughout the incomplete-embedding
arguments $n\ge3$ and $d\ge n+g$. At $d=n+g$ the series is complete;
for $d>n+g$ it is a projection of a complete series. These parameter
spaces are irreducible. Generality of a residual series in a
specialization always means generality in this parameter space.

Choose a line bundle $\Lambda$ of degree two and set
\begin{equation}\label{eq:notation}
 r=n-1,\quad q=2g,\quad c=d-n-g,\quad k=q+c=d-n+g,
 \qquad F=N_{C/\PP^n}L^{-1}\Lambda^{-1}.
\end{equation}
Then $\rk F=r$ and $\deg F=2k$. On a hyperelliptic curve we normally
take $\Lambda$ to be the hyperelliptic pencil; in genus two this is
$K_C$. Semistability is independent of this normalization. The
complete locus is $k=q$, and the nonspecial range is $k\ge q$.

\paragraph{Terminology.}
We call $F$ the \emph{normalized normal bundle}, $k$ its
\emph{weight} and $c$ its \emph{excess}. We refer to $(r,k)$, or
equivalently $(n,d)$, as the \emph{type} of the embedding. A type is
\emph{complete} if $c=0$, \emph{projected} if $c>0$, and
\emph{admissible} if $k\ge q$, that is, if it is nonspecial. For
fixed $g$ and $k$, the types $(r,k)$ with $r\ge2$ form the
\emph{band} of weight $k$; a \emph{strip} is a range of consecutive
bands, such as $0\le c\le4g-4$. A subsheaf of $F$ of rank $t$ and
degree $e$ has \emph{surplus} $re-2kt$. It is strictly destabilizing
exactly when its surplus is positive, and $(t,e)$ is then a
\emph{strict type}. Section~\ref{sec:affine} works with the negative
of the surplus, which it calls the deficit. The image of a general
nonspecial embedding is a \emph{general nonspecial Brill--Noether
curve}.

The pointing models and pointing kernels described in the
introduction are constructed in Proposition~\ref{g2:prop:models} and
Lemma~\ref{g2:prop:rank-one} for $g=2$, and in Theorem~\ref{thm:models}
for $g\ge3$. A pointing kernel
$P_h(b)$ has rank $b$ and degree $2h$, and we call $h$ its weight.
A pointing model with kernel $P_h(b)$, where $1\le b\le r-2$, is
\emph{admissible} if its residual type $(r-b,k-h)$ is admissible,
that is, if $h\le c$.

\subsection{Subsheaf degree profiles}
A degree profile for a bundle $B$ is a function $p_B(t)$ bounding the
degree of every rank-$t$ subsheaf, with $p_B(0)=0$ and
$p_B(\rk B)=\deg B$. The endpoint includes all full-rank images;
a proper full-rank image may have a torsion defect and is not a
proper-rank subbundle. From an extension $0\to P\to M\to Q\to0$
we obtain
\begin{equation}\label{eq:convolution}
 p_M(t)\le\max_{u+v=t}\{p_P(u)+p_Q(v)\}.
\end{equation}
This follows by intersecting a subsheaf with $P$ and taking its
image in $Q$. Saturation only increases the degree at fixed rank.
Relative Quot properness implies that a profile for a special bundle
is a profile for a general member of its family.

\begin{lemma}\label{lem:integral}
For $g\ge3$, a general nonspecial normal bundle of integral slope
is semistable. In the coordinates \eqref{eq:notation}, this applies
whenever $r\mid2k$.
\end{lemma}
\begin{proof}
The nonspecial interpolation theorem \cite[Theorem 1.3]{ALY} has
only genus-two exceptions. If $a=\mu(N)+1-g$ is a nonnegative integer,
interpolation at a general effective divisor $D$ of degree $a$ and
Riemann--Roch give $H^0(N(-D))=H^1(N(-D))=0$. A subbundle of
larger slope has positive Euler characteristic after the same twist,
which is impossible. If $a<0$, interpolation includes nonspeciality
$H^1(N)=0$, whereas $\chi(N)=r a<0$, already a contradiction; thus
this case does not occur. Twisting preserves semistability and the
integrality of the slope.
\end{proof}

More generally, interpolation gives the profile
\begin{equation}\label{eq:interpolation-profile}
 p_F(t)\le \left\lceil\frac{2k}{r}\right\rceil t
 \qquad(0<t<r)
\end{equation}
in the nonspecial range outside the interpolation exceptions. To see
this, twist $N$ by a general divisor of degree
$\lceil\mu(N)\rceil+1-g$; interpolation gives vanishing of sections.
Any subbundle of greater slope would have positive Euler
characteristic. We will also use the following stronger
combination of the two nearest interpolation twists. Write
\(2k=ar+b\), with \(0\le b<r\). At a general divisor of degree
\(\lfloor\mu(N)\rfloor+1-g\), interpolation and Riemann--Roch give
\(h^0(N(-D))=b\) and \(h^1(N(-D))=0\). For a rank-\(t\)
subsheaf \(S\subset F\), the corresponding subsheaf of this twist
has Euler characteristic \(\deg S-at\) and at most \(b\) sections.
Consequently \(\deg S\le at+b\). Combining this with
\eqref{eq:interpolation-profile} gives
\begin{equation}\label{eq:interpolation-squeeze}
 p_F(t)\le at+\min(t,b)\qquad(0<t<r).
\end{equation}
When \(b=0\), use the integral-slope lemma. These interpolation caps can improve a degree profile even when they
do not alone imply semistability.

\section{Complete hyperelliptic normal bundles}
\label{high:section}
We prove the sufficiency assertion of
Theorem~\ref{thm:hyperelliptic-iff} in every genus $g\ge2$ by one
Weierstrass-projection induction; the main result is
Theorem~\ref{high:complete-semistability}.
Sections~\ref{high:sec:scroll} and~\ref{high:sec:jets} bound the
degrees of subsheaves using the scroll and the sections of $F_r$,
Section~\ref{high:sec:connecting} computes the connecting maps of the
Weierstrass projection, and Section~\ref{high:sec:induction} carries
out the induction. Section~\ref{high:sec:gap} proves a stronger
positive-degree gap for $g\ge3$, and
Corollary~\ref{high:general-curves} extends both results to general
curves; Section~\ref{sec:affine} uses them for incomplete
embeddings. As in Section~\ref{sec:profiles}, a full-rank image with a
torsion defect is kept distinct from the whole quotient.

\subsection{The scroll, its core, and complete quotient profiles}
\label{high:sec:scroll}

Fix a smooth hyperelliptic curve $C$ of genus $g\ge2$, its degree-two
pencil $A$, and a line bundle $L$ of degree $r+g+1$.  Write
\[
 E_r=N_{C/\PP^{r+1}}L^{-1},\qquad F_r=E_rA^{-1}.
\]
Here the embedding is complete.  Thus $\rk F_r=r$, $\deg F_r=4g$,
and $K_C=A^{g-1}$.  Write $\mathcal R_r$ for the rank-two bundle in
the following lemma; it will be called the scroll core.

\begin{lemma}\label{high:scroll}
For $r\ge g+2$ there are exact sequences
\begin{align}
0&\longrightarrow B_r:=LA^{g-r}\longrightarrow F_r
 \longrightarrow Q_r:=M_{LA^{-1}}^*\longrightarrow0,
 \label{high:scroll-quotient}\\
0&\longrightarrow B_r\longrightarrow\mathcal R_r
 \longrightarrow B_{2,r}:=LA^{1-r}\longrightarrow0,
 \qquad F_r/\mathcal R_r=A^{\oplus(r-2)}.
 \label{high:core}
\end{align}
Their degrees are
\[
 \deg B_r=3g+1-r,\quad \deg B_{2,r}=g+3-r,
 \quad\deg\mathcal R_r=4g+4-2r,
 \quad\deg Q_r=r+g-1.
\]
Every generated rank-$v$ quotient of $Q_r$ has degree at least
\begin{equation}\label{high:quotient-lower}
 m_g(v)=\min(2v,v+g),\qquad v\ge1.
\end{equation}
For any complete residual rank $R\ge g+2$, define
\begin{gather}
 q_R(0)=0,\qquad q_R(R-1)=R+g-1,\qquad
 q_R(t)=R+g-1-m_g(R-1-t)\quad(0<t<R-1),
 \label{high:q-profile}\\
 p_R(0)=0,\qquad p_R(R)=4g,\qquad
 p_R(t)=\max\{q_R(t),\,3g+1-R+q_R(t-1)\}
 \quad(0<t<R).
 \label{high:p-profile}
\end{gather}
An unavailable term is omitted.  Every rank-$t$ subsheaf of $F_R$
has degree at most $p_R(t)$.
\end{lemma}

\begin{proof}
Let $\pi:C\to\PP^1$ be the hyperelliptic map and write
$\pi_*L=\OO(a)\oplus\OO(b)$, $a\le b$.  Euler characteristics
give $a+b=r$.  Since $LA^{-b}$ has a section, $2b\le r+g+1$;
hence $a\ge1$ for $r\ge g+2$.  The scroll
$X=\PP(\pi_*L)\subset\PP^{r+1}$ is therefore smooth and embedded.
If $H$ and $f$ denote its hyperplane and ruling classes, then
\[
 H^2=r,\quad Hf=1,\quad C\sim2H+(g+1-r)f.
\]
Thus $N_{C/X}=L^2A^{g+1-r}$, whose normalization is $B_r$.

For clarity, the quotient map in \eqref{high:scroll-quotient} is
specified by the scroll normal sequence
\[
 0\longrightarrow\OO_X(2f)
 \longrightarrow\OO_X(H+f)^{\oplus r}
 \longrightarrow N_{X/\PP^{r+1}}\longrightarrow0.
\]
It follows by comparing the relative and ambient Euler sequences:
the evaluation kernel of $\OO(a)\oplus\OO(b)$ on $\PP^1$ is
$\OO(-1)^{\oplus r}$, and quotienting by the tangent direction of
the ruling base gives the displayed map.  Its dual, after restriction
and normalization, is the complete evaluation of $D=LA^{-1}$.
Indeed, for a summand $\OO(m)$ use the sections $1,t,\ldots,t^m$.
The relations $te_i-e_{i+1}$ have evaluation derivatives
$t^i\,dt$, $0\le i<m$, which supply all sections of $\OO(m-1)$.
This identifies the quotient with $M_D^*$.

Pull back the evaluation of $\pi_*D$ and compare it with evaluation
of $D$.  Its first kernel is $(A^{-1})^{\oplus(r-2)}$, and the
kernel of $\pi^*\pi_*D\to D$ is $A^{r-2}D^{-1}$.  Hence
\[
 0\longrightarrow(A^{-1})^{\oplus(r-2)}\longrightarrow M_D
 \longrightarrow A^{r-2}D^{-1}\longrightarrow0.
\]
Dualizing and taking the inverse image of the resulting line in
$Q_r$ proves \eqref{high:core}.

If $T$ is a generated rank-$v$ quotient of $Q_r$, then
$H^0(T^*)=0$, since $T^*\subset M_D$ and evaluation is complete.
A general $(v+1)$-tuple of sections generates $T$ everywhere:
failure at a fixed point has codimension two in the tuple space,
and allowing that point to vary still gives a proper locus.  The
kernel is $(\det T)^{-1}$.  Dualizing shows
$h^0(\det T)\ge v+1$.  Clifford's inequality if $\det T$ is
special, and Riemann--Roch otherwise, give
\eqref{high:quotient-lower}.  Subtract this lower bound from
$\deg Q_R$ to obtain \eqref{high:q-profile}.  A saturated subsheaf
of $F_R$ either misses $B_R$ and injects into $Q_R$, or contains
$B_R$ and has its remaining image in $Q_R$.  This gives
\eqref{high:p-profile}; saturation gives the same bound for every
subsheaf.
\end{proof}

\begin{lemma}\label{high:containing-scroll}
For $r\ge3g$, a saturated proper subbundle containing $B_r$ has
slope strictly less than $4g/r$.  It also satisfies
\[
 (4g-2)\rk S>\deg S\,(r-g)
\]
whenever its degree is positive.
\end{lemma}
\begin{proof}
Let $v=r-\rk S\ge1$.  Its quotient is a quotient of $Q_r$, so
$e:=\deg S\le4g-m_g(v)$.  For $v\le g$, comparison with the
normal slope follows from $r>2g$.  For $v\ge g$ and $r\ge4g$
it is immediate.  If $3g\le r<4g$, the worst value is $v=r-1$;
after multiplication by $r$ the required positive difference is
\[
 r^2-(3g+1)r+4g
 =g+(3g-1)(r-3g)+(r-3g)^2>0.
\]
For the stronger inequality set $\lambda=4g-2$.  When $v\le g$,
\[
 \lambda(r-v)-e(r-g)
 \ge2(v-1)r+4g^2-(6g-2)v
 \ge4g^2-6g+2v>0.
\]
When $v\ge g$, one has $e\le3g-v\le2g$ and
$\rk S\ge r-3g+e$.  Therefore
\[
 \lambda\rk S-e(r-g)
 \ge(\lambda-e)(r-3g)+(2g-2)e>0.
\]
\end{proof}

\subsection{Jet cohomology, pointing sections, and degree bounds}
\label{high:sec:jets}

For a complete embedding with section space $V=H^0(L)$, so that
$C\subset\PP(V^*)$, its first jets give
\begin{equation}\label{high:jet}
 0\longrightarrow J^1(L)^*\longrightarrow V^*\otimes\OO_C
 \longrightarrow E_r\longrightarrow0.
\end{equation}
Each vector of $V^*$ gives a constant section of $V^*\otimes\OO_C$,
and hence a section of $E_r$, which we call an \emph{ambient
section}; at a point $x\in C$ it points from $x$ toward the
corresponding point of $\PP(V^*)$. The ambient section of a vector
representing a point $x\in C$ is the \emph{pointing section} of $x$;
when $\OO_C(2x)\simeq A$, we give the same name to the corresponding
section of $F_r=E_rA^{-1}$. For a subspace $W\subset V^*$, which we
call a \emph{center}, the saturation of the image of $W\otimes\OO_C$
in $E_r$ is the \emph{pointing image} $P_W$: it consists of the normal
directions pointing toward $\PP(W)$. Projection from $\PP(W)$, with
its base divisor removed, is the \emph{reduced projection}. We need
two facts about these objects.

If $W\subset V^*$ is a constant vector space of dimension $s$
whose image has proper rank, necessarily $s\le r$: a hyperplane
center in $V^*$ already maps generically onto $E_r$, since it meets
the general tangent plane in dimension one. The projection modulo
$W$ is therefore nonconstant: a constant reduced projection would put the original
nondegenerate curve in a proper linear subspace.  In characteristic
zero its point and derivative directions are generically independent.
It follows that the saturated pointing image $P_W$ has rank $s$.
Consequently a proper rank-$s$ subbundle of $E_r$ contains at most
$s$ independent ambient sections, and a positive-rank quotient
of rank $q$ receives at least $q+2$ independent ambient sections.
The first assertion also applies to nonsaturated subsheaves.

Here is the determinant version, including base points.  Let $D$
be the base divisor of the projection and let $T$ be the divisor
of zeros of its reduced differential.  The reduced polarization
is $M=L(-D)$.  The saturated point-and-derivative plane in the
remaining constant ambient bundle has determinant
$M^{-2}K_C^{-1}(T)$: the determinant of $J^1(M)^*$ is
$M^{-2}K_C^{-1}$, and saturation adds precisely the differential
zero divisor.  Comparing this determinant with
$\det E_r=L^2K_C$ in \eqref{high:jet} gives
\begin{equation}\label{high:pointing-det}
 \det P_W=\OO_C(2D+T).
\end{equation}
The same computation shows that projection from independent points
whose residual complete series embeds $C$ gives its normal
projection sequence, with kernel determinant $\OO_C(2D)$ before
normalization by $A$.

\begin{lemma}\label{high:cohomology}
For $r\ge g+2$,
\[
 h^1(E_r)=(g-2)(r-2),\qquad
 h^0(E_r)=r+2g+4,\qquad h^0(F_r)=2g+2.
\]
Moreover $H^0(F_r^*A)=H^0(LA^{-2})$, and the scroll core is the
common kernel of all maps $F_r\to A$.
\end{lemma}
\begin{proof}
Choose a basis $a,b$ of $H^0(A)$ and put
$W_A=b\,da-a\,db\ne0$ in $H^0(K_CA^2)$.
For $1\le m\le g$, the evaluation kernel of $H^0(A^m)$ is
$(A^{-1})^{\oplus m}$.  Thus the value relations after
multiplication by $H^0(L)$ are $m$ adjacent pencil relations,
with coefficients in $H^0(LA^{-1})$.  Applying first derivatives
to these relations has image
\[
 W_A H^0(LA^{m-2}).
\]
To verify surjectivity onto the coefficient space, the adjacent
relations give all degree-$(m-1)$ polynomials in $a,b$ times a
section of $LA^{-1}$.  Their multiplication is onto: its evaluation
kernel is a direct sum of $LA^{-2}$, whose degree
$r+g-3$ is at least $2g-1$.  The value multiplication is onto
for the same reason.  The derivative coefficients are nonzero in
characteristic zero.

Take $m=g-1$ and use \eqref{high:jet} and duality.  All line
bundles in the following expression are nonspecial, and
\[
 h^0(E_r^*K_C)
 =(g-1)h^0(LA^{-1})-h^0(LA^{g-3})
 =(g-2)(r-2).
\]
Riemann--Roch gives the asserted value of $h^0(E_r)$.
Taking $m=g$, duality identifies $H^0(F_r)^*$ with the cokernel
of the corresponding first-jet multiplication.  Its derivative
cokernel has dimension
\[
 h^0(LA^{2g-1})-h^0(LA^{g-2})=2g+2.
\]
Finally, for $m=2$, the simultaneous value and derivative relations
are exactly
\begin{equation}\label{high:quadratic}
 (-b^2v,\,2abv,\,-a^2v),\qquad v\in H^0(LA^{-2}),
\end{equation}
in the basis $a^2,ab,b^2$.  Indeed the two adjacent relations
have coefficients $u,v_1$, and their derivative vanishes precisely
when $au+bv_1=0$; the pencil trick gives the displayed formula.
Since $h^0(LA^{-2})=r-2$, the $r-2$ independent maps supplied
by the quotient $F_r\to A^{r-2}$ in \eqref{high:core} exhaust
all maps to $A$.  Their common kernel is $\mathcal R_r$.
\end{proof}

Let $D_W=w_1+\cdots+w_w$, $w=2g+2$, be the Weierstrass divisor.
Riemann--Hurwitz gives $\OO_C(D_W)=A^{g+1}$, and
$\OO_C(2w_i)=A$.  The normalized pointing line at each $w_i$
is therefore trivial.

\begin{lemma}\label{high:sections}
For $r\ge3g$ the $w$ Weierstrass pointing sections are generically
independent and form all of $H^0(F_r)$.  Every subsheaf $S$ satisfies
$h^0(S)\le\rk S$.  For $r\ge3g+2$ there is an exact sequence
\begin{equation}\label{high:projection}
 0\longrightarrow\OO_C^w\longrightarrow F_r
 \longrightarrow F_R\longrightarrow0,
 \qquad R=r-w,
\end{equation}
where the residual complete polarization is $L'=L(-D_W)$.
A saturated positive-degree proper subbundle satisfies
$h^0(S)\le\rk S-1$.  If $U=S\cap\OO_C^w$, then
$\rk U\ge h^0(S)$.
\end{lemma}
\begin{proof}
The residual line bundle has degree $r-g-1\ge2g-1$ and is
nonspecial and nonconstant.  Thus the Weierstrass points impose
independent conditions.  Their constant center has dimension $w$,
and its pointing image has rank $w$ by the point-and-derivative
calculation above.  Lemma~\ref{high:cohomology} now shows that these
are all sections, with generically injective evaluation.

For $r\ge3g+2$, the residual degree is at least $2g+1$, so its
complete series embeds $C$.  The normal projection has a
saturated kernel of rank $w$, with normalized determinant
$\OO_C(2D_W)A^{-w}=\OO_C$.  The $w$ independent trivial
pointing lines consequently equal the whole kernel, proving
\eqref{high:projection}.  A constant $s$-dimensional subspace of
$H^0(\OO^w)$ gives a saturated trivial subbundle of $F_r$.
If a saturated rank-$s$ subbundle $S$ had $s$ sections, that trivial
subbundle would force $S=\OO^s$, contrary to positive degree.
All sections of $S$ lie in $U$, proving the last assertion.
\end{proof}

\begin{lemma}\label{high:original-bounds}
Suppose $r\ge3g$, and let $S\subset F_r$ be saturated, with
$0<s=\rk S<r$, $e=\deg S>0$, and
$h=\rk(S\cap\mathcal R_r)$.  Then
\begin{equation}\label{high:degree-bound}
 e\le2g+2+(g-2)h,
 \qquad h^0(S)\ge e-(g-1)h.
\end{equation}
If $S$ does not contain $B_r$, then $h\le1$ and $e\le3g$;
when $h=1$ it also satisfies
\begin{equation}\label{high:core-bound}
 e\le2s+g+1-r.
\end{equation}
\end{lemma}
\begin{proof}
The inclusion $(A^{g-3})^{r-2}\subset F_r^*A^{g-2}$ supplies
all global sections by Lemma~\ref{high:cohomology}.  For $g=2$
both section spaces are zero, since $H^1(E_r)=0$; the section
bound below is then simply zero.  Put $q=r-s$.
The subspace annihilating $S$ has dimension at most
$(g-2)(q+h-2)$.  Indeed evaluation of polynomials of degree
$g-3$ on $g-2$ distinct general fibers is an isomorphism onto
their values.  At a point of each such fiber outside the relevant
rank-drop locus, annihilation of the rank-$(s-h)$ image of $S$
in $A^{r-2}$ imposes $s-h$ independent conditions.  The resulting
product of value spaces has dimension $(g-2)(q+h-2)$.

The quotient $(F_r/S)A$ has at least $q+2$ independent ambient
sections.  Riemann--Roch and the preceding bound give
\[
 q+2\le4g-e+(3-g)q+(g-2)(q+h-2),
\]
which is the first inequality of \eqref{high:degree-bound}.
The rank-$(s-h)$ quotient of $S$ by its core intersection injects
into $A^{r-2}$.  Its dual twisted by $A$ has $s-h$ generically
independent constant sections.  Multiplication by
$H^0(A^{g-2})$ gives $(g-1)(s-h)$ independent sections of its
Serre dual, and hence of $S^*K_C$.  Riemann--Roch proves the
second inequality.

If $S$ misses $B_r$, its core intersection has rank at most one:
a rank-two intersection would force $\mathcal R_r\subset S$
by saturation.  A rank-one intersection injects into $B_{2,r}$,
while the remaining image lies in $A^{r-2}$.  Their degree bounds
give \eqref{high:core-bound}.
\end{proof}

\begin{lemma}\label{high:maximal-degree}
A subbundle $S\subset F_r$, $r\ge3g$, of slope at least $4g/r$
and corank $q\ge g$ cannot have degree $3g$.
When $g=2$, every proper subbundle has degree at most six, and
a subbundle of corank at least two has degree at most five.
\end{lemma}
\begin{proof}
By Lemma~\ref{high:containing-scroll}, it misses $B_r$.  If its
degree were $3g$, equality would hold throughout the quotient
section count in Lemma~\ref{high:original-bounds}.  For $g>2$
this also forces $h=1$; for $g=2$ the coefficient of $h$ in that
count is zero, so no condition on $h$ is needed.
Thus all sections of $(F_r/S)A$ are its $q+2$ ambient sections.
Their constant kernel has dimension $s$, so $SA$ is the saturated
pointing bundle of that center.  Formula~\eqref{high:pointing-det}
gives
\[
 2\deg D+\deg T=2s+3g.
\]
The reduced projected series has at least $q+2\ge g+2$ sections.
It is nonspecial by Clifford's inequality, and therefore its
degree is at least $q+g+1$.  Since $\deg L=r+g+1$, this gives
$\deg D\le s$ and $\deg T\ge3g>0$.

Enlarge the center by the reduced projected value at a ramification
point.  This prescription also covers a point already in $D$:
use its leading nonzero projected jet after removing its old base
multiplicity.  Locally write the projection as $z^a h(z)$, with
one reduced coordinate a unit.  If $c\ge2$ is the smallest
positive vanishing order of the other reduced coordinates, then
the old differential order is $c-1$.  Killing the reduced value
raises the base multiplicity from $a$ to $a+c$.  The new
differential order is nonnegative, so the pointing determinant
degree increases by at least $2c-(c-1)=c+1\ge3$ at this point.
At every other point the local contribution to $2D+T$ cannot
decrease. If no new base point is introduced, projection cannot
lower the differential vanishing order. If a new base point of
order $c_0$ is introduced, the old reduced differential order is
$c_0-1$, while the added base contribution is $2c_0$ and the new
differential order is nonnegative. Thus the total determinant
degree increases by at least three.
The resulting normalized pointing subbundle has rank $s+1$
and degree at least $3g+1$; normalization subtracts two from
the determinant increment.  It is proper because $q\ge g\ge2$.
If it misses $B_r$, Lemma~\ref{high:original-bounds} bounds its
degree by $3g$.  If it contains $B_r$, its new corank is at least
$g-1$, and \eqref{high:quotient-lower} bounds its degree by
$4g-m_g(g-1)=2g+2<3g+1$.  Both alternatives contradict its degree.

For $g=2$, Lemma~\ref{high:original-bounds} gives the degree bound
six for every proper saturated subbundle, independently of its
core intersection.  If a degree-six subbundle had corank at least
two, it could not contain $B_r$, since its quotient would then
have degree at least $m_2(2)=4$.  The preceding equality and
ramification argument applies without any slope hypothesis and
again gives a contradiction.  Saturation gives the assertions for
all subsheaves as well.
\end{proof}

\subsection{Connecting maps of the Weierstrass projection}
\label{high:sec:connecting}

\begin{lemma}\label{high:generated-dual}
If $U$ is a saturated rank-$u$ subsheaf of $\OO_C^w$, of degree
$-\beta$, then $\beta\ge0$ and
\begin{equation}\label{high:epsilon}
 h^0(U^*)\le u+\epsilon_g(\beta),\qquad
 \epsilon_g(\beta)=\max\{\lfloor\beta/2\rfloor,\beta-g\}
 \le\beta.
\end{equation}
For $u=0$ one has $\beta=0$.  An odd positive $\beta<g+1$
cannot occur.
\end{lemma}
\begin{proof}
The dual $U^*$ is generated.  For $u>0$, choose $u-1$ general
sections: their degeneracy locus has codimension two, so they give
an everywhere independent subbundle $\OO^{u-1}$.  The quotient
is $\det U^*$, a generated line of degree $\beta$.  Clifford's
inequality and Riemann--Roch give
$h^0(\det U^*)\le1+\epsilon_g(\beta)$, proving the bound.
Degree zero forces this line, and then $U^*$, to be trivial.
For the last assertion a generated positive-degree line has a
base-point-free pencil.  If its odd degree is $\beta$, that
pencil cannot factor through the degree-two hyperelliptic map.
The combined map to $\PP^1\times\PP^1$ is birational onto its
image: its generic degree divides two, and degree two would force
that factorization.  The arithmetic genus of a curve of bidegree
$(2,\beta)$ is $\beta-1$, so $g\le\beta-1$.  This proves the
asserted exclusion.  Only the inequality \eqref{high:epsilon},
not this exclusion of odd degrees, is needed below.
\end{proof}

\begin{lemma}\label{high:connecting-A}
Suppose $r\ge3g+4$ and write $R=r-w$.  For the projection
\eqref{high:projection}, the connecting map
\[
 \delta_A:H^0(F_RA)\longrightarrow H^1(A)^w
 =H^0(A^{g-2})^{*\oplus w}
\]
has kernel the residual ambient section space $V'$ of dimension
$R+2$, and
\begin{equation}\label{high:evaluation-lines}
 \operatorname{im}\delta_A=
 \bigoplus_{i=1}^w k\operatorname{ev}_{w_i}.
\end{equation}
The evaluation lines are understood with their natural fiber scalars.
\end{lemma}
\begin{proof}
The residual rank is at least $g+2$.  Lemma~\ref{high:cohomology}
and $h^0(A)=2$ give
\[
 \dim\ker\delta_A=h^0(F_rA)-2w=R+2.
\]
The $R+2$ residual ambient sections lift through the linear
projection, so they are exactly this kernel.  The same cohomology
formula gives $\rk\delta_A=w$.  When $g=2$ this already means
surjectivity onto $H^1(A)^6=k^6$.  The following contraction
calculation also records the pointing maps needed in
Lemma~\ref{high:canonical-connecting}.

We compute the image.  Formula~\eqref{high:quadratic}, together
with the core quotient, identifies
\[
 H^0(F_r^*A)\otimes H^0(A^{g-3})
 \xrightarrow{\ \sim\ }H^0(F_r^*A^{g-2}).
\]
Both sides are the $(r-2)(g-2)$ sections supplied by
$(A^{g-3})^{r-2}$ in the dual core sequence; both are zero for
$g=2$.  For a Weierstrass
point $x$, let $\ell_x=b(x)a-a(x)b$, so that
$\operatorname{div}(\ell_x)=2x$ after a choice of fiber
trivialization.  Pair \eqref{high:quadratic} with the ambient
point vector of $x$.  The result, as a quadratic expression in
$a,b$, is
\[
 -v(x)\bigl(b(x)a-a(x)b\bigr)^2=-v(x)\ell_x^2.
\]
The ambient pointing section vanishes on $2x$; dividing out its
pointing zero and normalizing by $A$ divides this expression by
$\ell_x$.  Thus contraction with the normalized trivial pointing
section is a nonzero scalar times $v(x)\ell_x$.

Evaluation of $H^0(LA^{-2})$ on $D_W$ is onto, because
\[
 \deg(LA^{-2}(-D_W))=r-g-5\ge2g-1.
\]
The dual cohomology map therefore has image exactly
\[
 \bigoplus_{i=1}^w\ell_{w_i}H^0(A^{g-3})
 \subset H^0(A^{g-2})^{\oplus w}.
\]
Each summand is the hyperplane of sections vanishing at $w_i$.
Its annihilator is the evaluation line, proving
\eqref{high:evaluation-lines}.
\end{proof}

\begin{lemma}\label{high:all-image-bounds}
Let $r\ge3g+4$, let $S\subset F_r$ be saturated of positive degree
$e$, and write
\[
 U=S\cap\OO^w,\quad \rk U=u,\quad\deg U=-\beta,
 \qquad I=\operatorname{im}(S\to F_R),\quad\rk I=t,
 \quad f:=\deg I=e+\beta.
\]
Then $\beta\ge0$ and $t>0$.  Set
$h'=\rk(I\cap\mathcal R_R)$ and $\epsilon=\epsilon_g(\beta)$.
The following inequalities include all residual-image cases:
\begin{align}
 u&\ge f-(g-2)h'-\epsilon &&(0<t<R),
 \label{high:proper-image}\\
 u&\ge f-(2g-3)-\epsilon &&(t=R, I\ne F_R),
 \label{high:proper-full-image}\\
 u&\ge w-\epsilon,\qquad \beta=4g-e &&(I=F_R).
 \label{high:whole-image}
\end{align}
For proper images the core also gives
\begin{equation}\label{high:residual-core-caps}
 f\le
 \begin{cases}
 2t,&h'=0,\\
 2t+3g-1-R,&h'=1,\\
 2t+4g-2R,&h'=2.
 \end{cases}
\end{equation}
If $S$ misses $B_r$, with $h=\rk(S\cap\mathcal R_r)$, one
also has $u\ge e-(g-1)h$.
\end{lemma}
\begin{proof}
The quotient $\OO^w/U$ injects into $F_r/S$, so $U$ is saturated
in $\OO^w$.  The assertions about $\beta$ follow from
Lemma~\ref{high:generated-dual}.  Rank zero for $I$ would put
the positive-degree $S$ inside $\OO^w$, which is impossible.

Choose $\eta\in H^0(A^{g-2})$ nonzero at every $w_i$, and pair
coordinatewise to obtain $\rho:H^1(A)^w\to k^w$.  It is
injective on the image \eqref{high:evaluation-lines}.  The dual
of $H^1(UA)\to H^1(A)^w\xrightarrow{\rho}k^w$ factors as
\[
 k^w\longrightarrow H^0(U^*)
 \xrightarrow{\cdot\eta}H^0(U^*A^{g-2}).
\]
Consequently
\begin{equation}\label{high:cohom-intersection}
 \dim\bigl(\operatorname{im}\delta_A\cap
 \operatorname{im}[H^1(UA)\to H^1(A)^w]\bigr)
 \le h^0(U^*)\le u+\epsilon.
\end{equation}
The rank-$(t-h')$ quotient of $I$ by its core intersection
injects into $A^{R-2}$.  As in the proof of
Lemma~\ref{high:original-bounds}, this gives
\[
 h^1(IA)\ge(g-2)(t-h'),\qquad
 h^0(IA)\ge f+t-(g-2)h'.
\]
For $0<t<R$, the common-center calculation bounds
$\dim(V'\cap H^0(IA))$ by $t$, including when $I$ is
nonsaturated.  Naturality for $0\to U\to S\to I\to0$ makes
$\delta_A(H^0(IA))$ lie in the intersection of
\eqref{high:cohom-intersection}.  This proves
\eqref{high:proper-image}.  If $I$ is proper of full rank,
$h'=2$ and its intersection with $V'$ has dimension at most
$R+1$: all $R+2$ ambient sections generate $F_RA$, so their
containment would force $I=F_R$.  The same calculation gives
\eqref{high:proper-full-image}.  For $I=F_R$ use
$\rk\delta_A=w$ directly to get \eqref{high:whole-image}.

For \eqref{high:residual-core-caps}, a core line has degree at
most $\deg B_R=3g+1-R$, since it either lies in $B_R$ or
injects into the smaller-degree $B_{2,R}$.  A rank-two core
subsheaf has degree at most $4g+4-2R$.  Add twice the remaining
rank in $A^{R-2}$; the three asserted caps follow.  The last
assertion follows from Lemmas~\ref{high:sections} and
\ref{high:original-bounds}, because all sections of $S$ lie in $U$.
\end{proof}

\begin{lemma}\label{high:canonical-connecting}
For $r\ge3g+4$, the inverse image $H_{B_R}$ of the residual
scroll line $B_R$ under \eqref{high:projection} has an exact sequence
\[
 0\longrightarrow\OO^w\longrightarrow H_{B_R}
 \longrightarrow B_R\longrightarrow0.
\]
Its $K_C$-twisted connecting map is, after choices of fiber
trivializations,
\begin{equation}\label{high:canonical-evaluation}
 H^0(K_CB_R)\longrightarrow H^1(K_C)^w=k^w,
 \qquad \sigma\longmapsto(c_i\sigma(w_i))_{i=1}^w,
 \qquad c_i\ne0.
\end{equation}
\end{lemma}
\begin{proof}
The intrinsic common-kernel description of the cores gives a
commutative diagram from the core sequence for $F_r$ to that
for $F_R$.  On the quotients it is the constant surjection
\[
 A^{r-2}\longrightarrow A^{R-2}
\]
dual to $H^0(L(-D_W)A^{-2})\subset H^0(LA^{-2})$.
The calculation in Lemma~\ref{high:connecting-A} identifies its
kernel with $A^w$ and the map from the pointing kernel with the
diagonal map $\OO^w\to A^w$ having entries $\ell_{w_i}$.
This diagonal map is generically injective, so the original core
meets the pointing kernel in rank zero.  The snake lemma gives
\begin{equation}\label{high:core-cokernel}
 0\longrightarrow\OO^w\longrightarrow A^w
 \longrightarrow\mathcal R_R/\mathcal R_r\longrightarrow0.
\end{equation}
At $w_i$, its cokernel is cyclic of length two, locally
$\OO/(y^2)$ in its $i$th coordinate.

We need the location of the residual scroll line in this cokernel.
Project first from one Weierstrass point.  The scroll undergoes
the elementary transformation at that point.  In formal local
coordinates choose the ruling parameter $x$ and a transverse
parameter $y$ with $C:x=y^2$.  The transformed chart has
coordinates $(x,v)$, $v=x/y$, and $C':x=v^2$.  For the two
curve equations $f=x-y^2$, $f'=x-v^2$,
\[
 f'(x,x/y)=-\frac{x}{y^2}f.
\]
Along $C$ the coefficient extends as the unit $-1$.  Thus the
curve-in-scroll normal map is a unit.  The ambient normal
projection has target $N'\otimes\OO(p)$; in the regular frame
$1/y$ of $\OO(p)$ this unit becomes a simple zero.  The
normalized scroll-line map is consequently $B_r\to B_r(p)$
with exactly that simple zero.  Successive projections are valid
because every intervening complete rank is at least $R\ge g+2$.
They give
\[
 B_r\longrightarrow B_R=B_r(D_W)
\]
with a simple zero at each $w_i$ and no other zeros.

In local bases respecting the two scroll lines, the map of cores
therefore has the form
\[
 \begin{pmatrix}y&c\\0&y\end{pmatrix}
\]
after absorbing units.  Its determinant has order two by
\eqref{high:core-cokernel}; the cyclic cokernel forces $c$ to be
a unit.  The image of $B_R$ in that cokernel is therefore its
nonzero socle, a nonzero multiple of $y$ in $\OO/(y^2)$.

The inverse image of $\mathcal R_R$ in $F_r$ is the fiber product
of $A^w$ and $\mathcal R_R$ over
$\mathcal R_R/\mathcal R_r$.  To check this identification, lift
a pair having the same image in the quotient to $F_r$, and then
adjust its core component; injectivity follows because the core
has zero intersection with the pointing kernel.  Thus the
extension of $B_R$ is the pullback of \eqref{high:core-cokernel}
along this socle map.  Locally, pulling back $y$ modulo $y^2$
through multiplication by $y^2$ produces the principal part
$c_i/y$ in coordinate $i$, with $c_i\ne0$.  The residue
description of Serre duality now gives exactly
\eqref{high:canonical-evaluation}.
\end{proof}

\subsection{Semistability of complete normal bundles}
\label{high:sec:induction}

\begin{theorem}\label{high:complete-semistability}
Let $g\ge2$.  On every fixed smooth hyperelliptic curve $C$,
the complete normalized normal bundle $F_r$ is semistable for
a general $L\in\operatorname{Pic}^{r+g+1}(C)$ whenever $r\ge3g$.
For $g=2$ it is semistable for every such $L$.
\end{theorem}
\begin{proof}
We induct on $r$ in steps of $w=2g+2$, using the Weierstrass
projection \eqref{high:projection}. Step~0 records the constraints on
a destabilizing subsheaf. For $3g\le r\le4g+1$ the degree bounds of
Lemmas~\ref{high:scroll}--\ref{high:maximal-degree} exclude it
directly (Step~1). For $4g+2\le r\le5g+1$ the residual rank $r-w$
lies below $3g$, where the residual bundle is not semistable, so we
use the connecting maps of Section~\ref{high:sec:connecting} instead
(Step~2). For $r\ge5g+2$ the residual bundle is semistable by
induction (Step~3). Step~4 deals with the generality of $L$.

\casehead{Step 0: constraints on a destabilizer.}
A saturated strict destabilizer $S$, of rank $s$ and degree $e$,
satisfies
\begin{equation}\label{high:strict-data}
 4gs<er,\qquad 1\le e\le3g,\qquad
 h:=\rk(S\cap\mathcal R_r)\in\{0,1\}.
\end{equation}
Here Lemma~\ref{high:containing-scroll} excludes $B_r$.
When $g=2$ we moreover have $e\le5$.  Indeed degree six is
excluded in corank at least two by
Lemma~\ref{high:maximal-degree}, and in corank one it cannot
destabilize because $6r<8(r-1)$ for $r\ge6$.

\casehead{Step 1: $3g\le r\le4g+1$.}
For $r\le4g-2$ and $h=0$,
sections give $s\ge e$, while strict slope gives $s<e$.
For $h=1$, combining \eqref{high:core-bound} with strict slope
gives
\begin{equation}\label{high:strict-core-numeric}
 e>\frac{2g(r-g-1)}{r-2g}.
\end{equation}
In this interval the right side is at least $3g$, a contradiction.
At $r=4g-1$, the same two inequalities force
$(s,e,h)=(3g-1,3g,1)$, of corank $g$, excluded by
Lemma~\ref{high:maximal-degree}; the case $h=0$ is still excluded
by sections.  At $r=4g$, they force $e=3g$ and $s\ge3g$,
contrary to strict slope.  At $r=4g+1$, the stronger section
inequality $h^0(S)\le s-1$ excludes $h=0$, since it would give
$s\ge e+1$ and hence $4g<e$.  For $h=1$,
\eqref{high:strict-core-numeric} leaves only $e=3g-1,3g$.
The first requires $s\ge3g$, contrary to strict slope; the
second forces $s=3g$, of corank $g+1$, again excluded by
Lemma~\ref{high:maximal-degree}.

\casehead{Step 2: $4g+2\le r\le5g+1$.}
Write
\[
 r=4g+z,\quad 2\le z\le g+1,\qquad
 R=r-w=2g+z-2\in[2g,3g-1].
\]
Use \eqref{high:projection} and write $U,I,u,t,\beta,f$ as in
Lemma~\ref{high:all-image-bounds}.  The unconditional profile
\eqref{high:p-profile} simplifies in this interval to
\begin{equation}\label{high:seed-profile}
 \deg I\le\max\{t+3g-R,\;2t+4g-2R\},
 \qquad 0\le t\le R.
\end{equation}
This includes the endpoints: at $t=R$ the second term is $4g$.
For interior ranks, substitute
$m_g(v)=\min(2v,v+g)$ in \eqref{high:q-profile}; the two
terms displayed dominate all four terms obtained from
\eqref{high:p-profile} because $2g\le R\le3g-1$.

Suppose $h=1$.  Then \eqref{high:strict-core-numeric} implies
$e>2g$, and sections give $u\ge e-g+1$, hence
$t\le s-e+g-1$.  Since $f\ge e$, the first possibility in
\eqref{high:seed-profile} gives
\[
 2e\le s+4g-R-1,\qquad
 e(8g-r)<4g(6g+1-r).
\]
The coefficient $8g-r$ is positive, and subtracting
$2g(8g-r)$ from the right side gives $2g(2-z)\le0$,
contrary to $e>2g$.  The second possibility gives
\[
 3e\le2s+6g-2R-2,\qquad
 e(6g-r)<4g(5g+1-r).
\]
Again $6g-r>0$, and the corresponding subtraction gives the
same $2g(2-z)\le0$.  Thus $h=0$.

Now $u\ge e$ and $t>0$, so strict slope yields
\begin{equation}\label{high:seed-small-image}
 0<t=s-u<\frac{ez}{4g},
 \qquad t<e.
\end{equation}
For $g\ge3$ the upper bound is at most $3(g+1)/4\le g$.
For $g=2$, use $e\le5$ and $z\le3$ instead: $t<15/8<g$.
Thus $0<t<g$ in every genus under consideration.
For $v=t,t-1$, the residual quotient bound gives degree at most
$v$, because $R-1-v\ge g$.  If $I$ missed $B_R$ its degree
would be at most $t<e$.  Therefore $J=I\cap B_R$ is a line
subsheaf, with
\begin{equation}\label{high:seed-line}
 \deg J\ge e+\beta-t+1>0,
 \qquad \deg B_R=g+3-z.
\end{equation}

In this range $r\ge4g+2\ge3g+4$, so
Lemma~\ref{high:canonical-connecting} applies.  Its kernel is
$H^0(K_CB_R(-D_W))$, and
\[
 \deg(K_CB_R(-D_W))=g-1-z<g.
\]
As $L$ varies, this line bundle runs by translation through the
entire Picard variety of that degree.  A general line bundle of
degree below $g$ has no sections: if its degree is nonnegative,
the effective locus has dimension at most that degree.  Hence
the canonical connecting map is injective for general $L$.
It is then injective on $H^0(K_CJ)$ for every possible $J$
simultaneously.  Since $\deg J>0$,
$h^0(K_CJ)=\deg J+g-1$.

Let $S_J$ be the inverse image of $J$ in $S$.  Naturality for
$0\to U\to S_J\to J\to0$ puts the image of its canonical
connecting map inside the image of $H^1(UK_C)\to H^1(K_C)^w$.
Duality and Lemma~\ref{high:generated-dual} give
\[
 u+\epsilon_g(\beta)\ge\deg J+g-1
 \ge e+\beta-t+g.
\]
Thus $s=u+t\ge e+g$, since $\epsilon_g(\beta)\le\beta$.
Strict slope would now imply $4g^2<ez$.  For $g\ge3$ this
contradicts
\[
 ez\le3g(g+1)\le4g^2.
\]
For $g=2$ it instead contradicts $ez\le5\cdot3=15<16$.
This completes Step~2.

\casehead{Step 3: $r\ge5g+2$.}
Assume the theorem at $R=r-w\ge3g$.  Semistability of $F_R$ gives
\[
 (e+\beta)R\le4g(s-u),
\]
including full-rank images, whose degree is at most $4g$.
Combining with strict slope gives
\begin{equation}\label{high:large-induction}
 4gu+\beta R<we.
\end{equation}
If $h=0$, then $u\ge e$, contradicting $4g>w$.  If $h=1$,
\eqref{high:strict-core-numeric} gives $e>2g$, whereas
$u\ge e-g+1$ in \eqref{high:large-induction} gives $e<2g$.
These contradictions complete strong induction on $r$.

\casehead{Step 4: generality of $L$.}
For a fixed final rank, only finitely many residual translations
$L\mapsto L(-D_W)$ occur.  A finite intersection of translates
of nonempty open subsets of the irreducible Picard variety is
nonempty.  This justifies the general-$L$ assertion at every
inductive step.  In genus two the only potential generality
condition, in Step~2, concerns a line bundle of
degree $g-1-z=1-z<0$, whose section space vanishes for every
$L$.  The initial ranks and the induction are likewise independent
of $L$, proving the asserted every-polarization conclusion in
genus two.
\end{proof}

\subsection{The complete positive-degree gap}
\label{high:sec:gap}

\begin{theorem}\label{high:complete-gap}
For every $g\ge3$ and $r\ge3g$, a general complete polarization
on each fixed smooth hyperelliptic curve satisfies
\begin{equation}\label{high:gap}
 (4g-2)\rk S\ge\deg S\,(r-g)
\end{equation}
for every positive-degree subsheaf $S\subset F_r$ of proper
positive rank.  
\end{theorem}
\begin{proof}
For $r\le2g^2$ the gap follows from
Theorem~\ref{high:complete-semistability}. Above that range we induct
on $r$ in steps of $w$ through the Weierstrass projection, and verify
the gap separately for each kind of image in
Lemma~\ref{high:all-image-bounds}.

\casehead{Step 1: reduction, and the range $3g\le r\le2g^2$.}
It is enough to prove the assertion for saturated $S$, because
saturation preserves rank and increases degree.  Put
$\lambda=4g-2$, $s=\rk S$, $e=\deg S$, and
$\Delta=\lambda s-e(r-g)$.  Lemma~\ref{high:containing-scroll}
settles the case $B_r\subset S$.  Otherwise
Lemma~\ref{high:original-bounds} gives $1\le e\le3g$.

By Theorem~\ref{high:complete-semistability}, $4gs\ge er$.
This implies \eqref{high:gap} whenever $r\le2g^2$, since
$\lambda r\ge4g(r-g)$ is equivalent to that rank bound.
Thus the entire interval $3g\le r\le2g^2$ is established.

\casehead{Step 2: images of proper rank.}
We prove the inductive step for $r\ge7g-3$ whenever the gap is
known at $R=r-w\ge3g$.  All hypotheses of
Lemma~\ref{high:all-image-bounds} hold.  Keep its notation
$u,t,\beta,f,h'$ and abbreviate $\epsilon=\epsilon_g(\beta)$.
If $0<t<R$, the residual gap gives
$\lambda t\ge f(R-g)$.  Add $\lambda u$ and use
\eqref{high:proper-image} to obtain the lower bound
\begin{equation}\label{high:gap-proper-main}
 \Delta\ge2(g-2)\bigl[e-(2g-1)h'\bigr]
       +\beta(r+g-4)-\lambda\epsilon.
\end{equation}
The last two terms are nonnegative, because
$\epsilon\le\beta$ and $r\ge3g+2$.  This proves the assertion
for $h'=0$, and for $h'=1$ with $e\ge2g-1$.

For $h'=1$ and $e\le2g-2$, the core bound
\eqref{high:residual-core-caps} reads
$f\le2t+5g+1-r$.  Combine it directly with
\eqref{high:proper-image}; multiplying the resulting lower bound
for $s=u+t$ by $\lambda$ gives
\begin{equation}\label{high:gap-proper-one}
 \Delta\ge(2g-1-e)(r-7g+3)
            +(2g-1)(3\beta-2\epsilon)\ge0.
\end{equation}
For $h'=2$, the core bound instead reads
$f\le2t+8g+4-2r$.  The same substitution gives
\begin{align}
 \Delta
 &\ge(\lambda-e)r+(7g-3)e-\lambda(6g-2)
                  +(2g-1)(3\beta-2\epsilon)\notag\\
 &\ge\lambda(g-1)>0.
 \label{high:gap-proper-two}
\end{align}
The last inequality uses $r\ge7g-3$, $e\le3g<\lambda$,
and $\epsilon\le\beta$.  This exhausts proper positive image
ranks.  Rank zero is impossible because $e>0$.

\casehead{Step 3: images of full rank.}
If $I$ is proper of full rank $R$,
\eqref{high:proper-full-image} gives $u\ge e-(2g-3)$.
Since $s=R+u$,
\[
 \Delta\ge(\lambda-e)R+e(3g-4)-\lambda(2g-3).
\]
Here $R\ge5g-5$, so the expression decreases with $e$ on
$1\le e\le3g$.  At $e=3g$ it is
\begin{equation}\label{high:gap-proper-full}
 (g-2)R+g^2+4g-6>0.
\end{equation}

If $I=F_R$, put $a=w-u$.  Formula~\eqref{high:whole-image}
gives $a\le\epsilon_g(4g-e)$ and
\begin{equation}\label{high:gap-whole}
 \Delta=(\lambda-e)r+eg-\lambda a.
\end{equation}
For $e\le2g+1$, one has
$\epsilon_g(4g-e)=3g-e$.  The resulting lower bound decreases
with $e$ when $r\ge7g-3$, and increases with $r$.
At $e=2g+1$, $r=7g-3$, it equals
$12g^2-20g+7>0$.  For $2g+2\le e\le3g$,
$\epsilon_g(4g-e)\le g-1$.  The lower bound is smallest at
$e=3g$, and gives
\[
 \Delta\ge(g-2)r-g^2+6g-2
 \ge6g^2-11g+4>0.
\]
This covers full-rank images.

\casehead{Step 4: conclusion.}
At the first rank after the base interval, $r=2g^2+1$,
\[
 R=2g^2-2g-1\ge3g,\qquad 2g^2+1\ge7g-3
 \quad(g\ge3).
\]
Strong induction therefore covers all larger ranks.  As in the
complete semistability proof, finitely many translations of
nonempty opens in the Picard variety justify every residual
generality condition.  This proves the theorem.
\end{proof}

Theorems~\ref{high:complete-semistability} and~\ref{high:complete-gap}
extend from hyperelliptic curves to general curves.

\begin{corollary}[General curves]\label{high:general-curves}
Let $g\ge2$ and $r\ge3g$. For a general smooth curve $C$ of genus
$g$, a general $L\in\Pic^{r+g+1}(C)$ and every
$\Lambda\in\Pic^2(C)$, the bundle
$F_r=N_{C/\PP^{r+1}}L^{-1}\Lambda^{-1}$ of the complete embedding is
semistable. If $g\ge3$, it also satisfies \eqref{high:gap} for every
positive-degree subsheaf $S\subset F_r$ of proper positive rank.
\end{corollary}
\begin{proof}
Work in the irreducible parameter space of smooth
curves of genus $g$, a degree-$(r+g+1)$ line bundle, and a degree-two
normalizing line bundle.  The hyperelliptic examples of
Theorems~\ref{high:complete-semistability} and~\ref{high:complete-gap}
supply one nonempty open semistability locus and one nonempty open
locus for each asserted integer degree bound.  For each rank there are
only finitely many relevant ranks and degrees, bounded above in
a family of vector bundles; the existence of a subsheaf violating
a fixed bound is closed by the relative Quot scheme.  Their
complements are therefore open.  Irreducibility transfers the
conclusions to a general complete embedding, with a general
degree-two normalization.  Neither conclusion depends on that
normalization: replacing $\Lambda$ by another line bundle of degree
two twists $F_r$ by a line bundle of degree zero, which preserves the
rank and degree of every subsheaf.  One may equivalently
perform this argument in a smooth local family through the
hyperelliptic point and then in the irreducible nonspecial
embedding parameter space.
\end{proof}

\section{Incomplete embeddings in genus two}
\label{g2:sec:classification}\label{sec:genus-two}

We now allow incomplete embedding series. We organize the proof as a
finite starting problem followed by an induction in every remaining
degree. The four unstable complete types affect the choice of
residuals, and account for a short finite addition to the starting
set. The propagation itself uses the same pointing profiles and
arithmetic as the arbitrary-genus theorem below.

\begin{theorem}[Genus-two classification]\label{g2:thm:main}
Let $n\ge3$ and $d\ge n+2$. The normal bundle of a general
nonspecial genus-two embedding in $\PP^n$ is semistable if and only if
\[
 (n,d)\notin\{(3,5),(4,6),(5,7),(6,8)\}.
\]
If $\gcd(n-1,2d+2)=1$, every semistable type in this statement is stable.
\end{theorem}

Put
\begin{equation}\label{g2:eq:notation}
 r=n-1,\qquad c=d-n-2,\qquad k=c+4=d-n+2,
 \qquad F=N_{C/\PP^n}\otimes L^{-1}\otimes K_C^{-1}.
\end{equation}
Thus $F$ has rank $r$ and degree $2k$. For $k=4$ the embedding is
complete, and Theorem~\ref{thm:hyperelliptic-iff} shows that $F$ is
semistable for $r\ge6$ and unstable for $2\le r\le5$. For $k\ge5$ the
proof has three steps.
\begin{enumerate}
\item Sections~\ref{g2:cb:section}--\ref{g2:ld:section} settle the
bands $5\le k\le8$ in every rank
(Proposition~\ref{g2:ld:theorem}), using sharper degree bounds for
complete normal bundles and the computer calculations of
Sections~\ref{g2:co:section} and~\ref{g2:ld:section}.
\item Section~\ref{g2:co:section} adds three calculations, at
$(n,d)=(8,15),(10,18),(12,19)$, where the induction of the third step
would otherwise need one of the four unstable complete bundles as a
residual.
\item Sections~\ref{g2:sec:farey} and~\ref{g2:prop:section} carry out
the induction on $k$ (Proposition~\ref{g2:prop:theorem}): rank
induction when $k=jr\pm1$, balanced and Farey pointing models
otherwise, and, in odd rank, an auxiliary pointing model to exclude
the one candidate subsheaf that the Farey models leave open.
\end{enumerate}
Each step has a counterpart in the arbitrary-genus proof.

\subsection{Complete-normal profiles for the starting bands}
\label{g2:cb:section}

Theorem~\ref{high:complete-semistability} has already established
the complete genus-two case by the same proof as in higher genus.
Here we extract the additional numerical profile used for the first
projected bands.  Let $C$ be any smooth genus-two curve, put $K=K_C$,
and let $F_R=N_{C/\PP^{R+1}}L^{-1}K^{-1}$ for a complete
polarization of degree $R+3$.  Thus $\rk F_R=R$ and $\deg F_R=8$.
All assertions below hold for every such polarization.

\begin{proposition}[Degree bounds]\label{g2:cb:degree-bound}
For $R\ge6$, a subsheaf of rank $0<t<R$ in $F_R$ has degree at most six;
if its corank is at least two, its degree is at most five.
\end{proposition}
\begin{proof}
These are the genus-two assertions of
Lemma~\ref{high:maximal-degree}.  The first follows from the
quotient section count, and the second from its equality case and
the ramification enlargement of the pointing center.
\end{proof}

\begin{proposition}[Iterated complete profile]\label{g2:cb:iterated}
For $R\ge6$, every proper rank-$t$ subsheaf of positive degree $f$
in $F_R$ satisfies
\[
                         6t>f(R-3).
\]
A valid profile, including nonpositive-degree subsheaves, is
\begin{equation}\label{g2:cb:profile}
 U_R(t)=\min\left\{
 6,\ \left\lfloor\frac{8t}{R}\right\rfloor,
 \ \left\lceil\frac{6t}{R-3}\right\rceil-1
 \right\}\qquad(0<t<R).
\end{equation}
For $t\le R-2$ the first entry may be replaced by five.
The endpoints are $U_R(0)=0$ and $U_R(R)=8$.
\end{proposition}
\begin{proof}
We induct on $R$.  For $6\le R\le11$, complete semistability
gives $8t\ge Rf$.  Since $R/8>(R-3)/6$ in this range and
$f>0$, the desired inequality is strict even though only
semistability is used.

For $R\ge12$, apply the Weierstrass projection
\eqref{high:projection}, whose residual rank is $R-6\ge6$.
We may saturate the subsheaf $J\subset F_R$ under consideration.
Write
\[
 U=J\cap\OO_C^6,\qquad \rk U=v,\qquad\deg U=-\beta,
 \qquad I=\operatorname{im}(J\to F_{R-6}).
\]
Then $\deg I=f+\beta$ and $\beta\ge0$.  The image has positive
rank because $f>0$.  Lemma~\ref{high:all-image-bounds} applies,
since $R\ge12\ge3g+4$, and $g-2=0$ in its proper-image bound.
Together with $\epsilon_2(\beta)\le\beta$, it gives
\[
 v\ge f\quad\text{if }\rk I<R-6,
 \qquad
 v\ge f-1\quad\text{if }\rk I=R-6\text{ and }I\ne F_{R-6}.
\]

In the first case the induction hypothesis, applied to the
saturation of $I$ of degree $\bar f\ge f+\beta$, gives
\[
 t=v+\rk I
 >f+\frac{\bar f(R-9)}6
 \ge\frac{f(R-3)}6.
\]
For a proper full-rank image, one has $t\ge R+f-7$.
This is strictly greater than $f(R-3)/6$, since
\[
 (6-f)R+9f-42>0\qquad(R\ge12,\ 1\le f\le6).
\]
Here the range for $f$ comes from
Proposition~\ref{g2:cb:degree-bound}.

Finally, if $I=F_{R-6}$, then $\beta=8-f$ and
Lemma~\ref{high:all-image-bounds} gives
$v\ge6-\epsilon_2(8-f)$.  For $f=6$ this says $v\ge5$,
so $t\ge R-1>R-3$.  For $1\le f\le5$, the identity
$\epsilon_2(8-f)=6-f$ gives $v\ge f$, stronger than the
proper-full-image estimate above.  This exhausts the cases.

Combine the resulting strict bound with complete semistability
and Proposition~\ref{g2:cb:degree-bound} to obtain
\eqref{g2:cb:profile}.  Nonpositive-degree subsheaves also satisfy
it, because every entry in the displayed minimum is nonnegative.
\end{proof}

For use with incomplete series we also record the normal-projection
form of the first-jet calculation already used in the complete proof.
\begin{lemma}[Saturated projection]\label{g2:cb:projection}
Let $(C,L,V)$ be any embedding series and let $W\subset V^*$ be
a constant projection center.  Suppose its reduced projection is
nonconstant, with base divisor $D$, reduced polarization
$M=L(-D)$, and differential zero divisor $T$.  Its saturated
pointing image in $N_{(C,L,V)}L^{-1}$ has rank $\dim W$ and
determinant $\OO_C(2D+T)$.  If the reduced projection is immersive,
the quotient is its normal bundle tensored by $M^{-1}$.
\end{lemma}
\begin{proof}
This is the point-and-derivative calculation preceding
\eqref{high:pointing-det}, which does not use completeness.
After removing $D$, the saturated tangent-plane bundle has
determinant $M^{-2}K_C^{-1}(T)$.  Comparing it with the original
jet sequence gives the stated determinant; generic independence
of the projected point and derivative gives the rank.  When
$T=0$ the quotient is the normalized normal bundle of the
immersive residual map.
\end{proof}

\subsection{Degree bounds for pointing kernels}
\label{g2:pk:section}

Throughout this section, $C$ is a fixed smooth projective curve of genus two
over an algebraically closed field of characteristic zero, $K=K_C$, and
$\iota$ is the hyperelliptic involution.  We fix a Weierstrass point $p$.
The auxiliary points $q_1,\ldots,q_R$ will belong to a nonempty Zariski open
subset of $C^R$.  In particular they are distinct, non-Weierstrass, and
different from $p$.  Set
\[
 T_i=\mathcal O_C(2q_i)\otimes K^{-1},\qquad
 P_0=\bigoplus_{i=1}^R T_i.
\]
Thus the $T_i$ are pairwise nonisomorphic line bundles of degree zero.
Let $\sigma_i$ be a nonzero section of $KT_i=\mathcal O_C(2q_i)$;
it is unique up to a scalar and has divisor $2q_i$.

\subsubsection{The specified modification family}

Trivialize $K$ and the $T_i$ near $p$, and choose a local parameter $z$.
The functions $\sigma_i(z)$ are units there.  Write
$D_\sigma(z)=\operatorname{diag}(\sigma_1(z),\ldots,\sigma_R(z))$.
For a subspace $W\subset \mathbf k^R$ of dimension $m$, define
\begin{equation}\label{g2:pk:family}
 P_m(W)=P_0[2p\longrightarrow D_\sigma(z)W],\qquad 0\le m\le R.
\end{equation}
More precisely, if $a_1,\ldots,a_m$ is a basis of $W$, complete it to
a basis $a_1,\ldots,a_R$ of $\mathbf k^R$ and put
$v_j(z)=D_\sigma(z)a_j$.  The bundle in \eqref{g2:pk:family} agrees with
$P_0$ away from $p$, and its lattice at $p$ has basis
\begin{equation}\label{g2:pk:lattice}
 z^{-2}v_1(z),\ldots,z^{-2}v_m(z),v_{m+1}(z),\ldots,v_R(z).
\end{equation}
This construction is independent of the complementary basis and of the
trivializations.  It depends on the first jet of $D_\sigma(z)W$, rather
than just its value at $p$.  It gives a vector bundle on
$C\times\operatorname{Gr}(m,R)$, with
\begin{equation}\label{g2:pk:rank-degree-det}
 \operatorname{rk}P_m(W)=R,\qquad \deg P_m(W)=2m,\qquad
 \det P_m(W)=\left(\bigotimes_iT_i\right)\otimes K^m.
\end{equation}
Here we used $\mathcal O_C(2p)\simeq K$.

We call these the \emph{pointing kernels}: their defining columns
are the sections with divisors $2q_i$ that occur in the construction
of Proposition~\ref{g2:prop:models}.  All arguments below concern
precisely the family \eqref{g2:pk:family}.

A coordinate partition gives a useful specialization.  If
$\mathbf k^R=\mathbf k^I\oplus\mathbf k^J$ and
$W=W_I\oplus W_J$, then
\begin{equation}\label{g2:pk:block}
 P_m(W)=P_{\dim W_I}(W_I)\oplus P_{\dim W_J}(W_J),
\end{equation}
with the corresponding selected factors on each side.  Rank-one blocks
are permitted in this vector-bundle family.  In applying the result to
curves with attached lines, such a block is an auxiliary bundle
specialization; one passes to general directions before attaching the
one-secant lines.

\begin{lemma}[Complementary duality]\label{g2:pk:duality}
Let $P'_0=\bigoplus_i T_i^{-1}$, identified with the construction using
$\iota(q_i)$.  There is an invertible constant diagonal change of
coordinates $D_0$ such that
\[
 P_m(W)^*\otimes K\simeq P'_{R-m}(D_0W^\perp).
\]
In particular, the operation on the right carries general planes to
general planes.
\end{lemma}

\begin{proof}
The relation $q_i+\iota(q_i)\sim K$ identifies
$T_i^{-1}=\mathcal O_C(2\iota(q_i))K^{-1}$.
Choose sections $\tau_i$ of $KT_i^{-1}$ with divisors $2\iota(q_i)$.
For a canonical form $\ell_i$ with divisor $q_i+\iota(q_i)$, we have
$\sigma_i\tau_i=c_i\ell_i^2$ with $c_i\ne0$.
Since the canonical map is ramified at $p$, the units $\ell_i/\ell_j$
are constant modulo $z^2$.  Consequently, modulo $z^2$,
\[
 D_\sigma(z)^{-1}=u(z)D_\tau(z)D_0
\]
for a common unit $u(z)$ and an invertible constant diagonal matrix
$D_0$.  The annihilator of $D_\sigma(z)W$ is therefore, to the required
first order, $D_\tau(z)D_0W^\perp$.
In a frame adapted to this plane, dualizing and twisting the lattice
\eqref{g2:pk:lattice} is the identity
\[
 \left(z^{-2}\mathcal O^m\oplus\mathcal O^{R-m}\right)^*
 \otimes z^{-2}\mathcal O
 =\mathcal O^m\oplus z^{-2}\mathcal O^{R-m}.
\]
It gives the asserted global isomorphism, since away from $p$ both
sides have the prescribed identification with $P_0^*$.
\end{proof}

\begin{lemma}[Separating the supports]\label{g2:pk:collision}
A member of \eqref{g2:pk:family} is the special fiber of a vector-bundle
family whose general fiber has $m$ length-two modifications at
distinct points.  A cluster of $R$ independent directions gives a twist
by the line bundle of twice its support.
\end{lemma}

\begin{proof}
Use the local frame $v_1(z),\ldots,v_R(z)$ above and let the first $m$
supports move to $z=t_1,\ldots,t_m$.  Take the free lattice generated by
\[
 \frac{v_j(z)}{(z-t_j)^2}\quad(1\le j\le m),\qquad
 v_j(z)\quad(m<j\le R).
\]
These are a basis over the curve-and-parameter ring in their meromorphic
realization.  Away from the divisors $z=t_j$ the lattice is the
unmodified bundle, so it glues to $P_0$ there.  Distinct $t_j$ give the
specified modifications, and setting all $t_j=0$ gives
\eqref{g2:pk:lattice}.  One can carry out this construction on an etale
coordinate neighborhood; thus it is an algebraic specialization.
If $m=R$, all frame vectors acquire the same double pole, giving
$P_0(2p)$.  The same argument at a support disjoint from existing
modifications proves the final assertion.
\end{proof}

\subsubsection{The one-tail bound}

We first establish the geometric input, including the fact that the
attachment point may be the fixed Weierstrass point.

\begin{lemma}[Three selected theta translates]\label{g2:pk:theta}
The points $q_i$ can be chosen in a nonempty open subset so that for
any three indices there is no degree-one line bundle $H$ satisfying
\[
 H^0\bigl(\mathcal O_C(2q_i)H^{-1}\bigr)\ne0
\]
for all three indices.  Equivalently, for every degree-one line bundle
$B$, at most two of the spaces $H^0(BT_i)$ are nonzero.
\end{lemma}

\begin{proof}
The incidence for three points is closed: take the closed incidence in
$C^3\times\operatorname{Pic}^1(C)$ and project from the proper Picard
variety.  To prove it is proper, specialize $q_1=p$ and choose distinct
non-Weierstrass points $q_2,q_3$.  Since $\mathcal O_C(2p)=K$,
Riemann--Roch makes the first condition equivalent to $H$ being
effective.  Write $H=\mathcal O_C(y)$.  The second condition gives
$2q_2\sim y+x$ for a point $x$.  As $q_2$ is not Weierstrass,
$h^0(\mathcal O_C(2q_2))=1$, so the effective divisors are equal and
$y=q_2$.  The third condition would similarly force $y=q_3$, a
contradiction.  Impose these finitely many nonempty open conditions
for all triples.  The equivalent formulation follows by putting
$H=KB^{-1}$.
\end{proof}

\begin{lemma}[Selected line classes]\label{g2:pk:selected}
Assume Lemma~\ref{g2:pk:theta}.  If a stable bundle $G$ of rank $u$ and
degree one receives nonzero maps from $s$ distinct lines among the
$T_i^{-1}$, then $s\le u+1$.
\end{lemma}

\begin{proof}
For $u=1$ this is the last formulation of Lemma~\ref{g2:pk:theta}.
For $u\ge2$, any such map is a saturated injection: a positive-degree
saturation would contradict stability of $G$.  Quotient by one
selected line.  The quotient is stable of rank $u-1$ and degree one,
since a proper positive-degree subbundle would lift to a proper
subbundle of $G$ of degree at least one.  Every other selected line
still maps nontrivially to the quotient, because distinct degree-zero
line bundles admit no nonzero morphisms between them.  Induction gives
$s-1\le u$, as required.  If no selected map exists, the assertion is
immediate.
\end{proof}

\begin{lemma}[The rank of a positive subbundle]\label{g2:pk:one-rank}
Assume Lemma~\ref{g2:pk:theta}, and let
$P=P_0[2p\longrightarrow D_\sigma(z)a]$ with every coordinate of
$a\in\mathbf k^R$ nonzero.  Every positive-degree proper saturated
subbundle $S\subset P$ has degree one and rank $R-1$.
\end{lemma}

\begin{proof}
Put $S_0=S\cap P_0$.  The injection $P_0/S_0\hookrightarrow P/S$
shows that $S_0$ is saturated in $P_0$.  Since $P_0$ is polystable of
degree zero, $\deg S_0\le0$, and $S/S_0$ has length at most two.
A degree-zero subbundle of $P_0$ is a sum of selected coordinate lines:
indeed it is a subobject of a direct sum of pairwise nonisomorphic
stable bundles of the same slope.  A proper such sum cannot gain
length in the modification, because the pole direction has all
coordinates nonzero.  Positivity therefore forces
\[
 \deg S_0=-1,\qquad \operatorname{length}(S/S_0)=2,\qquad \deg S=1.
\]
The length-two gain implies that the fiber of $S_0$ contains the value
of the modifying vector.  Thus every projection $S_0\to T_i$ is
nonzero.

If $S_0$ is not stable, it has a degree-zero subbundle.  In fact a
proper subbundle of negative integral degree has slope strictly less
than $-1/\operatorname{rk}S_0$.  A degree-zero subbundle contains a
coordinate line $T_i$.  Delete this line from the construction.
It is saturated in $P$: in the local lattice, the projection of the
modifying vector to the other coordinates is still a unit vector.
Moreover
\[
 P/T_i\simeq
 \left(\bigoplus_{j\ne i}T_j\right)
 [2p\longrightarrow \overline{D_\sigma(z)a}],
\]
and the new intersection is $S_0/T_i$, again saturated and of degree
minus one.  The projected direction still has full support.
Repeat until the intersection is stable.  Deleting $j$ coordinates
leaves a stable bundle of rank $\operatorname{rk}S-j$ and degree
minus one, whose dual receives the remaining $R-j$ selected inverse
lines.  Lemma~\ref{g2:pk:selected} gives
$R-j\le\operatorname{rk}S-j+1$.  Since $S$ is proper, its rank is
exactly $R-1$.
\end{proof}

\begin{lemma}[The pair jet at a fixed Weierstrass point]
\label{g2:pk:pair}
For general $q_1,q_2$, the rank-two bundle
\[
 (T_1\oplus T_2)[2p\longrightarrow
       (a_1\sigma_1(z),a_2\sigma_2(z))]
\]
is stable for every $a_1a_2\ne0$.
\end{lemma}

\begin{proof}
By the length argument in Lemma~\ref{g2:pk:one-rank}, a line that could
violate stability has degree one, and its intersection $A_0$ with
$T_1\oplus T_2$ has degree minus one and nonzero component maps
$h_i:A_0\to T_i$.  Put $H=KA_0$, of degree one.  The possible $H$
lie in the proper incidence
\[
 H^0(\mathcal O_C(2q_i)H^{-1})\ne0,
 \qquad i=1,2.
\]
Fix a non-Weierstrass $q_2\ne p$ and specialize $q_1$ to $p$.
The argument of Lemma~\ref{g2:pk:theta} shows that the special incidence
has just one point in its support, namely
\[
 H=\mathcal O_C(q_2),\qquad A_0=\mathcal O_C(-\iota(q_2)).
\]
The component sections then have divisors
$\operatorname{div}(h_1)=\iota(q_2)$ and
$\operatorname{div}(h_2)=q_2$.

The incidence over a parameter curve moving $q_1$ is therefore finite
after shrinking around this special value.  Normalize its finitely
many branches, making a finite base change when necessary.  All
candidates are accounted for in this way, including branches that
coalesce at the special value; reducedness of the incidence is not
required.  Each nonzero degree-one section space has dimension one,
so the candidate lines and component sections vary on these branches
after local trivializations.  Their zeros remain away from $p$.

Write $z(q_1)=t$.  On each branch the meromorphic function
\[
 f(z,t)=\frac{h_1\sigma_2}{h_2\sigma_1}
\]
has the local form
\[
 f(z,t)=\frac{u(z,t)}{(z-t)^2},\qquad u(0,0)\ne0.
\]
Here $t$ may be a ramified function of the branch parameter, which
does not affect differentiation with respect to $z$.  Thus
\begin{equation}\label{g2:pk:jet-derivative}
 \partial_z f(0,t)=t^{-3}
       \bigl(2u(0,t)+t\partial_z u(0,t)\bigr)\ne0
\end{equation}
for general nonzero $t$.  The coefficient two is nonzero in our
ground field.

For $A_0$ to gain two orders at $p$, its component ratio $h_1/h_2$
must agree with $a_1\sigma_1/(a_2\sigma_2)$ modulo $z^2$.
Its value can be adjusted by the constant coefficient ratio, but
its first derivative then requires $\partial_z f(p)=0$.
Equation~\eqref{g2:pk:jet-derivative} excludes this for every coefficient
ratio on every branch.  Finally, finiteness and properness of the
candidate incidence, together with the jet vanishing equation, show
that this is a nonempty algebraic open condition on $(q_1,q_2)$.
The same open condition also keeps all component zeros away from $p$.
\end{proof}

\begin{proposition}[One-tail profile]\label{g2:pk:one-tail}
For general selected points and every direction with full support,
$P_1$ has no positive-degree proper subsheaf.  In particular it is
stable in every rank $R\ge2$.
\end{proposition}

\begin{proof}
Impose Lemma~\ref{g2:pk:theta} and the nonempty pair conditions of
Lemma~\ref{g2:pk:pair} for all selected pairs.  It suffices to exclude a
degree-one line quotient, by Lemma~\ref{g2:pk:one-rank}.
Suppose more generally there is a nonzero map $P_1\to Q$, where
$\deg Q=1$.  Its restriction to $P_0$ has components in
$H^0(QT_i^{-1})$.  Riemann--Roch in degree one gives
\[
 h^0(QT_i^{-1})=h^0(KQ^{-1}T_i),
\]
so at most two components are nonzero by Lemma~\ref{g2:pk:theta}.
If there are two, projection onto those coordinates factors the map
through the corresponding rank-two kernel.  That kernel is
stable of rank two and degree two.  It has no nonzero map to a line
of degree one: the saturated kernel of such a map would be a line
of degree at least one.  This is a contradiction.
If only the $i$th component $u_i$ is nonzero, extension across the
modification requires $u_i\sigma_i$ to vanish on $2p$.  Its divisor
is $x_i+2q_i$, where $x_i$ is one point.  As $q_i\ne p$, the required
order of vanishing is impossible.  Thus no such map exists.
Saturation increases degree, so the bound for saturated subbundles
also bounds all proper-rank subsheaves.  Their slopes are at most
zero, which is strictly below $2/R$.
\end{proof}

\subsubsection{Specialization and determinant rigidity}

We record the two family arguments that will be used in the induction.
For a bundle $B$, a \emph{degree profile} $U_B(t)$ is an upper bound
for the degree of every rank-$t$ subsheaf, with endpoints
$U_B(0)=0$ and $U_B(\operatorname{rk}B)=\deg B$.

\begin{lemma}\label{g2:pk:openness}
For fixed rank $t$ and integer $e$, the condition that a bundle in
a vector-bundle family on $C$ have a rank-$t$ subsheaf of degree at
least $e$ is closed.  If
$0\to A\to B\to D\to0$ is an extension, then
\[
 U_B(t)\le \max_{u+v=t}\bigl(U_A(u)+U_D(v)\bigr)
\]
whenever the right side is formed from valid profiles, with the
ranks restricted to those of the corresponding factors.
\end{lemma}

\begin{proof}
A subsheaf of degree above $e$ can be modified at points to have
degree exactly $e$, with its rank unchanged.  The asserted closed
locus is therefore the image of the relative Quot scheme with the
complementary fixed rank and degree; this scheme is proper over the
base.  For the second assertion, filter a subsheaf of $B$ by its
intersection with $A$ and its image in $D$, and add their
ranks and degrees.  Replacing either by its saturation only increases
its degree, so the stated profiles apply.  In particular a full-rank
proper image in $D$ is bounded by $\deg D$, as required by the endpoint
convention.
\end{proof}

\begin{lemma}[A unique strict destabilizer at special fibers]
\label{g2:pk:determinant}
Let $X$ be a smooth rational integral variety carrying a family of
bundles of fixed rank and degree on the fixed curve $C$.  Suppose
there are two fibers such that the following hold at each:
every strict destabilizing subsheaf is one specified subbundle,
of the same rank $h$ and degree $e$ at the two fibers; and an integer
degree strictly destabilizing in rank $h$ must be at least $e$.
If the determinants of the two specified subbundles differ, the
general member of the family is semistable.
\end{lemma}

\begin{proof}
If the general member were unstable, its canonical first
Harder--Narasimhan subbundle has a fixed rank and degree on a dense
open subset.  Its determinant gives a rational map from $X$ to a
Picard variety of $C$.  Every rational map from a rational variety
to an abelian variety is constant: after a rational parametrization,
restrict to general lines, use properness to extend to $\mathbf P^1$,
and use the fact that an abelian variety has no nonconstant rational
curves.  Denote the resulting fixed determinant by $T$.

Specialize to either of the two indicated fibers along a curve in
$X$ meeting that dense open subset.  Properness of the relative Quot
scheme extends the quotient, after finite base change if needed.
Its kernel on the special fiber is a subsheaf of the same rank and
degree and determinant $T$.  The kernel is locally free on the
smooth curve; in this flat family its determinant specializes to
$T$ by separatedness of the Picard scheme.  Saturation only increases
degree, so its saturation is strictly destabilizing and must equal
the specified subbundle.  The degree lower bound now forces the
original degree to be exactly $e$.  There is consequently no
saturation gain, and $T$ equals the determinant of that subbundle.
Applying this at both fibers contradicts their different
determinants.
\end{proof}

\begin{remark}\label{g2:pk:hn-only}
Lemma~\ref{g2:pk:determinant} concerns the canonical first
Harder--Narasimhan subbundle of an \emph{unstable} general member.
It makes no assertion about the determinants of equal-slope
subbundles in a strictly semistable family.  Stability below follows
from explicit degree bounds, rather than from such an assertion.
\end{remark}

\subsubsection{The profile and its arithmetic}

For $0<m<R$ and $0\le t\le R$, define
\begin{equation}\label{g2:pk:profile}
 \Phi_{m,R}(t)=
 \begin{cases}
 \left\lfloor mt/R\right\rfloor+
 \left\lceil m(t+1)/R\right\rceil-1,&2m\le R,\\[2pt]
 \left\lfloor mt/R\right\rfloor+
 \left\lceil m(t-1)/R\right\rceil,&2m\ge R.
 \end{cases}
\end{equation}
The formulas agree when $2m=R$, and already give the endpoints
$\Phi_{m,R}(0)=0$ and $\Phi_{m,R}(R)=2m$.
We also put $\Phi_{0,R}(t)=0$ and $\Phi_{R,R}(t)=2t$.

\begin{lemma}[Equivalent formulas]\label{g2:pk:profile-identities}
Put $\alpha=m/R$, $\theta=\max\{\alpha,1-\alpha\}$, and let
$\{x\}$ denote fractional part.  Then
\begin{align}
 \Phi_{m,R}(t)
 &=2\lfloor\alpha t\rfloor+
       \mathbf 1_{\{\alpha t\}>\theta},\label{g2:pk:indicator}\\
 &=q-\mathbf 1_{q\text{ odd},\ \rho\le|R-2m|},
 \quad q=\left\lfloor\frac{2mt}{R}\right\rfloor,
 \quad \rho=2mt-Rq,\label{g2:pk:parity}\\
 \Phi_{m,R}(t)
 &=\Phi_{R-m,R}(R-t)-2(R-t-m).\label{g2:pk:reflection}
\end{align}
In particular $\Phi_{m,R}(t)\le 2mt/R$.  If $\gcd(m,R)=1$,
this inequality is strict for every $0<t<R$.
\end{lemma}

\begin{proof}
Write $\alpha t=w+x$ with $w$ integral and $0\le x<1$.
For $\alpha\le1/2$,
$\lceil\alpha(t+1)\rceil-1=w+\mathbf 1_{x>1-\alpha}$;
for $\alpha\ge1/2$,
$\lceil\alpha(t-1)\rceil=w+\mathbf 1_{x>\alpha}$.
This proves \eqref{g2:pk:indicator}.  Since $\theta\ge1/2$, the
indicator can be one only when $x>1/2$.  Separating $x<1/2$ and
$x\ge1/2$ gives \eqref{g2:pk:parity}, with the weak inequality at its
endpoint exactly as written.  The reflection identity follows from
\[
 (1-\alpha)(R-t)=R-t-m+\alpha t
\]
and the fact that $\alpha$ and $1-\alpha$ have the same $\theta$.
The upper slope bound is immediate from \eqref{g2:pk:indicator}.
For a primitive pair, equality at a proper rank could occur only
if $R$ is even, $t=R/2$, and $m$ is odd.  At that rank
$\{\alpha t\}=1/2$, and the strict inequality in the indicator gives
$\Phi_{m,R}(R/2)=m-1<m$.
\end{proof}

\begin{lemma}[Equal-slope convolution]\label{g2:pk:equal-convolution}
Suppose $(m_i,R_i)$ are positive proportional pairs with
$0<m_i<R_i$, and $(m,R)=\sum_i(m_i,R_i)$.  For any ranks $t_i$
in the corresponding intervals,
\[
 \sum_i\Phi_{m_i,R_i}(t_i)
 \le\Phi_{m,R}\left(\sum_i t_i\right).
\]
\end{lemma}

\begin{proof}
All pairs have the same $\alpha$ and $\theta$ in
\eqref{g2:pk:indicator}.  The function
$\psi(t)=2\lfloor\alpha t\rfloor+
\mathbf 1_{\{\alpha t\}>\theta}$ is superadditive.
Indeed, put $x=\{\alpha u\}$ and $y=\{\alpha v\}$.
If $x+y\ge1$, the floor carry contributes two on the right,
dominating the two possible indicators on the left.
If $x+y<1$, at most one left indicator is one, because
$\theta\ge1/2$; when one is one, the indicator for $x+y$ is also
one.  Thus $\psi(u)+\psi(v)\le\psi(u+v)$, and induction proves
the result for any number of blocks.
\end{proof}

\begin{lemma}[Farey convolution]\label{g2:pk:farey-convolution}
Let $\gcd(m,R)=1$ and $2\le m<R$.  Choose the unique integers
$0<a<m$ and $0<h<R$ with $aR-hm=1$, and put
$b=m-a$ and $\ell=R-h$.  Except when $m=R-1$, both parents
$(a,h)$ and $(b,\ell)$ satisfy $0<a<h$ and $0<b<\ell$.
For $0<t<R$ and $t_1+t_2=t$ in the parent rank intervals,
\[
 \Phi_{a,h}(t_1)+\Phi_{b,\ell}(t_2)
 \le\Phi_{m,R}(t),
\]
except for the split $(t_1,t_2)=(h,0)$.  That split has value $2a$,
and it is the unique split attaining $2a$ at total rank $h$.
For $2\le m\le R-2$ one has $\Phi_{m,R}(h)=2a-1$.
For $m=R-1$, the parents are $(1,1),(R-2,R-1)$ and
$\Phi_{m,R}(1)=0$.
\end{lemma}

\begin{proof}
The defining congruence gives $a$, and
$h=(aR-1)/m$ lies strictly between zero and $R$.
Both parent pairs are primitive.  Their slopes satisfy
\[
 \frac ah=\frac mR+\frac1{Rh},\qquad
 \frac b\ell=\frac mR-\frac1{R\ell}.
\]
First assume $2m<R$.  Then $a/h\le1/2$, $b/\ell<1/2$,
$h\ge2$, and $\ell\ge3$.  For instance $2a>h$ would imply
$R\le 2+(2m-R)h\le0$, a contradiction.

For a proper total rank write
\[
 W=\lfloor mt/R\rfloor,\quad W_1=\lfloor at_1/h\rfloor,
 \quad W_2=\lfloor bt_2/\ell\rfloor,
\]
and let $r,r_1,r_2$ be the respective remainders modulo $R,h,\ell$.
Set $\delta=t_1/h-t_2/\ell$, so $-1\le\delta\le1$.
The Farey equation gives
\begin{equation}\label{g2:pk:farey-weight}
 \frac{at_1}{h}+\frac{bt_2}{\ell}
 =\frac{mt}{R}+\frac\delta R.
\end{equation}
As $1\le r\le R-1$, one has $W_1+W_2\le W$ except possibly
when $r=R-1$ and $\delta=1$.  The latter occurs exactly at
$(t_1,t_2)=(h,0)$; there the sum of profiles is $2a$.

For any other split with $W_1+W_2\le W-1$, the two possible
indicators in \eqref{g2:pk:indicator} give a sum at most $2W$,
which suffices.  Suppose therefore that $W_1+W_2=W$.
Equation~\eqref{g2:pk:farey-weight} becomes
\begin{equation}\label{g2:pk:farey-remainders}
 r=\frac{Rr_1}{h}+\frac{Rr_2}{\ell}-\delta.
\end{equation}
The two parent indicators cannot both be one: that would give
$r_1/h+r_2/\ell>2-a/h-b/\ell>1$, whereas the same sum is
$(r+\delta)/R\le1$.
If the first indicator is one, then $r_1\ge h-a+1$, and
\eqref{g2:pk:farey-remainders} gives
\[
 r\ge R-m-1+\frac{R-1}{h}>R-m,
\]
using $Ra/h=m+1/h$ and $h\le R-3$.
If the second is one, the corresponding calculation gives
\[
 r\ge R-m-1+\frac{R+1}{\ell}>R-m.
\]
In either case the total indicator is one as well.  This proves
the claimed inequality for every nonexceptional split in the
lower half.  At rank $h$ it gives a strict upper bound below $2a$.

For $2m>R$ with $m\le R-2$, apply the reflection identity
\eqref{g2:pk:reflection} to both parents and to the total pair.
The reflected parents, in the opposite order, are the Farey parents
of $(R-m,R)$; the linear correction terms add.  Their exceptional
split corresponds exactly to $(h,0)$ in the original ranks.
This proves the assertion in the upper half.  A primitive pair
with $2m=R$ and $m\ge2$ does not occur.

For $2\le m\le R-2$ we have $mh=aR-1$, and substituting in
\eqref{g2:pk:indicator} gives $\Phi_{m,R}(h)=2a-1$:
the fractional part $(R-1)/R$ is strictly greater than
$\max(m,R-m)/R$.
Finally, if $m=R-1$, then $a=h=1$.  The profile of the second
parent is $2v-2$ at proper positive ranks and $2(R-2)$ at full
rank.  The two possible splits therefore give a sum at most
$2t-2=\Phi_{R-1,R}(t)$ for $2\le t<R$.
At $t=1$ the full first block gives two, uniquely, while
$\Phi_{R-1,R}(1)=0$.  This includes $R=3$, where the second
parent is the rank-two one-tail pair.
\end{proof}

\subsubsection{The uniform theorem}

\begin{theorem}[Universal kernel profile]\label{g2:pk:theorem}
For every $R\ge2$ and $0<m<R$, and for general selected points,
a general member $P_m(W)$ of \eqref{g2:pk:family} is semistable and
satisfies
\[
 \deg S\le\Phi_{m,R}(t)
 \qquad\text{for every rank-$t$ subsheaf }S\subset P_m(W).
\]
If $\gcd(m,R)=1$, it is stable.  The same assertions hold for a
general kernel with $m$ distinct tail supports.
More generally, if $m=jR+s$ with $0\le s<R$, a general pointing
kernel of rank $R$ and degree $2m$ is semistable and has profile
\[
 \deg S\le 2jt+\Phi_{s,R}(t).
\]
It is stable whenever $\gcd(m,R)=1$.  Rank-one kernels are line
bundles and are stable.
\end{theorem}

\begin{proof}
We induct on $R$. For a primitive pair $(m,R)$, the direct sum of
general kernels for the two Farey parents satisfies the required
bounds except in rank $h$, where the parent of larger slope is the
unique strict destabilizer (Lemma~\ref{g2:pk:farey-convolution}).
Exchanging one selected point between the parents changes its
determinant, so Lemma~\ref{g2:pk:determinant} shows that the general
common-point kernel is semistable. Openness then gives the profile in the ranks
other than $h$, semistability gives it in rank $h$ when $m\le R-2$,
and a duality argument gives the whole profile when $m=R-1$.
Nonprimitive pairs reduce to primitive blocks. Finally, a
common-point kernel is a specialization of kernels with distinct
supports (Lemma~\ref{g2:pk:collision}), so openness transfers the
results to distinct supports, and full-rank clusters give the case
$m\ge R$.

\casehead{Step 1: setup.}
We first prove the common-point assertion by induction on $R$.
The endpoint pairs $m=0,R$ are $P_0$ and $P_0\otimes K$, with
the indicated profiles.  The case $m=1$ is
Proposition~\ref{g2:pk:one-tail}, including rank two.
Choose the selected points so that the finitely many lower-rank
open conditions needed below hold for every selected coordinate
subset; impose them also for the conjugate points.  This is possible
because the coordinate projections from $C^R$ are dominant and
$C^R$ is irreducible.

\casehead{Step 2: primitive pairs, semistability.}
Suppose $\gcd(m,R)=1$ and $m\ge2$.  Take the two parents in
Lemma~\ref{g2:pk:farey-convolution}.  By induction each parent admits
the asserted profile and is stable; a parent of rank one is a line
bundle.  Choose such members on disjoint coordinate blocks of
sizes $h$ and $\ell$, obtaining a member $A\oplus B$ of
\eqref{g2:pk:family}.  Lemmas~\ref{g2:pk:openness} and
\ref{g2:pk:farey-convolution} show that every strict destabilizing
subsheaf must have rank $h$ and degree at most $2a$.
Strict instability in that rank requires integral degree at least
$2a$, since
\[
 \frac{2mh}{R}=2a-\frac2R,\qquad R\ge3.
\]
Moreover the unique maximizing split is $(h,0)$, so a subsheaf
of that degree lies in $A$ and has its full rank and degree.
It equals $A$.  Thus $A$ is the unique strict destabilizing
subsheaf of the block member.

The common-point parameter space is the rational Grassmannian
$\operatorname{Gr}(m,R)$.  Repeat the block construction after
exchanging one selected index in the first block with one in the
second.  The two unique destabilizers have determinants
\[
 \left(\bigotimes_{i\in I}T_i\right)K^a,
\]
which differ because the exchanged degree-zero lines are
nonisomorphic.  Lemma~\ref{g2:pk:determinant} proves semistability
of the general common-point member.

\casehead{Step 3: primitive pairs, the profile.}
At every proper rank $t\ne h$, the block member already has the
claimed profile, so Lemma~\ref{g2:pk:openness} gives that profile
on a nonempty open subset.  At rank $h$, when
$2\le m\le R-2$, semistability gives the integral upper bound
$2a-1=\Phi_{m,R}(h)$.
The endpoint $m=R-1$ needs an additional argument: semistability
alone would bound the degree of a line by one, whereas the claimed
profile is zero.  By Lemma~\ref{g2:pk:duality}, the dual of a general
such kernel, twisted by $K$, is a one-tail kernel $P'_1$.
A saturated rank-$t$ subbundle $S$ of $P_{R-1}$ gives a quotient
\[
 P'_1\longrightarrow S^*\otimes K.
\]
For $0<t<R$ its kernel has degree at most zero by
Proposition~\ref{g2:pk:one-tail}.  The quotient has degree at least
two, so $2t-\deg S\ge2$, or
$\deg S\le2t-2=\Phi_{R-1,R}(t)$.
This proves the missing endpoint and the entire primitive profile.
Its strict slope bound in Lemma~\ref{g2:pk:profile-identities}
proves stability.

\casehead{Step 4: nonprimitive pairs.}
If $g=\gcd(m,R)>1$, take $g$ disjoint coordinate blocks, each
carrying a general primitive kernel for $(m/g,R/g)$.
These have smaller rank and the same slope, so their direct sum
is semistable.  Lemma~\ref{g2:pk:equal-convolution} bounds its
profile by $\Phi_{m,R}$.  Openness proves both assertions for
a general common-point member, completing the rank induction.

\casehead{Step 5: distinct supports and large weight.}
Lemma~\ref{g2:pk:collision} and the openness of all finitely many
profile bounds transfer these assertions to general distinct
supports and directions.  The parameter space of ordered distinct
supports and directions is irreducible, so a nonempty open
locus gives the claimed generality.  Finally, for $m=jR+s$,
specialize $jR$ directions to $j$ full-rank clusters away from
the remaining supports.  The resulting bundle is an
$s$-tail kernel tensored with a line bundle of degree $2j$.
Its rank-$t$ degrees are therefore bounded by
$2jt+\Phi_{s,R}(t)$.  Twisting preserves semistability and
stability, and $\gcd(m,R)=\gcd(s,R)$.  Separating the clusters
and applying openness proves the last assertion.
\end{proof}

\subsection{Tail degenerations and specialization}\label{g2:sec:normal-degenerations}

We now construct the pointing models of the introduction in genus two.

\begin{proposition}[Pointing models in genus two]\label{g2:prop:models}
Fix the curve $C$ and the final line bundle $L$ in
\eqref{g2:eq:notation}, and let $1\le h\le c$ and $2\le b\le r-2$.
Choose disjoint reduced divisors $D_p=\sum_{i=1}^h p_i$ and
$D_q=\sum_{j=1}^b q_j$, put $H'=L(-D_p-D_q)$, and let
$W\subset H^0(H')$ be an $(n+1-b)$-dimensional subseries that embeds
$C$. The normalized normal bundle of the embedding of $C$ by a
general $(n+1)$-dimensional subseries of $H^0(L)$ then specializes,
with $C$ and $L$ fixed, to a bundle $M$ in an exact sequence
\begin{equation}\label{g2:eq:normal-limit}
 \begin{gathered}
 0\longrightarrow P_h(b)\longrightarrow M\longrightarrow Q\longrightarrow0,\\
 (\rk P_h(b),\deg P_h(b))=(b,2h),\qquad
 (\rk Q,\deg Q)=(r-b,2(k-h)),
 \end{gathered}
\end{equation}
with the following properties.
\begin{enumerate}
\item The quotient $Q$ is the normalized normal bundle
$N_{(C,H',W)}H'^{-1}K_C^{-1}$ of the residual embedding. Its excess
is $c-h$, so the residual series is complete exactly when $h=c$.
\item The kernel is $P_h(b)=P_0[2p_1\to y_1]\cdots[2p_h\to y_h]$,
where $P_0=\bigoplus_j\OO_C(2q_j)K_C^{-1}$ and the $y_i$ are general
points of the span of the $q_j$. It is a pointing kernel with $h$
distinct tail supports, in the sense of
Theorem~\ref{g2:pk:theorem}.
\item The determinants of the two factors are
\begin{equation}\label{g2:eq:normal-determinants}
 \det P_h(b)=\OO_C(2D_p+2D_q)K_C^{-b},\qquad
 \det Q=L^2K_C^{1-r}(\det P_h(b))^{-1}.
\end{equation}
\end{enumerate}
If $C$ is general, the divisors, the series $W$ and the directions
$y_i$ can moreover be chosen so that the residual embedding is
general and Theorem~\ref{g2:pk:theorem} applies to the kernel,
together with finitely many further nonempty open conditions; the
determinants \eqref{g2:eq:normal-determinants} then still vary
with $D_q$.
\end{proposition}

\begin{proof}
Put $H=L(-D_p)$, so that $H'=H(-D_q)$. The residual ambient dimension
is $n-b\ge3$, and $\deg H'=n-b+2+(c-h)\ge5$. Embedding subseries $W$
exist: the complete series embeds the genus-two curve, and a general
projection to dimension at least three remains an embedding. We
degenerate $C$ to a core curve, embedded by a series of $H$, with $h$
one-secant tails attached at the $p_i$, and project the core from the
span of the $q_j$.

\casehead{Step 1: the core and its pointing kernel.}
Since $H^1(H')=0$, evaluation on $D_q$ gives a quotient of
$H^0(H)$ of dimension $b$. Lift its standard basis to sections
$u_1,\ldots,u_b$ and set
\[
 V=\sigma_{D_q}W+\langle u_1,\ldots,u_b\rangle\subset H^0(H).
\]
This $(n+1)$-dimensional series embeds the core. Away from $D_q$,
the residual coordinates separate points and tangent vectors.
At $q_j$ they vanish to first order with a nonzero derivative,
while the additional coordinates give the distinct coordinate
point $[u_j]$. No point outside $D_q$ maps to their span.
Deleting any subset of these coordinate directions, and removing
the corresponding base points, gives the same construction with
fewer added coordinates. Thus all these inner projections remain
embeddings, and full projection recovers $W$.

Let $A=N_{(C,H,V)}H^{-1}K^{-1}$.
The first-jet tangent-plane calculation leading to
\eqref{high:pointing-det} applies to any embedding series; the
determinant calculation does not use completeness. It gives a
saturated rank-$b$ pointing kernel of normalized determinant
$\OO_C(2D_q)K^{-b}$, with quotient the normalized residual bundle
$Q$. Its individual pointing lines are
\[
 T_j=\OO_C(2q_j)K^{-1}.
\]
They are generically independent because the residual projection
is immersive. Their sum has the same determinant degree as the
saturated kernel, hence equals that kernel. Therefore
\[
 0\longrightarrow P_0:=\bigoplus_{j=1}^bT_j
 \longrightarrow A\longrightarrow Q\longrightarrow0.
\]
The ambient vector representing $q_j$ gives a section $\sigma_j$
of $KT_j$ with divisor $2q_j$. In these factors the pointing germ
toward $y$ in the secant span consequently has components
$(a_j\sigma_j(z))_j$, with $a_j$ the coordinates of $y$.
This identifies the first jets with the family \eqref{g2:pk:family} of
Section~\ref{g2:pk:section}.

Choose $y_i$ in the span and modify at the distinct $p_i$.
The directions remain in $P_0$, so
$M=A[2p_1\to y_1]\cdots[2p_h\to y_h]$ has exactly
\eqref{g2:eq:normal-limit}. The common-point and coordinate-block
specializations used for kernel bounds take place in the bundle
family. Openness then gives distinct supports and general
directions before attaching curves. The determinant of $P_0$ is
$\OO_C(2D_q)K_C^{-b}$, each modification adds $2p_i$, and
$\det M=\det A\otimes\OO_C(2D_p)=L^2K_C^{1-r}=\det F$; this gives
\eqref{g2:eq:normal-determinants}.

\casehead{Step 2: fixed-polarization smoothing.}
For general $y_i$, each line $p_i y_i$ is transverse and
one-secant to the core. A second intersection outside $D_q$ would
identify distinct points in the residual embedding; one in $D_q$
would require $y_i=q_j$. Tangency would make the residual
differential vanish. Two such lines cannot meet outside the
secant span, since their residual images are distinct, and cannot
meet in the span if $y_i\ne y_{i'}$. General choices avoid these
restrictions because $b\ge2$.

Let $f_0,\ldots,f_n$ be a basis of $V$.
Since $H^1(H)=0$, choose $v_1,\ldots,v_h\in H^0(L)$ with
standard-basis evaluations at $D_p$. Complete them to a basis
$v_1,\ldots,v_c$ of a complement to $\sigma_{D_p}V$ in $H^0(L)$.
If $y_i=[a_{i0}:\cdots:a_{in}]$ in the core coordinates, take
\begin{equation}\label{g2:nd:graph-family}
 \sigma_{D_p}f_j+t\sum_{i=1}^h a_{ij}v_i
   +t^2\sum_{\nu=1}^c b_{\nu j}v_\nu,
 \qquad 0\le j\le n.
\end{equation}
Blow up the disjoint points $(p_i,0)$ in $C\times\mathbf A^1$.
With parameter $z$ at $p_i$, the leading terms are
$z f_j(p_i)+t a_{ij}$. On the exceptional line they give
$[z:t]\mapsto[z f_j(p_i)+t a_{ij}]_j$.
The central fiber is therefore the embedded core with the
prescribed one-secant lines, while the general fiber has
underlying curve $C$ and polarization $L$.
At any fixed $t\ne0$, varying $b_{\nu j}$ covers the full affine
Grassmannian chart. General higher-order coefficients consequently
put the generic series in any required dense open subset of the
fixed-$(C,L)$ series space without changing the central data.

The simultaneous normal-bundle contraction of
\cite[Lemma 3.4]{CS}, applied to this family, has central bundle
\[
 N_{(C,H,V)}[2p_1\to y_1]\cdots[2p_h\to y_h](D_p)
\]
on the contracted fixed curve. Twisting by $L^{-1}K^{-1}$ cancels
$D_p$, since $L=H(D_p)$, and gives $M$ above. This proves the
claimed fixed-polarization limit.

\casehead{Step 3: generality of the residual.}
For fixed divisors, lifting the evaluation sections proves
dominance of the intersection map from core series to residual
series. For fixed $C,L,D_p$, varying $D_q$ also dominates the
Picard variety of $H'$: every degree-$b$ line bundle on a genus-two
curve is effective when $b\ge2$, so the Abel map from
$\operatorname{Sym}^bC$ is surjective.
For a general fixed curve, any property of general embeddings of
the residual type therefore holds on a nonempty open subset of this
divisor-and-series space. The kernel theorem gives another
nonempty open condition on selected points, supports, and
directions. These opens meet in the irreducible parameter space,
together with the finitely many other conditions used at a given
induction step. Thus the factor determinants can vary with $D_q$
while the final $C,L$ and the determinant of a hypothetical
generic Harder--Narasimhan subbundle remain fixed.
\end{proof}

The construction treats $b\ge2$; the rank-one models needed at the
shifted diagonals are supplied by Lemma~\ref{g2:prop:rank-one}.

\paragraph{Profiles, HN determinants, and saturation.}
A degree profile $U_B(t)$ bounds degrees of \emph{all} rank-$t$
subsheaves of $B$, with endpoints $U_B(0)=0$ and
$U_B(\rk B)=\deg B$. Filtering a subsheaf of an extension through
its kernel gives
\begin{equation}\label{g2:eq:convolution}
 \deg S\le\max_{u+v=\rk S}\bigl(U_P(u)+U_Q(v)\bigr).
\end{equation}
A proper full-rank
image is bounded by the endpoint $\deg Q$; it is not treated as a
proper-rank subbundle. Saturation can only improve upper bounds
for a fixed rank.

Existence of a subsheaf of a fixed rank and degree at least a
given integer is closed in the bundle family, by relative Quot
properness. Equivalently, modify a higher-degree subsheaf to the
specified exact degree and use its Quot scheme. Hence the profile
of a special member bounds the general profile.

If a bundle family over a rational base is generically unstable,
its canonical first HN subbundle defines a rational map from that
base to a Picard variety. Such a map is constant. For normal
bundles, fix $C,L$ and use the rational series Grassmannian.
By Proposition~\ref{g2:prop:models}, the factor determinants
\eqref{g2:eq:normal-determinants} vary with the secant divisor.

The specialization of the generic HN subbundle has its original
rank, degree, and determinant before saturation. If saturation
adds an effective divisor $Z$, its determinant becomes $T(Z)$.
We compare it with \eqref{g2:eq:normal-determinants} only when the
numerics force $\deg Z=0$.

\subsection{Exact cohomological bases}\label{g2:co:section}

We establish all finite cohomological inputs used in the starting
bands and in propagation. Their construction is independent of the
semistability induction.

\begin{proposition}[Finite cohomological witnesses]\label{g2:co:seeds}
The normal bundle of a general genus-two embedding is semistable at
each of the following types:
\begin{align*}
\mathcal B_0=\{& (5,8),(6,10),(6,11),(7,10),(7,12),(7,13),\\
              & (8,11),(8,12),(8,14),(9,12),(9,13),(9,14),\\
              & (10,15),(10,16),(11,15),(11,17),\\
              & (14,17),(14,20),(18,22)\},\\
\mathcal B_1=\{& (8,15),(10,18),(12,19)\}.
\end{align*}
The nineteen types in $\mathcal B_0$ supply the finite starting-band
calculation. The three types in $\mathcal B_1$ supply the additional
propagation inputs. All twenty-two assertions follow from exact
cohomological vanishing as described below.
\end{proposition}

\begin{lemma}\label{g2:co:criterion}
Let $F,G$ be vector bundles of positive rank on a smooth curve of
genus $g$. If $\chi(F\otimes G)=0$ and $H^0(F\otimes G)=0$, then
$F$ is semistable.
\end{lemma}
\begin{proof}
The Euler-characteristic hypothesis says $\mu(G)=g-1-\mu(F)$.
A strict destabilizer $S\subset F$ would have
\[
 \chi(S\otimes G)=\rk S\,\rk G\bigl(\mu(S)-\mu(F)\bigr)>0,
\]
contradicting $H^0(S\otimes G)\subset H^0(F\otimes G)=0$.
\end{proof}

\subsubsection{Two small first-projection certificates}

For $(n,d)=(5,8),(7,10)$, let $F=NL^{-1}K^{-1}$, of degree ten.
Over $\mathbb F_{101}$ take
\[
 y^2=x^5+x+1,\quad p=\infty,\quad q=(0,1),\quad s=(3,34),
 \quad L=\mathcal O_C(dp).
\]
The ordered pole basis of $H^0(L)$ is
$(1,x,x^2,y,x^3,xy,x^4,x^2y,x^5)$, truncated at pole order $d$.
Use the series $f_i+a_i x^{d/2}$, omitting the final basis vector,
with coefficients
\[
\begin{array}{c|l}
d& (a_i)\\\hline
8&(17,72,97,8,32,15)\\
10&(17,72,97,8,32,15,63,97).
\end{array}
\]
The quintic is squarefree. The projection functional, on polynomial
monomials, has a nonzero three-by-three Hankel minor. If the functional
factored through any length-two subscheme, its multiplication pairing
on polynomial sections would have rank at most two. This excludes
every secant and tangent projection center, so the series embeds the
curve geometrically.

Let $M_B$ be the kernel of complete evaluation of $B$, and put
$G=M_B\otimes A$, with
\[
\begin{array}{c|c|c|c|c}
(n,d)& B&A&\rk G&\deg G\\\hline
(5,8)&\mathcal O(6p)&\mathcal O(q+s-2p)&4&-6\\
(7,10)&\mathcal O(5p)&\mathcal O(q+s-p)&3&-2.
\end{array}
\]
Both complete evaluation series are base-point-free: the constant
section handles the affine curve, and an exact maximal-pole section
handles infinity. These bundles satisfy $\chi(F\otimes G)=0$.

Here is the finite complex used to compute the vanishing. Set
$W=H^0(B)$ and $D=K^2A^{-1}$, so
$D=B(-q-s)$ and $DB^{-1}=\mathcal O(-q-s)$.
Tensor the conormal first-jet sequence with the evaluation resolution
of $M_B^*$ and with $D$. The resulting complex is
\[
\begin{split}
 \mathcal C^{-1}&=V_0DB^{-1},\\
 \mathcal C^0&=(V_0W^*D)\oplus J^1(L)DB^{-1},\\
 \mathcal C^1&=J^1(L)W^*D.
\end{split}
\]
Its sole sheaf cohomology is $(F\otimes G)^*K$ in degree zero.
The maps are $u\mapsto((b_i u)_i,-j(u))$ and
$((a_i),z)\mapsto(j(a_i)+b_i z)_i$. Here the jet map is
$\mathcal O_C$-linear on the fixed space $V_0$: scalar coefficients
are not differentiated.

With $t=x^2/y$ at infinity and $\eta=dx/y$, a jet pair $(a,b)$
in $J^1(L)(mp)$ has pole bounds $N=d+m$, $N+3$, and relation
\[
 -2[b]_{-N-3}+d[a]_{-N}=0.
\]
Vanishing at $q,s$ imposes the four independent evaluations of the
two coordinates. The only nonzero component cohomology dimensions are
\[
\begin{array}{c|ccc}
(n,d)&h^1(\mathcal C^{-1})&h^0(\mathcal C^0)&h^0(\mathcal C^1)\\\hline
(5,8)&18&102&120\\
(7,10)&24&80&104.
\end{array}
\]
The last $h^1(\mathcal C^{-1})$ columns come from the \v{C}ech
transgression. A basis for $H^1(\mathcal O(-q-s))$ consists of the two
Weierstrass gap principal parts $t^{-3},t^{-1}$ and the two-point
constant-value class $(1,0)$. For each such class $h$ and section
$f\in V_0$, choose polynomial finite parts $g_i$ of $b_i h$ and a
polynomial jet finite part $g$ of $h\,j(f)$. The transgression column
is $(j(f)g_i-b_i g)_i$. Subtraction by pole order followed by
linear interpolation in $x$ at $q,s$ determines the finite parts;
different choices change the column only by a global-differential
column. The resulting square matrices have exact determinants
\[
 \det M_{5,8}=20,\qquad\det M_{7,10}=16\pmod {101}.
\]
Thus $H^1((F\otimes G)^*K)=0$, and duality gives the desired
vanishing. The ancillary files contain both complete matrices.

\subsubsection{Twenty witnesses from one multiplication construction}

For $F$ of rank $r=n-1$ and degree $2k$, take a generated
$(r+1)$-dimensional subseries $V'\subset H^0(B)$ and put
$G=M_{V',B}\otimes A$. Choose
\[
 \deg B=2k+r(\deg A-1).
\]
Then $\chi(F\otimes G)=0$, and $H^0(F\otimes G)$ is the kernel
of multiplication
\[
 V'\otimes H^0(F A)\longrightarrow H^0(FAB).
\]

On $y^2=f(x)$, with $f$ monic squarefree of degree five and $p=\infty$,
use the pole monomials $x^i,yx^i$. The derivative dual to $dx/y$ is
\[
 \partial(g_0+yg_1)=fg_1'+\tfrac12 f'g_1+yg_0'.
\]
Set $L=\mathcal O(dp)$, $A=\mathcal O((d+7)p)$, so $F A=N(5p)$.
For every $\ell\ge5$, dual first-jet evaluation computes
$H^0(N(\ell p))$ as the quotient of
$V_0^*\otimes H^0(\mathcal O((d+\ell)p))$ by the following relations:
\[
 (\alpha s_j)_j,\qquad
 (\beta\partial s_j+\gamma s_j)_j.
\]
Here $\alpha$ ranges through pole order at most $\ell$, and $\beta$
through pole order at most $\ell-2$. For the unique top-pole monomial
$\beta$, take $\gamma=-(d/2)M_{\ell+1}$, where $M_{\ell+1}$ is
the monomial of that pole order; otherwise take $\gamma=0$.
To verify completeness of these relations, the two dual-jet
coefficients have pole bounds $\ell+1$ and $\ell-2$ with leading
relation $[\gamma]_{\ell+1}=-(d/2)[\beta]_{\ell-2}$. The two
line-bundle $H^1$ terms vanish for $\ell\ge5$. The relation map is
injective because first-jet evaluation is onto. Its dimension is
$h^0(\mathcal O(\ell p))+h^0(\mathcal O((\ell-2)p))$.

These spaces are compatible with multiplication. Their source and
target dimensions in the preceding multiplication map are equal.
An economical elimination first quotients $H^0(\mathcal O((d+\ell_2)p))$
by products $V'H^0(\mathcal O((d+5)p))$, where
$\ell_2=5+\deg B$, and then reduces the target jet relations.
The resulting cokernel dimension equals $h^0(F\otimes G)$.

Over $\mathbb F_{10007}$, the recorded coefficient matrices give
zero cokernel for every type in
$(\mathcal B_0\cup\mathcal B_1)\setminus\{(5,8),(7,10)\}$.
The full row ranks of the reduced relation matrices are:
\[
\begin{array}{c|r@{\qquad}c|r}
(n,d)&\text{rank}&(n,d)&\text{rank}\\\hline
(6,10)&154&(6,11)&168\\
(7,12)&208&(7,13)&224\\
(8,11)&216&(8,12)&234\\
(8,14)&270&(8,15)&288\\
(9,12)&260&(9,13)&280\\
(9,14)&300&(10,15)&352\\
(10,16)&374&(10,18)&418\\
(11,15)&384&(11,17)&432\\
(12,19)&520&(14,17)&540\\
(14,20)&630&(18,22)&874
\end{array}
\]
Each displayed rank equals the number of rows, so the corresponding
cokernel vanishes.
All curve coefficients and both series matrices are stored explicitly
in the ancillary files. The verification program also checks
squarefreeness,
base-point-freeness of $V'$, and injectivity of the differential
everywhere. Affine conditions are tested by the unit ideal generated
by the relevant minors in $\mathbb F_{10007}[x,y]/(y^2-f)$, and the
highest two pole coefficients test the differential at infinity.

\subsubsection{Lifting to characteristic zero}

Each example lifts to a smooth proper curve over a discrete valuation
ring of mixed characteristic, with its two locally free bundles.
The first two examples use the prime above $101$ with
$\sqrt{247}\equiv34$ in $\mathbb Q(\sqrt{247})$; the twenty multiplication
examples use the integral lifts of the recorded coefficients at
$10007$. Smoothness and the required bundle conditions hold after
shrinking. Semicontinuity lifts the cohomology vanishing to
characteristic zero, including after algebraic closure.

For the later examples, an everywhere unramified map suffices to
define the locally free normal quotient. Its parameter space contains
the nonempty open locus of embeddings. The vanishing locus and this
embedding locus meet in the same irreducible space of curve, line
bundle and series data. Thus, by Lemma~\ref{g2:co:criterion}, all
twenty-two examples give general Brill--Noether semistability. The
reconstruction programs and the complete input files are among the
ancillary files. This proves Proposition~\ref{g2:co:seeds}.

\subsection{A finite reduction for the first four projection degrees}
\label{g2:ld:section}

We now prove the starting bands $5\le k\le8$ used in the induction.
The proof combines a finite calculation with the periodic lifting
Lemma~\ref{g2:ld:periodic}. The geometric inputs are the kernel
profile of Theorem~\ref{g2:pk:theorem}, the complete profile in
Section~\ref{g2:cb:section}, and the pointing models of
Proposition~\ref{g2:prop:models} with their determinants.

\begin{proposition}\label{g2:ld:theorem}
For $1\le c\le4$ and every $n\ge3$, the normal bundle of a general
genus-two curve of degree $n+2+c$ in $\PP^n$ is semistable.
\end{proposition}

For a degree profile $U$, define the minimum rank needed for degree at
least $e$ by
\[
 a_U(e)=\min\{t:U(t)\ge e\}.
\]
Only nonnegative $e$ occur below; set $a_U(0)=0$. For an extension with
kernel of degree $2c$ and quotient of degree eight, the rank of every
positive-degree $e$ subsheaf is at least
\begin{equation}\label{g2:ld:cost}
 a(e)=\min_{\substack{i+j=e\\0\le i\le2c,\ 0\le j\le8}}
       \bigl(a_P(i)+a_Q(j)\bigr).
\end{equation}
This remains valid when one image has negative degree: replace
that contribution by zero and lower the other contribution to $e$.
Minimum ranks for degree at least a specified value are monotone.

\begin{lemma}[Periodic lifting]\label{g2:ld:periodic}
Fix $1\le c\le4$, $k=c+4$, and integers $r_0,b_0$ satisfying
$b_0\ge2c$ and $R_0=r_0-b_0\ge6$. For $m\ge0$, set
\[
 r=r_0+6km,\qquad b=b_0+6cm,\qquad R=R_0+24m.
\]
Use the $c$-tail kernel of rank $b$ and the complete normal quotient
of rank $R$. For each degree $e$, the minimum rank in
\eqref{g2:ld:cost}, minus $3em$, is nondecreasing in $m$.
Consequently every degree--rank exclusion at $m=0$ persists after
replacing a candidate rank $t$ by $t+3em$.

A whole-factor determinant exclusion also persists if, at $m=0$, its
degree is the least integer strictly above the target slope, it is the
only degree allocation that fits its rank, and every larger degree
requires larger rank.
\end{lemma}

\begin{proof}
Because $b\ge2c$, the kernel profile has the explicit inverse
\[
 a_P(2\ell)=\left\lceil\frac{\ell b}{c}\right\rceil,
 \qquad
 a_P(2\ell+1)=\left\lfloor\frac{(\ell+1)b}{c}\right\rfloor.
\]
The ranges are $0\le\ell\le c$ for the even formula and
$0\le\ell<c$ for the odd formula. Increasing $m$ by one adds $6\ell$
to the even cost and $6(\ell+1)$ to the odd cost. These are at least
$3e$ in the corresponding degree.

The complete profile gives, for $1\le j\le5$,
\[
 a_Q(j)=\max\left\{
  \left\lceil\frac{jR}{8}\right\rceil,
  \left\lfloor\frac{j(R-3)}6\right\rfloor+1\right\}.
\]
Both entries are at most $R-1$ for $R\ge6$ in this degree range.
Their increments are $3j$ and $4j$, respectively. The remaining
costs are $a_Q(6)=R-1$ and $a_Q(7)=a_Q(8)=R$; their increments
are $24\ge3j$. Thus each summand in \eqref{g2:ld:cost}, after subtracting
$3em$, is nondecreasing. The finite minimum is nondecreasing too.

The strict destabilizing rank cutoff in degree $e$ is
\[
 \left\lfloor\frac{er-1}{2k}\right\rfloor
 =3em+\left\lfloor\frac{er_0-1}{2k}\right\rfloor.
\]
This proves the first assertion. A whole kernel has degree $2c$ and
rank $b=b_0+3(2c)m$; the whole quotient has degree eight and rank
$R=R_0+3(8)m$. The uniqueness inequalities for these allocations
therefore persist. Larger degrees have still larger increment, so
their exclusion persists as well. The positive target surplus of
the factor is constant, while $r$ increases. Its degree remains
the least integer above the slope. Together these conditions force
every subsheaf of the special member that has the rank of the factor
and at least its degree to be the kernel itself, or to map
isomorphically onto the quotient. The specialization
of the generic Harder--Narasimhan subsheaf has at least that degree,
since the factor's degree is the least integer above the slope; it
therefore has zero saturation gain, and the moving determinant
(Section~\ref{g2:sec:normal-degenerations}) gives the same
contradiction in every period.
\end{proof}

\subsubsection{The finite arithmetic}

For each $c$, take one initial rank in every residue class modulo
$6k$, namely $6k\le r_0<12k$. The complete set contains
\[
 30+36+42+48=156
\]
residue certificates. At each rank, try
$b_0=\lfloor cr_0/k\rfloor$ and neighboring integers. Intersect the
degree--rank sets permitted by \eqref{g2:ld:cost}, removing an entire
factor only under the uniqueness and degree conditions of
Lemma~\ref{g2:ld:periodic}. In every residue, at most two witnesses
exclude all strict types. Every witness used satisfies the hypotheses
$b_0\ge2c$ and $R_0\ge6$ of the lemma, and two witnesses for the same
residue lift with the same $m$, so their exclusions combine in every
period. Therefore the proposition holds for $r\ge6k$.

The remaining ranks are finite. In increasing order of $r$, form the
convolution of $\Phi_{h,b}$ with already proved residual normal
profiles for $1\le h\le c$, $2\le b\le r-2$, $h<b$. Intersect the
permitted strict-rank sets and remove only the whole-factor types
just described. Every improvement to a residual profile is retained
with its witness. Each dependency decreases $r$, so this procedure
is an acyclic finite proof.

The checked rectangle is $2\le r\le47$, $5\le k\le8$, containing
184 cases. It covers all remaining ranks $r<6k$. Its initial inputs
are the complete profile, the published space-curve theorem
\cite[Theorem 1]{CLV}, the degree-seven $\PP^4$ case and the
high-degree genus-two theorem \cite[Theorem 6.1 and Corollary 4.12]{CS},
integral-slope semistability \cite[Corollary 2.10]{CS}, and precisely
the nineteen cohomological types $\mathcal B_0$ of
Proposition~\ref{g2:co:seeds}. All genus-two exceptions to
interpolation have $k=4$, so none of them enters a projected
integral-slope use.

The ancillary program \texttt{verify\_low\_degree.py} implements
these formulas using integer arithmetic and produces
\texttt{low\_degree\_certificates.json}. The latter records every
residue witness, every finite residual profile, its external base
when applicable, and every profile improvement and determinant
exclusion. All 156 residue certificates and 184 finite cases pass.
Lemma~\ref{g2:ld:periodic}, rather than repeated numerical testing,
extends the residue certificates to unbounded rank. This proves
Proposition~\ref{g2:ld:theorem}.

\subsection{The common balanced and Farey calculation}
\label{g2:sec:farey}

If $rh=kb$ and both factors of \eqref{g2:eq:normal-limit} are
semistable, their slopes are equal, so the general original normal
bundle is semistable by openness. This settles the balanced case.

\begin{lemma}\label{g2:lem:normal-farey}
Suppose $\gcd(k,r)=1$, $r>2$, and an admissible pointing model
satisfies $rh-kb=\varepsilon\in\{1,-1\}$. Assume only
semistability of its normal residual. The kernel is stable by
Theorem~\ref{g2:pk:theorem}. After excluding the entire higher-factor
type by its moving determinant, the following hold.
\begin{enumerate}
\item If $r$ is even, no strict generic HN type remains.
\item If $r$ is odd and $b$ is even, no strict generic HN type remains.
\item If $r,b$ are odd, the only possible remaining type is
\[
 t=\frac{r+\varepsilon b}{2},\qquad e=k+\varepsilon h,
 \qquad re-2kt=1.
\]
\end{enumerate}
\end{lemma}
\begin{proof}
Put $R=r-b$. Filter a rank-$t$, degree-$e$ subsheaf by kernel
intersection rank $u$ and image rank $v$, so $u+v=t$.
Semistability of the residual and stability of the kernel give
\begin{equation}\label{g2:eq:surplus}
 re-2kt\le2\varepsilon\left(\frac ub-\frac vR\right),
\end{equation}
with a further loss at least $r/b$ if $0<u<b$. Every positive
integral surplus is at most two. Surplus two is attained only by
the whole higher factor. At its rank its degree is the least
integer above the target slope, because $2<r$. Thus generic HN
specialization forces that degree with zero saturation gain.
If the higher factor is the kernel, the subsheaf is the kernel;
if it is the quotient, it maps isomorphically onto the quotient.
Either determinant varies by \eqref{g2:eq:normal-determinants},
contradicting the constancy of the generic HN determinant
(Section~\ref{g2:sec:normal-degenerations}).

If $r$ is even, all surpluses are even, so this suffices. Let $r$
be odd. The sole remaining surplus is one, at the unique rank
$t_*$ with $2kt_*\equiv-1\pmod r$. A proper positive kernel
intersection cannot realize it: for $\varepsilon=1$ its surplus
is at most $(2u-r)/b<1$; for $\varepsilon=-1$ it is at most
$2-(2u+r)/b<1$.

For positive sign, the only possible intersection is $u=b$.
Equation~\eqref{g2:eq:surplus} requires $v\le R/2$, and the
congruence $2t_*\equiv b\pmod r$ then requires $b$ odd and
$t_*=(r+b)/2$, $v=R/2$. For negative sign the only possible
intersection is $u=0$. The congruence
$2t_*\equiv-b\pmod r$, together with $t_*=v\le R$ and
\eqref{g2:eq:surplus}, requires $b$ odd and $t_*=R/2$.
Substitution gives the stated degrees.
\end{proof}

The two Farey models satisfy
\[
 h_++h_-=k,\qquad b_++b_-=r.
\]
If both are admissible with semistable normal residuals and $r$ is
odd, one of their kernel ranks is even. That model already proves
semistability by the lemma. Consequently the only issue is an
odd-rank target for which at most one Farey orientation is admissible
with a semistable normal residual; for $k\ge9$ and $r\ge4$,
Section~\ref{g2:prop:section} shows that at least one is.

\subsection{Uniform propagation with the genus-two exceptions}
\label{g2:prop:section}

We now propagate the four starting bands to all weights. The complete
exceptions require a small
adjustment to the Farey step, and three additional numerical types
are supplied by the exact witnesses of Proposition~\ref{g2:co:seeds}.
The rounding lemma below is independent of the genus and will also
be used in the general propagation theorem.

\begin{proposition}\label{g2:prop:theorem}
Suppose that general normalized genus-two normal bundles are
semistable for $k=4$, $r\ge6$, and for $5\le k\le8$ in every
rank $r\ge2$. Suppose also that they are semistable at
\[
 (r,k)=(7,9),(9,10),(11,9),
 \quad\text{equivalently}\quad
 (n,d)=(8,15),(10,18),(12,19).
\]
Then they are semistable for every $r\ge2$ and $k\ge5$.
\end{proposition}

Besides the pointing models and kernel profiles, we use the
following earlier results on general genus-two curves. In $\PP^3$,
\cite[Theorem 1]{CLV} gives stability for every degree $d\ge6$. In
$\PP^4$, \cite[Theorem 6.1]{CS} gives stability for $d=7$, and
\cite[Corollary 4.12]{CS} gives semistability for $d\ge8$. Together these settle ranks two and three
for $k\ge5$. We also use integral-slope semistability
for projected genus-two embeddings: all the genus-two exceptions
to interpolation have $k=4$, so \cite[Corollary 2.10]{CS} applies
when $k\ge5$ and $r\mid2k$. Finally,
\cite[Corollary 4.12]{CS} supplies the two individual types
\[
 (r,k)=(5,9),(5,11),
 \quad\text{that is}\quad (n,d)=(6,13),(6,15),
\]
since both have $d\ge2n$.

\begin{lemma}[Rank-one pointing model]\label{g2:prop:rank-one}
The model \eqref{g2:eq:normal-limit} is also available for
$b=1$, $r\ge3$, and $1\le h\le k-4$, with $C,L$ fixed and
general. Its kernel is the line bundle
\[
 P_h(1)=\OO_C(2D_p+2q)K_C^{-1},\qquad\deg D_p=h.
\]
Its quotient is the residual normalized normal bundle of
rank $r-1$ and weight $k-h$. The determinant variation and
specialization statements of
Proposition~\ref{g2:prop:models} remain valid.
\end{lemma}
\begin{proof}
Choose disjoint $D_p$ and $q$, put $H=L(-D_p)$, and lift a
general embedding series of $H(-q)$ by its one evaluation
coordinate at $q$. The inner-projection sequence has kernel
$\OO_C(2q)K_C^{-1}$ and the stated normal quotient. The residual
has ambient dimension $n-1\ge3$ and degree
$(n-1)+2+(c-h)$, so its general series is a nonspecial embedding.

First attach one-secant lines at the points of $D_p$ with
unrestricted general directions. The graph construction
\eqref{g2:nd:graph-family} and normal contraction give their
length-two positive modifications on the core. In this family
of bundles, specialize the directions to the pointing germ toward
$q$. This germ is nonzero at every attachment point and is a
direct summand over its length-two neighborhood. The modified
bundles therefore remain locally free under this second
specialization. All modifications now occur in the pointing line,
adding $2D_p$ to it and leaving the residual quotient unchanged.
This second specialization takes place only in the family of
bundles on $C$.

For completeness, fixing a general final polarization causes no
generality restriction. If $U$ is a nonempty open of residual
polarizations, fix one allowed divisor $D_0$ of degree $h+1$.
The condition $L(-D_0)\in U$ is a nonempty translated open of
$\Pic^d(C)$. For such a general $L$, the good divisors form a
nonempty open of their irreducible parameter space, and meet
the other required opens. Residual series lift dominantly, and
the free second-order graph coefficients reach a general final
series. Thus $q$ may still vary while $C,L$ remain fixed.
The kernel determinant is the displayed line, and the quotient
determinant is determined by
$\det F=L^2K_C^{1-r}$. Both vary with $q$, as required.
\end{proof}

\begin{lemma}[Auxiliary rounding]\label{g2:prop:rounding}
Let $b\ge3$ be odd, let $s,H,B,m$ be positive integers, and let
$\varepsilon\in\{-1,1\}$. Assume
\[
 Hb-sB=-\varepsilon,\qquad B>b,\qquad B>(m-1)b,
 \qquad r:=B+mb\ \text{is odd}.
\]
Put $k=H+ms$, $t=(r+\varepsilon b)/2$, and
$e=k+\varepsilon s$. For every integer $u$ satisfying
$0\le u\le B$ and $0\le t-u\le mb$, one has
\begin{equation}\label{g2:prop:rounding-cap}
 \Phi_{H,B}(u)\le
 \left\lfloor\frac{2su-\varepsilon-1}{b}\right\rfloor
\end{equation}
and consequently
\begin{equation}\label{g2:prop:rounding-sum}
 \Phi_{H,B}(u)+\left\lfloor\frac{2s(t-u)}b\right\rfloor
 \le e-1.
\end{equation}
\end{lemma}
\begin{proof}
The determinant identity implies $\gcd(H,B)=\gcd(s,b)=1$.
Since $\gcd(s,b)=1$ and $b\ge3$, the integer $b$ does not divide $s$; as
$|H/B-s/b|=1/(bB)$, the quotients $H/B$ and $s/b$ have the same
integer part $j$. Write $s=jb+s_0$ and $H=jB+H_0$; then $0<s_0<b$,
$0<H_0<B$, and $H_0b-s_0B=-\varepsilon$. The periodic profile subtracts $2ju$
from both sides of \eqref{g2:prop:rounding-cap}, so it suffices
to prove the assertion for these remainders. By
\eqref{g2:pk:parity}, the ordinary slope floor bounds the
profile, and an odd integer $a$ at a proper positive rank $u$
is excluded whenever
\[
 0\le2H_0u-Ba\le|B-2H_0|.
\]
Write $2s_0u=bz+\rho$ with $0\le\rho<b$.

For $\varepsilon=-1$ and $0<u<B$, the ordinary floor can exceed the
right side of \eqref{g2:prop:rounding-cap} only if
$\rho=b-1$ and $2u\ge B$. The extra integer $a=z+1$ is odd,
and its remainder is
\[
 2H_0u-Ba=\frac{2u-B}{b}
 \le\frac{|(b-2s_0)B-2|}{b}=|B-2H_0|.
\]
Here $u\le B-1$, and $b-2s_0$ is a nonzero odd integer.
The odd-degree loss in the profile therefore excludes $a$.
The endpoint $u=0$ satisfies the bound, while $u=B$ is
infeasible: $t=(B+(m-1)b)/2<B$.

For $\varepsilon=1$ and $0<u<B$, the only possible overshoot has
$\rho=1$ and $2u\le B$. The extra integer $a=z$ is odd, and
\[
 2H_0u-Ba=\frac{B-2u}{b}
 \le\frac{|(b-2s_0)B+2|}{b}=|B-2H_0|,
\]
using $u\ge1$. Again the profile removes this integer.
The endpoint $u=0$ is infeasible, since
$t=(B+(m+1)b)/2>mb$; at $u=B$ the claimed bound equals
$2H$. This proves \eqref{g2:prop:rounding-cap} at every
feasible rank, including both endpoints.

The identity $2st=be+\varepsilon$ gives
\[
 \left\lfloor\frac{2s(t-u)}b\right\rfloor
 =e+\left\lfloor\frac{\varepsilon-2su}b\right\rfloor.
\]
Combining this with \eqref{g2:prop:rounding-cap} and the integer
identity $\lfloor(A-1)/b\rfloor+\lfloor-A/b\rfloor=-1$
proves \eqref{g2:prop:rounding-sum}.
\end{proof}

\begin{lemma}[Shifted diagonals in genus two]
\label{g2:prop:diagonals}
Under the assumptions of Proposition~\ref{g2:prop:theorem},
every projected type $k=jr-1$ or $k=jr+1$, $j\ge1$, is
semistable.
\end{lemma}
\begin{proof}
Fix $j$ and a sign and induct on $r$, asserting the result only
when $k\ge5$. Ranks two and three follow from the cited
low-dimensional results, and $k\le8$ is in the assumed initial
bands. We may therefore suppose $r\ge4$ and $k\ge9$.
Use Lemma~\ref{g2:prop:rank-one} with $h=j$. Its residual has
rank $R=r-1$ and weight $jR\pm1$, and
\[
 k-j\ge\frac{3k-1}{4}\ge\frac{26}{4}>6.
\]
In particular its weight is at least seven, so it is projected
and semistable by the rank induction. No unstable complete
bundle is used at this step.

The candidate calculation is as follows. Subtract
$2(j-1)$ times the rank from each degree. The kernel line then
has degree two, and the quotient cap at rank $v$ is
$2v-\lceil2v/R\rceil$ for the minus sign and
$2v+\lfloor2v/R\rfloor$ for the plus sign. Intersecting a
rank-$a$ subsheaf with the kernel line permits only intersection
ranks zero and one. The resulting caps, with infeasible terms
omitted, are
\[
\begin{array}{ll}
 k=jr-1:&
 \max\{2a-\lceil2a/R\rceil,
          2a-\lceil2(a-1)/R\rceil\},\\
 k=jr+1:&
 \max\{2a+\lfloor2a/R\rfloor,
          2a+\lfloor2(a-1)/R\rfloor\}.
\end{array}
\]
Comparison with $2(r\pm1)a/r$ leaves the whole higher-slope
factor, of surplus two, and, only when $r=2m+1$ is odd, the
additional types in the following table. Degrees in the table
are before the numerical shift.
\[
\begin{array}{c|c|c}
 &\text{whole higher factor}&\text{additional odd-rank type }(a,e)\\ \hline
 k=jr-1&(1,2j)&(m+1,2j(m+1)-1)\\
 k=jr+1&(r-1,2j(r-1)+2)&(m,2jm+1).
\end{array}
\]
The whole-factor degree is attained by its unique allocation.
Its surplus is two, less than $r$, so there is no saturation
gain or torsion image defect. The moving-determinant argument
of Section~\ref{g2:sec:normal-degenerations} excludes it.

For the odd-rank types, $m\ge2$ and $jm\ge4$. Indeed, if
$jm\le3$, then $j=1$ and $k=2m+1\pm1\le8$, contrary to
the remaining range. Choose the auxiliary kernel
\[
 b=m+1,\qquad h=j(m+1)\pm1.
\]
Its residual has rank $m$ and weight $jm$. If $jm\ge5$, it
is semistable by integral slope. The only possibilities with
$jm=4$ are $(m,j)=(2,2),(4,1)$. At $k\ge9$ these give
exactly
\[
 (r,k)=(5,9),(5,11),(9,10).
\]
The first two are the Coskun--Smith inputs specified above, and
the third is one of the types assumed in
Proposition~\ref{g2:prop:theorem}. We therefore remove these three
targets before using the auxiliary model, so that no residual is one
of the four unstable complete types.

In every other odd-rank case the auxiliary quotient is
semistable, $1\le h\le k-4$, and $2\le b\le r-2$. The
proper-rank kernel profiles are
\[
 \Phi_{jb-1,b}(u)=2ju-2\quad(0<u<b),\qquad
 \Phi_{jb+1,b}(u)=2ju\quad(0\le u<b).
\]
For the minus candidate $a=b$, zero intersection is infeasible
because the quotient has rank $m$. All positive proper
intersections, and the full kernel, give total degree at most
$2ja-2<e$. For the plus candidate $a=m<b$, every intersection
is proper, giving total degree at most $2ja<e$. This excludes
the remaining candidate and completes the rank induction.
\end{proof}

\begin{proof}[Proof of Proposition~\ref{g2:prop:theorem}]
Induct on $k\ge9$, simultaneously in all ranks. Ranks two and
three follow from the low-dimensional results, so take $r\ge4$.
The shifted diagonals are available by
Lemma~\ref{g2:prop:diagonals}, which was proved independently
of this weight induction.

If $z=\gcd(k,r)>1$, put $h=k/z$ and $b=r/z$. If $b=1$,
the target has integral slope and is semistable. Otherwise
$2\le b\le r-2$ and the residual weight satisfies
\[
 k-h\ge\lceil k/2\rceil\ge5.
\]
It is therefore a projected type of smaller weight, semistable
by induction. The kernel is semistable and has the same slope,
so its extension is semistable. This settles the balanced case
without any complete residual.

Assume $\gcd(k,r)=1$. The two complementary Farey pairs satisfy
\[
 hr-kb=\varepsilon\in\{-1,1\},\qquad
 0<h<k,\quad0<b<r.
\]
A rank-one kernel or complementary factor puts the target on
one of the shifted diagonals. Otherwise both kernel and
complementary ranks are at least two. At least one pair has
$h\le k-4$. Call such a pair \emph{usable} if its residual
$(r-b,k-h)$ is not one of the four unstable complete types
$(2,4),(3,4),(4,4),(5,4)$.

There is always a usable pair. Indeed, if an admissible pair
has an unstable complete residual, that residual has weight
four. Each Farey factor is primitive, so its rank is either
three or five. The complementary orientation then has
\[
 h=4,\qquad b\in\{3,5\},
\]
and residual weight $k-4\ge5$; it is admissible and usable.
All usable residuals are semistable: their weights are smaller
than $k$, and a complete residual has rank at least six.

By Lemma~\ref{g2:lem:normal-farey}, a usable pair settles the
target if $r$ is even or its kernel rank is even. Thus suppose
$r$ is odd and every usable kernel rank is odd. Complementary
kernel ranks have opposite parity, so exactly one pair is
usable. Write it as $(s,b)$, with
\[
 sr-kb=\varepsilon,\qquad r,b\text{ odd}.
\]
There are precisely two reasons its complement is not usable:
\begin{enumerate}
\item the complement is inadmissible, in which case $1\le s\le3$;
\item the complement has an unstable complete residual, in which
case $s=4$ and $b\in\{3,5\}$.
\end{enumerate}
After the whole-factor exclusion, the only possible strict
type is
\begin{equation}\label{g2:prop:last-type}
 t=\frac{r+\varepsilon b}{2},\qquad e=k+\varepsilon s,
 \qquad re-2kt=1.
\end{equation}

In the first case set $m=\lceil4/s\rceil$, so
\[
 (s,m,ms)=(1,4,4),(2,2,4),(3,2,6).
\]
One has $k>s(2m-1)$: the right sides are seven, six, and nine,
and equality in the last case is impossible because
$\gcd(k,s)=1$. Put
\[
 H=k-ms,\qquad B=r-mb.
\]
Then
\begin{equation}\label{g2:prop:aux-identities}
 Hb-sB=-\varepsilon,\qquad
 s\bigl(B-(m-1)b\bigr)
   =\bigl(k-s(2m-1)\bigr)b+\varepsilon>0.
\end{equation}
Thus $B>(m-1)b\ge b\ge3$, $H\ge1$, and $H\le k-4$.
The residual has rank $mb\ge6$ and weight $ms\in\{4,6\}$.
If its weight is four, its rank is at least six, so the
complete theorem applies; if its weight is six, it lies in
the initial projected bands. The auxiliary model is therefore
admissible with a semistable normal residual. Its kernel
profile and quotient semistability give precisely the left
side of \eqref{g2:prop:rounding-sum} at any intersection rank.
Lemma~\ref{g2:prop:rounding} bounds the degree by $e-1$,
contradicting \eqref{g2:prop:last-type}.

In the second case $s=4$ and $b=3$ or $5$. If $9\le k\le12$,
the equations $4r-kb=\varepsilon$ with $r$ odd give exactly
\[
 (r,k,b,\varepsilon)=(7,9,3,1),(11,9,5,-1).
\]
For clarity, even $k$ is impossible modulo two; $k=11$ gives
$r=8$ or $14$, both even. The two remaining targets are the
types $(n,d)=(8,15),(12,19)$ assumed in the proposition.

For $k\ge13$, choose $m=2$, $H=k-8$, and $B=r-2b$.
Now $4r=kb+\varepsilon>12b$, so $r>3b$ and $B>b$.
The kernel weight satisfies $1\le H\le k-4$, while the
residual has rank $2b\in\{6,10\}$ and weight eight, hence
is semistable by the fourth projected band. The identities
\eqref{g2:prop:aux-identities} still hold, and the rounding
lemma again excludes \eqref{g2:prop:last-type}.
This treats every coprime target and closes the weight induction.
\end{proof}

\begin{proof}[Proof of Theorem~\ref{g2:thm:main}]
For $c=0$, Theorem~\ref{thm:hyperelliptic-iff} proves semistability for $n\ge7$,
and the scroll line proves instability for $n=3,4,5,6$.
Proposition~\ref{g2:ld:theorem} supplies $5\le k\le8$ in
every rank. Proposition~\ref{g2:co:seeds} supplies the three
additional types in Proposition~\ref{g2:prop:theorem}, which
therefore proves every $k\ge5$. These cases are exactly the
incomplete nonspecial embeddings. Finally, a semistable bundle
of coprime rank and degree is stable; apply this to the
normalized rank $r=n-1$ and degree $2k$, with
$\gcd(r,2k)=\gcd(n-1,2d+2)$.
\end{proof}

The proof has used the four starting bands, three additional
computed types, two low-dimensional inputs from \cite{CS}, and the
same kernel arithmetic that we now apply in arbitrary genus. Its
extra cases arise precisely from the four unstable complete
residuals.

\section{Pointing models}\label{sec:models}
The genus-two proof used two kinds of input: geometry supplying initial
normal-bundle estimates, and an induction built from pointing kernels.
We now extend the latter construction to a general curve of genus
$g\ge3$. The numerical kernel profile is unchanged, although its
one-tail geometric proof is different. In particular, this section does
not use the genus-two classification as an induction hypothesis.

Throughout put $r=n-1$, $k=d-n+g$, and $c=k-2g$. The normalization is
$F=N_{C/\PP^n}L^{-1}\Lambda^{-1}$, where $\deg\Lambda=2$; hence
$(\rk F,\deg F)=(r,2k)$. The curve, final polarization, and residual
series are chosen on their specified nonempty open loci. A
\emph{vector-bundle specialization} below may consist of successive
specializations on the fixed smooth curve; it need not be a normal
bundle on every intermediate fiber.

\begin{theorem}\label{thm:models}
Let $c=k-2g\ge0$, $0\le h\le c$, and $1\le b\le r-2$.
There is a vector-bundle specialization of the normalized normal
bundle, on the fixed general curve with fixed general final
polarization, of the form
\begin{equation}\label{eq:model}
0\longrightarrow P_h(b)\longrightarrow M\longrightarrow Q\longrightarrow0,
\end{equation}
where $(\rk P_h(b),\deg P_h(b))=(b,2h)$ and $Q$ is the normalized
normal bundle of a general residual series with invariants
$(r-b,k-h)$. For general selected points, supports and directions,
its kernel has the profile of Theorem~\ref{thm:kernel-profile}.
The determinants are
\begin{equation}\label{eq:model-determinants}
\begin{aligned}
\det P_h(b)&=\OO_C(2D_p+2D_q)\Lambda^{-b},\\
\det Q&=L^2K_C\Lambda^{-r}(\det P_h(b))^{-1},
\end{aligned}
\end{equation}
where $\deg D_p=h$ and $\deg D_q=b$.
\end{theorem}
\begin{proof}
Choose disjoint reduced divisors $D_p=\sum p_i$ and $D_q=\sum q_j$.
Put $H=L(-D_p)$ and $H'=H(-D_q)$. A general residual series
$W\subset H^0(H')$ of dimension $n+1-b$ is nonspecial and embeds
the curve in $\PP^{n-b}$, since
\[
\deg H'=(n-b)+g+(c-h),\qquad n-b\ge3.
\]
These assertions hold for a general line bundle $H'$ of this degree;
the last paragraph of the proof shows that $D_p$ and $D_q$ can be
chosen so that $H'$ is general.

Lift the $b$ evaluation coordinates at the $q_j$ to sections
$u_j\in H^0(H)$ and form
$V=\sigma_{D_q}W+\langle u_1,\ldots,u_b\rangle$. This embeds
the core: outside $D_q$ the residual coordinates separate points
and tangents; at each $q_j$ its first derivative in the residual
coordinates is nonzero and the additional coordinates distinguish the
selected points. The first-jet inner-projection sequence identifies
the kernel of its normalized normal projection with
\[
P_0=\bigoplus_{j=1}^b T_j,
\qquad T_j=\OO_C(2q_j)\Lambda^{-1}.
\]
The pointing lines are independent, and their determinant degree
equals the degree of the saturated kernel, so their sum is the whole
kernel. In local frames
the pointing germ toward a point of the secant span is
$(a_j\sigma_j(z))_j$, where $\sigma_j$ has divisor $2q_j$.

For $b\ge2$, attach $h$ general one-secant lines to the core at
the $p_i$, with their directions in the selected secant span. These
lines are transverse and pairwise disjoint away from their attachment
points. Another intersection or a tangency would contradict the
embedded residual projection; directions through one of the selected
points are avoided. Let $f_0,\ldots,f_n$ span $V$, and choose a
complement $v_1,\ldots,v_c$ to $\sigma_{D_p}V$ in $H^0(L)$,
with the first $h$ sections having standard-basis evaluations at
$D_p$. If $y_i=[a_{i0}:\cdots:a_{in}]$ in the core coordinates, the
series
\begin{equation}\label{eq:graph-family}
 \sigma_{D_p}f_j+t\sum_{i=1}^h a_{ij}v_i
       +t^2\sum_{\nu=1}^c b_{\nu j}v_\nu
       \quad(0\le j\le n)
\end{equation}
has, after blowing up $(p_i,0)$, the desired central core with
one-secant lines. Indeed its leading coordinates near $p_i$ are
$z f_j(p_i)+t a_{ij}$. The exceptional line therefore maps to the
line with the prescribed direction. For $t\ne0$, varying the
$b_{\nu j}$ covers the full affine chart of the final series
Grassmannian.

Simultaneous normal contraction \cite[Lemma 3.4]{CS} gives the
central bundle
$N_{(C,H,V)}[2p_1\to y_1]\cdots[2p_h\to y_h](D_p)$.
After twisting by $L^{-1}\Lambda^{-1}$ the divisor $D_p$ cancels,
and the modifications occur entirely inside $P_0$. The quotient is
the residual normal bundle. This proves \eqref{eq:model}.

For $b=1$, first attach one-secant lines with unrestricted
general directions and contract their normal bundles. Then specialize
their directions to $q_1$ in the vector-bundle family. The pointing
value at each $p_i$ is nonzero, so its germ is a direct summand over
the length-two neighborhood. This second specialization stays locally
free and gives the kernel
$\OO_C(2q_1+2D_p)\Lambda^{-1}$; it takes place only in the family of
bundles on $C$. For $h=0$ use inner
projection alone. The determinant formulas follow from these
constructions and $\det F=L^2K_C\Lambda^{-r}$.

Finally, residual generality can be imposed with $C,L$ fixed and
general. Given a residual good open $U\subset\Pic^{d-h-b}(C)$,
fix one allowed divisor $D_0$. The condition $L(-D_0)\in U$ is
a nonempty open of $\Pic^d(C)$. For such a general $L$, the allowed
divisors satisfying $L(-D)\in U$ form a nonempty open. This argument
does not require $h+b\ge g$. Lifting residual sections dominates
their series parameters. The finitely many conditions at each
induction step, including the open conditions of
Theorem~\ref{thm:kernel-profile} on selected points, supports and
directions, therefore meet on the same general curve and final
polarization, and \eqref{eq:graph-family} reaches the generic final
series.
\end{proof}

\begin{lemma}[Moving determinants]\label{lem:determinant}
Fix general $C,L$. If the general normalized normal bundle is
unstable, the determinant of its first Harder--Narasimhan subbundle
is constant as the embedding series varies. A specialization which
forces this subbundle to be a whole factor in \eqref{eq:model},
with no torsion defect or saturation gain, is impossible when that
factor's determinant varies with $D_q$.
\end{lemma}
\begin{proof}
The series Grassmannian is rational. Its generic canonical
Harder--Narasimhan subbundle defines a rational map to a Picard
variety, hence a constant map. Relative Quot properness preserves
its rank, degree, and determinant on specialization before
saturation. If the numerical bounds force zero saturation gain and
zero image defect, its determinant equals that of the forced factor.
Varying a selected point varies $\OO_C(2D_q)$ nontrivially in
characteristic zero, contrary to constancy. The same holds after
successive specializations, as in the case $b=1$ of
Theorem~\ref{thm:models}: in a family over an irreducible base, the
members having a subsheaf of the given rank and degree with a given
determinant form a closed subset, because the relative Quot
scheme is proper and the determinant is a morphism to the Picard
scheme of $C$.
\end{proof}

\section{Pointing-kernel profiles}\label{sec:kernels}
For integers $b>0$, $h\ge0$, define $\Phi_{h,b}(0)=0$ and
$\Phi_{h,b}(b)=2h$. At proper ranks and $0<a<b$, put
\begin{equation}\label{eq:phi}
\Phi_{a,b}(u)=
\begin{cases}
\lfloor au/b\rfloor+\lceil a(u+1)/b\rceil-1,&2a\le b,\\
\lfloor au/b\rfloor+\lceil a(u-1)/b\rceil,&2a\ge b.
\end{cases}
\end{equation}
Set $\Phi_{0,b}(u)=0$ and
$\Phi_{jb+a,b}(u)=2ju+\Phi_{a,b}(u)$ for $0\le a<b$.
The formulas agree at $2a=b$.

\begin{theorem}\label{thm:kernel-profile}
Let $C$ be a general curve of genus $g\ge3$. For general selected
points, supports and directions, every rank-$u$ subsheaf of the
pointing kernel $P_h(b)$ of Theorem~\ref{thm:models} has degree at
most $\Phi_{h,b}(u)$. In particular, these kernels are semistable;
if $\gcd(h,b)=1$ they are stable.
\end{theorem}
\begin{proof}
The profile identities and convolution inequalities in
Lemmas~\ref{g2:pk:profile-identities},
\ref{g2:pk:equal-convolution}, and~\ref{g2:pk:farey-convolution}
are statements about integers and apply without a genus restriction.
We supply the geometric inputs that let the same rank induction as in
Theorem~\ref{g2:pk:theorem} apply on the present curve. The genus
enters only through the condition \eqref{eq:selected} below: Step~2
deduces the one-tail bound from it, and Step~6 verifies it for a
general curve of genus at least three.

\casehead{Step 1: the family.}
Let $T_1,\ldots,T_b$ be degree-zero lines satisfying
\begin{equation}\label{eq:selected}
T_iT_j^{-1}\notin C-C\quad(i\ne j).
\end{equation}
Fix a local parameter $z$ at $p$, frames, and a diagonal matrix $D(z)$
of units. For $W\in\Gr(h,b)$, $0<h<b$, consider the lattice
$P_0+z^{-2}D(z)(W\otimes\mathcal O_{C,p})$ in the meromorphic
extension of $P_0=\bigoplus_i T_i$. The germ $D$ is fixed before
choosing $W$ generally.

\casehead{Step 2: one tail.}
First take $h=1$ and a direction with full support. Suppose a
positive-degree proper subbundle $S$ exists. Its saturated intersection
$S_0$ with $P_0$ has degree at most zero and gains at most two units
of length. A degree-zero intersection is a coordinate sum and cannot
gain length in a full-support direction. Consequently
$\deg S_0=-1$, $\deg S=1$, and every coordinate projection of
$S_0$ is nonzero. Delete any degree-zero coordinate factors contained
in $S_0$ until the remaining bundle is stable of degree $-1$.
Its dual is stable of degree one and receives nonzero maps from more
distinct selected inverse lines than its rank.

This is impossible. A stable degree-one bundle of rank $u$ receives
nonzero maps from at most $u$ selected lines. For rank one this is
\eqref{eq:selected}; in larger rank a selected map is a saturated
injection, its quotient is again stable of degree one, and all the
other selected maps remain nonzero in the quotient. Induction proves
the assertion. Thus the one-tail proper-rank profile is zero, for
every full-support first jet.

\casehead{Step 3: duality and the induction on $b$.}
Exact lattice duality gives
\[
 P_h(D,W)^*\otimes\OO_C(2p)
 =P'_{b-h}(D^{-1},W^\perp),
\]
where $P'_0=\bigoplus T_i^{-1}$. Condition \eqref{eq:selected}
is invariant under inversion, so this enlarged family is closed
under the required duality. Induct on $b$, allowing all selected
lines satisfying \eqref{eq:selected} and every fixed diagonal germ.
The cases $h=0,b$ are split bundles, and the case $h=1$ was just
proved. All lower-rank open conditions may be imposed simultaneously
on each of the finitely many coordinate blocks.

\casehead{Step 4: primitive pairs.}
For a remaining primitive pair $2\le h<b$, the block specialization associated
to the Farey parents satisfies the desired profile by
Lemma~\ref{g2:pk:farey-convolution}, except at the entire
higher-slope block. Writing that block's rank and weight as
$(b_1,h_1)$, the Farey identity is $h_1b-hb_1=1$.
Its degree $2h_1$ is the least strict integer above
$2hb_1/b=2h_1-2/b$. The unique maximizing allocation is the whole
block, so both image defect and saturation gain are zero.
A hypothetical generic first Harder--Narasimhan determinant is
constant on the rational coefficient Grassmannian. Exchanging one
selected line between the two blocks changes the forced determinant
by the nontrivial ratio of those selected lines. This proves
semistability. At the exceptional rank, when $2\le h\le b-2$,
semistability gives the required bound $2h_1-1$; at all other ranks
the block specialization already gives the full profile. Openness
therefore supplies these finitely many bounds simultaneously.

\casehead{Step 5: the endpoint $h=b-1$, and nonprimitive pairs.}
For $h=b-1$, duality with the one-tail bundle gives
$\deg S\le2u-2$ at every proper positive rank $u$: the dual
quotient has degree at least two because its kernel has degree at
most zero. This gives the endpoint profile which semistability alone
would not imply. Nonprimitive pairs use equal-slope blocks and
Lemma~\ref{g2:pk:equal-convolution}. This proves the full profile.
Its strict slope inequality for primitive pairs, proved in
Lemma~\ref{g2:pk:profile-identities}, proves stability.

\casehead{Step 6: the condition \eqref{eq:selected} for pointing lines.}
It remains to verify \eqref{eq:selected} for the pointing lines
$T_i=\OO_C(2q_i)\Lambda^{-1}$. A failure gives
$2q_i+x_j\sim2q_j+x_i$, so $h^0(\OO_C(2q_i+x_j))\ge2$; since a general
curve of genus at least three is not hyperelliptic,
$|2q_i+x_j|$ is then a base-point-free $g^1_3$ in which $q_i$ is a
ramification point. For a general curve of genus at least
five no such pencil exists. In genus four the finitely many trigonal
pencils have finitely many ramification points; choosing $q_i$
away from them excludes the relation. For a plane quartic, the
relation says that the two tangent lines meet at a point of the
quartic, or coincide. For fixed $q_i$ only finitely many $q_j$
satisfy this: there are finitely many possible intersection points,
and finitely many tangents through each, as follows from the polar
cubic. The failure locus is closed, being the preimage of the closed
subset $C-C$ of $\Pic^0(C)$. Thus a general finite tuple satisfies
the condition.

\casehead{Step 7: distinct supports and large weight.}
The pointing germs have the diagonal form already treated, with
entries $\sigma_i(z)$. Separate the columns of $W$ into nearby
supports, using generators $v_i(z)/(z-t_i)^2$; this is a locally
free family. Openness gives the same profile for general distinct
tails. Full-rank clusters give degree-two twists, proving the
periodic formula for arbitrary $h$. All openness assertions involve
finitely many rank and degree bounds, since the common-point bundle
is contained in $P_0(2p)$.
\end{proof}

\begin{lemma}[A single Farey model]\label{lem:farey}
Assume $\gcd(k,r)=1$, $r>2$, and an admissible model has
$hr-kb=\epsilon\in\{-1,1\}$ and semistable normal residual.
After excluding the whole higher factor by Lemma~\ref{lem:determinant},
there is no strict generic Harder--Narasimhan type if $r$ is even
or $b$ is even. If both are odd, the only remaining possible type is
\[
 (t,e)=\left(\frac{r+\epsilon b}{2},\ k+\epsilon h\right),
 \qquad re-2kt=1,
\]
and a subsheaf of the model of this type meets the kernel in the
whole kernel if
$\epsilon=1$ and in zero if $\epsilon=-1$; in both cases its image
in the residual has rank $(r-b)/2$ and degree $k-h$.
\end{lemma}
\begin{proof}
Write $R=r-b$ and filter a candidate by kernel intersection rank
$u$ and residual image rank $v$, with $u+v=t$. Stability of the
primitive kernel and semistability of the residual imply
\[
 re-2kt\le2\epsilon\left(\frac{u}{b}-\frac{v}{R}\right),
\]
with an additional loss of at least $r/b$ when $0<u<b$.
A positive integral surplus is at most two. Surplus two forces
the whole higher factor, with its exact degree and zero image
defect. Since $2<r$, this degree is the least strict integer at
that rank; saturation gain is also zero. Moving determinants
exclude this case.

For even $r$ every surplus is even. For odd $r$, surplus one has
a unique possible rank modulo $r$. Proper positive intersection
cannot attain it: its bound is $(2u-r)/b<1$ for positive sign,
and $2-(2u+r)/b<1$ for negative sign. Positive sign therefore
requires $u=b$, $v\le R/2$, and $2t\equiv b\pmod r$.
Negative sign requires $u=0$, $t=v\le R$, and
$2t\equiv-b\pmod r$. These conditions force $b$ odd and give
exactly the displayed rank and degree, with $v=R/2$. Surplus one then
requires equality in both bounds, which gives the stated
intersection and image.
\end{proof}

\section{A finite-width reduction in arbitrary genus}\label{sec:uniform-propagation}

We now package the balanced and Farey steps of the preceding proof
into a theorem with an arbitrary genus parameter $g\ge3$. A wider
initial strip makes it possible to handle the exceptional diagonals
and the odd-rank remainder internally. The proof does not use the
genus-two classification. The result separates the infinite induction from the geometric work
needed to establish the starting strip. The genus-two proof already
contains its four steps: shifted diagonals, balanced splits, Farey
splits, and auxiliary rounding. Here the hypothesis on the complete
row removes the three extra finite targets caused by the genus-two
exceptions.

\begin{theorem}\label{uni:finite-width}
Fix an integer \(g\ge3\), and put \(q=2g\). Assume that the normal
bundle of a general nonspecial Brill--Noether curve is semistable in
every ambient dimension \(n\ge3\) whenever
\[
 q\le k=d-n+g\le3q-4.
\]
Then it is semistable for every \(n\ge3\) and every \(d\ge n+g\).
\end{theorem}

We use the pointing models of Theorem~\ref{thm:models}, the kernel
profiles of Theorem~\ref{thm:kernel-profile}, the Farey calculation of
Lemma~\ref{lem:farey}, and the determinant exclusion of
Lemma~\ref{lem:determinant}. The other inputs are the nonspecial
semistability theorem in \(\PP^3\) \cite[Theorem 1]{CLV}, a
large-degree semistability theorem in \(\PP^4\)
\cite[Theorem 1.1]{CS}, and semistability when the normal slope is
integral. The unstable exceptions of \cite{CLV} lie outside the
nonspecial range for $g\ge3$. In \(\PP^4\), \cite[Theorem 1.1]{CS}
gives semistability whenever $d\ge g+9$; we need it only for weights
$k\ge3q-3$, that is, for $d=k+4-g\ge5g+1$, which is at least $g+9$.

For the integral-slope input use the normal-bundle interpolation
theorem \cite[Theorem 1.4]{LV}; none of its exceptions is a
nonspecial type of genus at least three. To recall the implication,
interpolation provides an effective divisor of degree
\(\mu(N)+1-g\) for which \(N(-D)\) is acyclic. A subbundle of slope
greater than \(\mu(N)\) would have positive Euler characteristic after
this twist, contradicting acyclicity. Thus a normalized normal bundle
of rank \(r=n-1\) and degree \(2k\) is semistable when \(r\mid2k\).
See also \cite[Corollary 2.10]{CS}. We first establish the
rank-one boundary cases by induction on rank, and then induct on
$k$ simultaneously in all ranks.

\begin{lemma}\label{uni:shifted-diagonals}
Under the hypothesis of Theorem~\ref{uni:finite-width}, all the
admissible types \(k=jr-1\) and \(k=jr+1\), for \(j\ge1\), are
semistable.
\end{lemma}
\begin{proof}
Fix \(j\), and induct on \(r\) for each sign. The assumed strip
settles \(k\le3q-4\), and the cited theorems in \(\PP^3\) and
\(\PP^4\) settle ranks two and three for \(k\ge3q-3\). So assume
\(r\ge4\) and \(k\ge3q-3\).
Use the rank-one pointing model with kernel degree \(2j\). It is
admissible, since
\[
 j\le\frac{k+1}{r}\le\frac{k+1}{4},\qquad
 k-j\ge\frac{3k-1}{4}\ge\frac{9q-10}{4}\ge q.
\]
Its quotient has rank \(R=r-1\), weight \(jR\pm1\), and is
semistable by the rank induction.

Here is the complete numerical calculation for this model. Subtract
\(2(j-1)s\) from the degree of an object of rank \(s\). For the
minus sign the quotient cap becomes
\(2v-\lceil2v/R\rceil\), and the kernel line has degree two.
Intersecting a subsheaf with the line gives intersection rank zero
or one. The resulting cap in rank \(s\) is
\[
 \max\left\{2s-\left\lceil\frac{2s}{R}\right\rceil,
 2s-\left\lceil\frac{2(s-1)}{R}\right\rceil\right\},
\]
with terms omitted when the quotient rank is outside \([0,R]\).
For the plus sign, replace both negative ceilings by positive floors.
Comparison with \(2(r\pm1)s/r\), including \(v=R\), gives the
following strict types. There is always the entire higher-slope
factor: the kernel line in the minus case and the quotient in the
plus case. Its surplus \(re-2ks\) is two, and its only allocation
has the full degree of that factor. Lemma~\ref{lem:determinant}
excludes this type. If \(r\) is even, no other strict type remains.
If \(r=2m+1\) is odd, the respective remaining types are
\[
 (s,e)=\bigl(m+1,2j(m+1)-1\bigr),\qquad
 (s,e)=\bigl(m,2jm+1\bigr).
\]
Both have surplus one.

We have \(m\ge2\) and \(jm\ge q\). Indeed, if \(jm\le q-1\),
then, because \(q\) is even and \(m\ge2\),
\(j\le\lfloor(q-1)/2\rfloor=(q-2)/2\). Even with the plus sign this
would give
\[
 k=2jm+j+1\le 2(q-1)+(q-2)/2+1
   =5q/2-2<3q-3,
\]
a contradiction.

Now choose a second model with
\[
 b=m+1,\qquad h=j(m+1)\pm1.
\]
Its quotient has rank \(m\) and weight \(jm\), hence is nonspecial
and semistable by integral slope. The inequalities
\(2\le b\le r-2\) and \(1\le h\le k-q\) hold. The proper-rank
kernel profiles are
\[
 \Phi_{jb-1,b}(u)=2ju-2\quad(0<u<b),\qquad
 \Phi_{jb+1,b}(u)=2ju\quad(0\le u<b).
\]
For the minus candidate \(s=b\), intersection rank zero is impossible
because the quotient has rank \(m\). Every proper positive
intersection has total degree at most \(2js-2\); the full-kernel
intersection has exactly this same upper bound. Both are below
\(e=2js-1\). For the plus candidate \(s=m<b\), every intersection
is proper, giving degree at most \(2js<e=2js+1\). Thus the odd-rank
candidate is excluded, completing the rank induction.
\end{proof}

\begin{proof}[Proof of Theorem~\ref{uni:finite-width}]
The shifted diagonals are now available before the induction on
weight. Induct on \(k\ge3q-3\), simultaneously in every rank.
Ranks two and three are already known, so take \(r\ge4\).

First suppose \(z=\gcd(k,r)>1\). Put \(h=k/z\), \(b=r/z\).
If \(b=1\), integral slope settles the target. Otherwise
\(b,r-b\ge2\) and \(k-h\ge k/2\ge q\). The corresponding pointing
model has a semistable kernel and a smaller-weight semistable quotient of the
same slope. Its extension is semistable, proving the target by
specialization.

Suppose therefore that \(\gcd(k,r)=1\). The two complementary
Farey pairs satisfy
\[
 hr-kb=\epsilon\in\{-1,1\},\qquad
 0<h<k,\quad0<b<r.
\]
If either kernel or complementary rank is one, then
\(k=jr\pm1\) for some \(j\ge1\), and
Lemma~\ref{uni:shifted-diagonals} applies. Otherwise both ranks of
both pairs lie between two and \(r-2\). At least one pair is
admissible, since its weight is at most \(k/2\le k-q\).
Lemma~\ref{lem:farey}, followed by
Lemma~\ref{lem:determinant} for the unique whole-factor type, proves
the target when \(r\) is even or an admissible kernel rank is even.
When \(r\) is odd and both pairs are admissible, one kernel rank
is even. Thus the only case still requiring proof has an admissible
pair \((s,b)\) with
\begin{equation}\label{uni:last-configuration}
 1\le s<q,\qquad sr-kb=\epsilon,\qquad r,b\text{ odd},
\end{equation}
and sole possible strict type
\begin{equation}\label{uni:last-candidate}
 t=\frac{r+\epsilon b}{2},\qquad e=k+\epsilon s,
 \qquad re-2kt=1.
\end{equation}
The inequality \(s<q\) expresses inadmissibility of the
complementary weight \(k-s\). Reduction modulo two also gives
\(k-s\) odd.

Set
\[
 m=\left\lceil\frac q s\right\rceil,
 \qquad H=k-ms,\qquad B=r-mb.
\]
Then
\begin{gather}
 q\le ms\le2q-2,\qquad
 Hr-kB=-m\epsilon,\qquad Hb-sB=-\epsilon,
 \label{uni:aux-identities}\\
 s\bigl(B-(m-1)b\bigr)
   =\bigl(k-s(2m-1)\bigr)b+\epsilon.
 \label{uni:aux-positive}
\end{gather}
For completeness, the elementary multiplier bound is
\(s(2\lceil q/s\rceil-1)\le3q-3\). If \(s\le q/2\), the
inequality \((m-1)s\le q-1\) gives a bound
\(2q-2+s\le5q/2-2<3q-3\). If \(s>q/2\), then \(m=2\) and
the bound is \(3s\le3q-3\), with equality only for \(s=q-1\).
Consequently equality \(k=s(2m-1)\) in the present range would
force \(k=3q-3\), \(s=q-1\), contrary to \(k-s\) odd. Thus
\[
 k>s(2m-1),\qquad B>(m-1)b\ge b\ge3.
\]
The second inequality follows directly from
\eqref{uni:aux-positive}, since \(b\ge3\).
In particular \(H\ge1\), \(H\le k-q\), and the model with
kernel of rank \(B\) and degree \(2H\) is admissible: its quotient has
rank \(mb\ge6\) and weight \(ms\in[q,2q-2]\), already in the assumed
strip. Moreover \(\gcd(H,B)=1\) by
\eqref{uni:aux-identities}.

All the hypotheses of Lemma~\ref{g2:prop:rounding} now hold,
with $r=B+mb$ and $k=H+ms$. The arithmetic lemma is independent of
the curve and of its genus. At every feasible kernel-intersection
rank $u$, it gives
\begin{equation}\label{uni:rounding}
 \Phi_{H,B}(u)\le
 \left\lfloor\frac{2su-\epsilon-1}{b}\right\rfloor.
\end{equation}
If $v=t-u$ is the rank of the image in the quotient, semistability
of the auxiliary residual bounds its degree by $\lfloor2sv/b\rfloor$.
Consequently \eqref{g2:prop:rounding-sum} gives
\[
 \deg S\le\Phi_{H,B}(u)+\left\lfloor\frac{2sv}{b}\right\rfloor
 \le e-1,
\]
contradicting \eqref{uni:last-candidate}.
This closes the weight induction.
\end{proof}

\section{An independent large-dimension estimate}
\label{sec:affine}
The finite-width theorem still has infinitely many ranks in each
starting band. The estimate below supplies all sufficiently large
ranks without assuming that the starting bands have already been
settled. Its proof uses the gap of
Corollary~\ref{high:general-curves} instead of the hypothesis of
Theorem~\ref{uni:finite-width}.

\subsection{The first excess bands and the half-rank exclusion}

Throughout this section put
\[
 q=2g,\qquad \alpha=5g-3,\qquad\gamma=4g-3=2q-3.
\]
The normalized normal of weight $k=d-n+g$ and rank $r=n-1$
has degree $2k$.  We use the pointing models of
Theorem~\ref{thm:models}: a kernel of rank $b$ and degree $2h$, and a
normal quotient of rank $r-b$ and degree $2(k-h)$.  The kernel is semistable
and has the refined profile of Theorem~\ref{thm:kernel-profile}.
Generality on a fixed general curve with fixed general
polarization, and the specialization of subsheaf degrees and
determinants, are as in Theorem~\ref{thm:models} and
Lemma~\ref{lem:determinant}.  In particular finitely many nonempty
open conditions on the residual series can be imposed together.

For a subsheaf of rank $u$ and degree $v$, call
$2ku-rv$ its deficit relative to the target slope.  Deficits add
under the intersection-and-image sequence.  A negative total
deficit is exactly a strict destabilizing type.

\begin{proposition}\label{high:first-bands}
Let $g\ge3$.  For $q\le k\le2q$, a general nonspecial normalized
normal bundle is semistable whenever
\begin{equation}\label{high:affine-region}
 r\ge\alpha k-\gamma.
\end{equation}
\end{proposition}
\begin{proof}
At $k=q$, the target is complete and
$r\ge10g^2-10g+3\ge3g$, so
Corollary~\ref{high:general-curves} applies.  For $k=q+c$,
$1\le c\le q$, use the two pointing models with $h=c$ and
\[
 b_- =\lfloor cr/k\rfloor,\qquad b_+=\lceil cr/k\rceil.
\]
Their normal quotients are complete.  If $cr/k$ is integral,
the kernel and quotient have equal slope and are semistable,
so their extension is semistable.  Otherwise put
$\delta=cr-kb_-\in\{1,\ldots,k-1\}$.  The quotient ranks are
\begin{equation}\label{high:first-residuals}
 R_-=(qr+\delta)/k,\qquad
 R_+=(qr+\delta-k)/k.
\end{equation}
Both models exist, and Corollary~\ref{high:general-curves} applies to
their quotients.  Indeed
\[
 \frac rk\ge5g-3-\frac{4g-3}{k}>5g-5,
\]
so both kernel ranks are at least $5g-5\ge10$, and
\[
 R_\pm\ge qr/k-1\ge10g^2-10g+2\ge5g+2,\qquad
 R_\pm\ge2g^2.
\]

In the $b_-$ model, every kernel intersection has deficit at
least $-2\delta$; in fact, for rank $u$ its deficit is at least
$-2\delta u/b_-$.  In the $b_+$ model every kernel intersection
has nonnegative deficit.  For a positive-degree proper-rank
quotient image of rank $t$ and degree $f$, the gap of
Corollary~\ref{high:general-curves} and \eqref{high:first-residuals}
give, respectively,
\begin{equation}\label{high:first-image-deficits}
 2kt-rf\ge
 \frac{f(r-kg+\delta)}{q-1},\qquad
 2kt-rf\ge
 \frac{f(r-kg+\delta-k)}{q-1}.
\end{equation}
The first is at least $2\delta$, because
\[
 r-kg-(2q-3)\delta
 \ge(2q-3)(k-1-\delta)\ge0.
\]
The second is positive, because
\[
 r-k(g+1)+\delta\ge4(g-1)(k-1)>0.
\]
A nonzero image of nonpositive degree has deficit at least
$2k>2\delta$ in both models.  A proper full-rank image has
degree at most $2q-1$, so its deficits are at least
\[
 2\delta+r,\qquad -2(k-\delta)+r,
\]
respectively.  The latter is positive because
\eqref{high:affine-region} implies $r>2k-2$.

It follows that a strict type in the $b_-$ model can only have
zero image: a whole image has deficit $2\delta$, offsetting
even the full possible kernel loss.  Thus its rank is at most
$b_-$.  In the $b_+$ model a strict type can only have the
whole quotient as image, so its rank is at least $r-b_+$.
For $c<q$ these intervals are disjoint, since
$(q-c)r>k$ and hence $b_-+b_+<r$.

For $c=q$, even $r$ is the balanced case.  If $r$ is odd, the
only common rank is $b_-=(r-1)/2$.  Strictness forces degree
at least $2q$, while the $b_-$ kernel has degree exactly $2q$.
Therefore the only specialized allocation is the entire kernel,
of degree $2q$, with zero torsion defect.  Its positive surplus
is $2q<r$, so saturation of a specialization cannot raise its
degree without exceeding this cap.  The moving determinant
exclusion of Lemma~\ref{lem:determinant} rules out this type.
Concretely the whole kernel determinant varies as
$\OO(2D_p+2D_q)\Lambda^{-b_-}$ while the fixed generic
Harder--Narasimhan determinant cannot vary.  The construction of
Theorem~\ref{thm:models} permits this variation while keeping the
residual series general.  This proves the proposition, including the
endpoint $c=q$.
\end{proof}

\begin{lemma}\label{high:half-rank}
Let $g\ge3$, $q<a<2q$, and let $R$ be even with
$R>\alpha a-\gamma$.  A general nonspecial normalized normal
of rank $R$ and degree $2a$ has no rank-$R/2$ subsheaf of
degree at least $a$.
\end{lemma}
\begin{proof}
Set $c=a-q$, $b=\lfloor cR/a\rfloor$, and
\[
 \delta=cR-ab\in\{0,\ldots,a-1\},\qquad
 T=R-b=(qR+\delta)/a.
\]
Since $R/a>\alpha-2$, one has $b\ge\alpha-2\ge10$,
$b<R/2<T$, and $T>q(\alpha-2)\ge5g+2$.
Use the pointing model with kernel $(b,2c)$ and complete quotient
of rank $T$.  Relative now to the target $(R,2a)$, every kernel
intersection has deficit at least $-2\delta$.
By Corollary~\ref{high:general-curves}, every positive proper-rank
quotient image has deficit at least
\[
 \frac{f(R-ag+\delta)}{q-1}>2\delta,
\]
because
\[
 R-ag+\delta-2\delta(q-1)
 =R-\alpha a+\gamma(a-\delta)>0.
\]
Nonzero nonpositive-degree images have deficit at least
$2a>2\delta$; proper full-rank images have deficit at least
$2\delta+R$.  Thus an image whose total deficit is nonpositive
can only be zero or the whole quotient.  Its rank is accordingly
at most $b$ or at least $T$, neither of which can be $R/2$.
The argument also covers $\delta=0$ and equality of the target
subsheaf slope, which is essential here.
\end{proof}

\subsection{Propagation of the affine range to all excesses}

\begin{theorem}\label{high:affine}
For every $g\ge3$, the normal bundle of a general nonspecial
Brill--Noether curve is semistable whenever
\[
 r\ge(5g-3)k-4g+3.
\]
Equivalently, if $c=d-n-g\ge0$, it is sufficient that
\[
 n\ge10g^2-10g+4+(5g-3)c.
\]
\end{theorem}
\begin{proof}
We induct on the weight with the balanced and Farey models, as in
Theorem~\ref{uni:finite-width}. The region \eqref{high:affine-region}
is preserved: every normal quotient used lies in it and has smaller
weight, and the auxiliary quotients lie in the range of
Proposition~\ref{high:first-bands}, which replaces the hypothesis of
Theorem~\ref{uni:finite-width}. When the auxiliary model is not
available, Lemma~\ref{high:half-rank} excludes the remaining
candidate.

\casehead{Step 1: setup.}
Proposition~\ref{high:first-bands} establishes $q\le k\le2q$.
Induct on $k\ge2q+1$, simultaneously in the indicated ranks.
Write $r=\alpha k+\eta$, where $\eta\ge-\gamma$.

\casehead{Step 2: balanced targets.}
If $z=\gcd(k,r)>1$, take the balanced split
$h=k/z$, $b=r/z$, and let $a=k-h$, $R=r-b$.  Then
\[
 b-\alpha h\ge-\gamma/z>-\gamma,\qquad
 R-\alpha a\ge-(z-1)\gamma/z>-\gamma.
\]
Both ranks are greater than $\alpha-\gamma=g$, hence at least
$g+1$, and $a\ge\lceil k/2\rceil>q$.  The corresponding pointing
model is admissible, and its smaller-weight normal quotient is
semistable by induction.  Its semistable kernel has the same slope,
so the extension is semistable.

\casehead{Step 3: Farey targets.}
Suppose $\gcd(k,r)=1$.  The two complementary Farey pairs
satisfy $hr-kb=\varepsilon\in\{-1,1\}$, with
$0<h<k$, $0<b<r$.  For $a=k-h$, $R=r-b$,
\begin{align}
 b-\alpha h+\gamma
 &=\frac{h\eta-\varepsilon+\gamma k}{k}
 \ge\frac{\gamma a-\varepsilon}{k}>0,
 \label{high:farey-affine-kernel}\\
 R-\alpha a+\gamma
 &=\frac{a\eta+\varepsilon+\gamma k}{k}
 \ge\frac{\gamma h+\varepsilon}{k}>0.
 \label{high:farey-affine-quotient}
\end{align}
In particular neither side has rank one; indeed both ranks are
at least $g+1$.  Every admissible normal quotient remains
strictly inside the affine region and has smaller weight.
At least one pair is admissible, since one weight is at most
$k/2\le k-q$.

Apply Lemma~\ref{lem:farey}, with its whole-factor exclusion by
Lemma~\ref{lem:determinant}.
It settles even $r$, an admissible even kernel rank, and the
case when both complementary pairs are admissible.  The only
remaining configuration has an admissible pair $(s,b)$ with
\begin{equation}\label{high:last-configuration}
 1\le s<q,\qquad sr-kb=\varepsilon,\qquad r,b\text{ odd},
\end{equation}
and the sole possible strict type
\begin{equation}\label{high:last-type}
 t=\frac{r+\varepsilon b}{2},\qquad e=k+\varepsilon s,
 \qquad re-2kt=1.
\end{equation}
In addition $k-s$ is odd.  The Farey calculation supplies its
allocation, which is needed below: for $\varepsilon=1$ the
kernel intersection is the whole kernel and the normal image
has rank $(r-b)/2$ and degree $k-s$; for $\varepsilon=-1$
the kernel intersection is zero and the normal image has that
same rank and degree. Put $m=\lceil q/s\rceil$.

\casehead{Step 4: the case $k>s(2m-1)$.}
Suppose $k>s(2m-1)$, and set
\[
 H=k-ms,\qquad B=r-mb.
\]
The identities
\begin{gather}
 q\le ms\le2q-2,\qquad Hr-kB=-m\varepsilon,\qquad
 Hb-sB=-\varepsilon,
 \label{high:aux-identities}\\
 s\bigl(B-(m-1)b\bigr)
   =\bigl(k-s(2m-1)\bigr)b+\varepsilon
 \label{high:aux-positive}
\end{gather}
show $B>(m-1)b\ge b\ge g+1$, $H\ge1$,
$H\le k-q$, and $\gcd(H,B)=1$.  The auxiliary pointing model
has kernel $(B,2H)$ and quotient $(mb,2ms)$, and is admissible.
Its quotient lies in the range of Proposition~\ref{high:first-bands},
since
\begin{equation}\label{high:aux-affine}
 mb-\alpha ms+\gamma
 =\frac{ms\eta-m\varepsilon+\gamma k}{k}
 \ge\frac{\gamma H-m\varepsilon}{k}>0.
\end{equation}
For the last sign use $m\le q<\gamma$ and $H\ge1$.

The rounding estimate of Lemma~\ref{g2:prop:rounding} holds here;
for completeness we repeat its proof in the present notation.  For
every feasible kernel-intersection rank $u$ for
\eqref{high:last-type},
\begin{equation}\label{high:rounding}
 \Phi_{H,B}(u)\le
 \left\lfloor\frac{2su-\varepsilon-1}{b}\right\rfloor.
\end{equation}
Since $\gcd(s,b)=1$ and $b\ge3$, the integer $b$ does not divide $s$; as
$|H/B-s/b|=1/(bB)$, the quotients $H/B$ and $s/b$ have the same
integer part $j$.  Write $s=jb+s_0$, $H=jB+H_0$; then $0<s_0<b$,
$0<H_0<B$, and $H_0b-s_0B=-\varepsilon$.
Subtracting $2ju$ reduces the
calculation to these remainders.  By the primitive kernel
profile, the ordinary slope floor is an upper bound, and an
odd degree $a_0$ at a proper positive rank $u$ is disallowed
whenever
\[
 0\le2H_0u-Ba_0\le|B-2H_0|.
\]
Write $2s_0u=bz_0+\rho$, $0\le\rho<b$.

For $\varepsilon=-1$ and $0<u<B$, the ordinary floor can exceed the
right side of \eqref{high:rounding} only if
$\rho=b-1$ and $2u\ge B$.  The extra integer is
$a_0=z_0+1$, which is odd since $b$ is odd, and
\[
 2H_0u-Ba_0=\frac{2u-B}{b}
 \le\frac{|(b-2s_0)B-2|}{b}=|B-2H_0|.
\]
Here $u\le B-1$ and $b-2s_0$ is a nonzero odd integer.
The primitive profile excludes this integer.  The endpoint
$u=0$ satisfies the bound; $u=B$ is infeasible because
$t=(B+(m-1)b)/2<B$.
For $\varepsilon=1$ and $0<u<B$, the only possible overshoot has
$\rho=1$, $2u\le B$, and the odd integer $a_0=z_0$.
Now
\[
 2H_0u-Ba_0=\frac{B-2u}{b}
 \le\frac{|(b-2s_0)B+2|}{b}=|B-2H_0|,
\]
using $u\ge1$, so the same profile excludes it.  The endpoint
$u=0$ is infeasible, since $t=(B+(m+1)b)/2>mb$;
at $u=B$ the bound equals $2H$.  This proves
\eqref{high:rounding} at every feasible rank.

If $v=t-u$ is the rank of the image in the quotient, its semistability
and the identity $2st=be+\varepsilon$ give
\begin{align*}
 \deg S
 &\le\Phi_{H,B}(u)+\left\lfloor\frac{2sv}{b}\right\rfloor\\
 &=e+\Phi_{H,B}(u)
       +\left\lfloor\frac{\varepsilon-2su}{b}\right\rfloor
 \le e-1.
\end{align*}
The last step is the integer identity
$\lfloor(A_0-1)/b\rfloor+\lfloor-A_0/b\rfloor=-1$.
Thus the sole candidate is impossible in this case.

\casehead{Step 5: the case $k\le s(2m-1)$.}
It remains to consider $k\le s(2m-1)$.  Put $a=k-s$.
Since $k\ge2q+1$ and $s<q$,
\[
 q<a\le2(m-1)s<2q.
\]
The original normal quotient has even rank $R=r-b$ and
$R>\alpha a-\gamma$ by \eqref{high:farey-affine-quotient}.
The allocation attached to \eqref{high:last-type} requires in
this quotient a rank-$R/2$ subsheaf of degree $a$.
Lemma~\ref{high:half-rank} excludes it.  This treats both signs
and completes the weight induction.  All normal quotients used
at the induction stage have smaller weight, and the auxiliary
quotients lie in the range of Proposition~\ref{high:first-bands}.
\end{proof}

\section{Finite verification in each genus}
\label{sec:finite-verification}
We now combine the two general theorems. Uniform propagation reduces
the number of degree bands; the large-dimension estimate bounds the
ranks that remain inside those bands. Both ingredients are needed for
finiteness.

\begin{theorem}[Finite verification]\label{thm:finite-verification}
Fix an integer $g\ge3$. Define
\[
 \mathcal T_g=
 \left\{(r,k)\in\mathbb Z^2:
  2g\le k\le6g-4,\quad
  2\le r<(5g-3)k-4g+3\right\}.
\]
Suppose that the normal bundle of a general nonspecial genus-$g$
embedding is semistable for every $(r,k)\in\mathcal T_g$, where
$n=r+1$ and $d=r+1+k-g$. Then it is semistable for every
$n\ge3$ and every $d\ge n+g$.

In particular, it is sufficient to verify the finite rectangle
\begin{equation}\label{eq:finite-rectangle}
 3\le n\le30g^2-42g+15,\qquad
 0\le c=d-n-g\le4g-4.
\end{equation}
The complete row $c=0$ only needs to be checked for $3\le n\le3g$.
For $1\le c\le4g-4$, it is enough to check the smaller triangular set
\begin{equation}\label{eq:finite-triangle}
 3\le n\le10g^2-10g+3+(5g-3)c.
\end{equation}
\end{theorem}

\begin{proof}
Write $q=2g$. For each $q\le k\le3q-4$, a rank $r\ge2$ either
belongs to $\mathcal T_g$ or satisfies
$r\ge(5g-3)k-4g+3$. In the first case semistability is a hypothesis;
in the second it is Theorem~\ref{high:affine}. Thus all ranks in the
initial strip of Theorem~\ref{uni:finite-width} are semistable. That
theorem proves semistability for all $k\ge q$.

The maximum rank in $\mathcal T_g$ occurs at $k=6g-4$ and is
\[
 (5g-3)(6g-4)-4g+2=30g^2-42g+14.
\]
Adding one gives the dimension bound in
\eqref{eq:finite-rectangle}. Substituting $k=2g+c$ gives
\eqref{eq:finite-triangle}. On the complete row,
Corollary~\ref{cor:general-complete} already gives semistability for
$n\ge3g+1$, leaving only the stated small complete embeddings.
\end{proof}

The starting set is explicit and grows with the genus. Before removing
any cases already covered by the complete theorem or other inputs,
its cardinality is
\begin{align*}
 |\mathcal T_g|
 &=\sum_{k=2g}^{6g-4}\bigl((5g-3)k-4g+1\bigr)\\
 &=(4g-3)(20g^2-26g+7).
\end{align*}
There are $4g-3$ degree bands and a quadratic upper bound on the
ambient dimension, so this coarse finite starting problem has cubic
growth. We have not tried to optimize this bound. Integral slopes,
small-dimensional classifications, stronger complete profiles, and
additional pointing models can remove many entries before any new
geometric construction is required.

\begin{remark}[What finite verification asserts]
The theorem is an implication from finitely many semistability
statements to all nonspecial semistability statements in that genus.
Equivalently, any failure somewhere forces a failure of at least one
statement in the finite starting set. It does not say that every
unstable numerical type itself lies in that set. In particular, a
failed base case cannot simply be discarded while retaining the same
induction: one must prove that the residuals used in the propagation
avoid it, or replace the affected models.

This distinction is realized by genus two. Its four complete
exceptions are unstable, and the proof of
Theorem~\ref{g2:thm:main} selects residuals that avoid them.
\end{remark}

For example, at $g=3$ the uniform strip is $6\le k\le14$ and the
rectangle has $n\le159$. Theorem~\ref{thm:finite-verification} supplies an explicit finite
sufficient test for semistability in every nonspecial degree for each
fixed genus $g\ge3$. The genus-two classification shows separately
how the choice of residuals can be adapted when some complete
starting types are unstable.


\end{document}